\documentclass[aap]{imsart-mod}

\RequirePackage{amsthm,amsmath,amsfonts,amssymb}
\RequirePackage[numbers,sort&compress]{natbib}
\RequirePackage[colorlinks,citecolor=blue,urlcolor=blue]{hyperref}
\RequirePackage{graphicx}

\usepackage{hyperref}

\startlocaldefs
\numberwithin{equation}{section}
\theoremstyle{plain}

\newtheorem{theorem}{Theorem}[section]
\newtheorem{lemma}[theorem]{Lemma}
\newtheorem{proposition}[theorem]{Proposition}
\theoremstyle{definition}
\newtheorem{definition}[theorem]{Definition}

\newtheorem{remark}[theorem]{Remark}

\usepackage{booktabs}
\usepackage{nicematrix}
\usepackage{mathtools}
\usepackage{float}

\renewcommand{\prec}{\preceq} 
\newcommand{\precn}{\precneqq}
\newcommand{\bo}{\boldsymbol1}
\newcommand{\lopt}{\lambda_{\mathrm{opt}}}
\newcommand{\Poi}{\mathsf{Poi}}
\newcommand{\PPP}{\mathsf{PPP}}
\newcommand{\MMPPt}[1]{MMPP(#1)}
\newcommand{\CPt}[1]{CP(#1)}

\def\emptyset{\varnothing}

\makeatletter
\def\do@not@check@textbox@alter{}
\makeatother

\usepackage{tikz}
\usetikzlibrary{arrows.meta}

\definecolor{redpath}{HTML}{FF0000} 
\definecolor{greenpath}{HTML}{00AF00} 
\definecolor{bluepath}{HTML}{0030FF} 

\tikzset{
	worldline/.style={line width=1.5pt},
	timeline/.style={line width=0.75pt},
	timeaxis/.style={->,line width=0.75pt},
	event/.style={inner sep=0.6pt,text height=1.75ex,text depth=.25ex,scale=1.5},
	eventstar/.style={inner sep=0pt},
	arrowcirc/.style={inner sep=0.6pt, scale=1.25},
	arrowstar/.style={inner sep=0pt},
	pathred/.style={
		draw=none, 
		preaction={draw=redpath,line width=6.0pt,line cap=butt,opacity=1,line join=butt} 
	},
	pathgreen/.style={
		draw=none,
		preaction={draw=greenpath,line width=6.0pt,line cap=butt,opacity=1,line join=butt} 
	},
	pathblue/.style={
		draw=none,
		preaction={draw=bluepath,line width=6.0pt,line cap=butt,opacity=1,line join=butt}
	}
}

\newcommand{\markx}[2]{%
	\node[event] at (#1,#2) {$\boldsymbol\times$};}
\newcommand{\markot}[2]{%
	\node[event] at (#1,#2) {$\boldsymbol\otimes$};}
\newcommand{\markotstar}[2]{%
	\node[event] at (#1,#2) {$\boldsymbol\otimes$};
	\node[eventstar,scale=1.375] at (#1+0.17,#2+0.23) {$\boldsymbol\ast$};}

\def\arrowcircleyshift{-0.006} 

\newcommand{\rarrowplain}[3]{%
	\draw[horarrow] (#1,#3) -- (#2,#3);}

\tikzset{
	arrowcirc/.style={
		draw,
		circle,
		inner sep=0pt,
		minimum size=1.125ex,   
		line width=0.75pt      
	},
}
\newcommand{\rarrowcirc}[3]{%
	\draw[horarrow] (#1,#3) -- (#2,#3);
	\node[arrowcirc,scale=1.5] at (#2,#3+\arrowcircleyshift) {};
}

\newcommand{\rarrowcircstar}[3]{%
	\draw[horarrow] (#1,#3) -- (#2,#3);
	\node[arrowcirc,scale=1.5] at (#2,#3+\arrowcircleyshift) {};
	\node[arrowstar,scale=1.325] at (#2+0.16,#3+0.18) {$\boldsymbol\ast$};}
	
\newcommand{\larrowcircstar}[3]{%
	\draw[horarrow] (#1,#3) -- (#2,#3);
	\node[arrowcirc,scale=1.5] at (#2,#3+\arrowcircleyshift) {};
	\node[arrowstar,scale=1.325] at (#2-0.16,#3+0.18) {$\boldsymbol\ast$};}

\tikzset{
	horarrow/.style={->,>={Stealth[length=10pt, width=10pt]},line width=1.2pt},
}

\newcommand{\gto}{%
	\mathrel{%
		\tikz[baseline=-0.5ex]{
			\draw[->,>={Stealth[length=6pt,width=7pt]},line width=0.65pt]
			(0,0) -- (1.8em,0);
		}%
	}%
}

\newcommand{\gtoCirc}{%
	\mathrel{%
		\tikz[baseline=-0.5ex]{
			\draw[->,>={Stealth[length=6pt,width=7pt]},line width=0.65pt]
			(0,0) -- (1.8em,0);
			\node[inner sep=0pt,scale=1.125] at (1.8em,-0.009) {$\circ$};
			\node[inner sep=0pt] at (1.8em+0.28em,-0.28em) {\scalebox{0.7}{$j$}};
		}%
	}%
}

\newcommand{\gtoCircStar}{%
	\mathrel{%
		\tikz[baseline=-0.5ex]{
			\draw[->,>={Stealth[length=6pt,width=7pt]},line width=0.65pt]
			(0,0) -- (1.8em,0);
			\node[inner sep=0pt,scale=1.125] at (1.8em,-0.009) {$\circ$};
			\node[inner sep=0pt] at (1.8em+0.28em,0.32em) {\scalebox{0.9}{$\ast$}};
			\node[inner sep=0pt] at (1.8em+0.32em,-0.28em) {\scalebox{0.7}{$j$}};
		}%
	}%
}

\endlocaldefs

\makeatletter
\def\url@fmt#1#2#3#4{%
	\begingroup
	\protected@edef\@tempa{#3}%
	\ifx\@tempa\@empty
	\else
	\protected@edef\@tempb{#4}%
	\ifx\@tempb\@empty
	#1{#2#3}
	\else
	#1{#2\ims@href{#4}{#3}}
	\fi
	\fi
	\endgroup
}

\def\journal@url{}%

\def\journal@name{%
	Author preprint. Published in the Annals of Applied Probability. \ %
}%
\makeatother

\begin{document}


\begin{frontmatter}
	\title{Downward conditional monotonicity gives survival and extinction for contact processes in random environments}
	\runtitle{Downward conditional monotonicity gives CPMRE survival}
	
	\begin{aug}
		\author{\fnms{Joseph P.}~\snm{Stover}\ead[label=e1]{stover@gonzaga.edu}\orcid{0000-0001-7664-2569}} 
		\runauthor{J. P. Stover}
		
		\address{Department of Mathematics, Gonzaga University\printead[presep={,\ }]{e1}}
		
	\end{aug}
	
	\begin{abstract}
		The concept of downward conditional monotonicity for the Markov-modulated Poisson process (MMPP) is introduced and used to derive the optimal stochastic domination of a standard Poisson point process. The maximum arrival rate for the Poisson process which allows this domination to exist is shown to be related to an eigenvalue extracted from the generator matrix of the quasi-birth--death (QBD) formulation of the MMPP. This allows derivation of survival and extinction regimes for a large family of contact processes whose infection and recovery rates vary over time according to an underlying random environment with a finite number of states. Direct comparison with standard contact processes which dominate from above and below accomplishes this. 
	\end{abstract}
	
	
	
\end{frontmatter}


%
\section{Introduction}
\label{sec:intro} 
The contact process is a classic stochastic spatial model of infectious disease, originally introduced by Harris \cite{Harris1974}. This process lives on a graph with edges indicating where the disease can be transmitted from one site to another. A common graph used is the $d$-dimensional nearest-neighbor integer lattice. An infected site infects each neighboring site at rate $\lambda$ and recovers at rate one. Our question of interest is whether the disease survives or goes extinct. This model is well-studied, especially on integer lattices and trees. Numerous variations of the contact process have also been studied over the past several decades. See \cite{LiggettIPS} and \cite{LiggettSIS} for a detailed theoretical background on the contact process and a compendium of research findings. 

We study a class of contact processes where each site has an underlying random environment causing its infection and recovery rates to randomly jump between several values over time. We refer to this model as the contact process in a multitype random environment (CPMRE). The environment has a finite number of discrete states, and independence between the environmental processes at distinct sites is assumed. The infection and recovery rates are governed by coupled Markov-modulated Poisson processes (MMPPs). We derive survival and extinction regimes for a broad class of these models. 

To achieve this, we introduce the concept of \emph{downward conditional monotonicity} (DCM) for MMPPs. Briefly, this involves the random environment being (partially) monotone when conditioned on there being no arrivals. This is used to obtain conditions for when the MMPP stochastically dominates a standard Poisson process. This result follows in a somewhat straightforward way from the matrix-analytic formulation of the conditional intensity function for the MMPP, although it does take some work to prove our sufficient conditions. This permits us to construct stochastic domination relationships between CPMREs and standard contact processes which then allows derivation of survival and extinction regimes for the CPMRE. To establish our sufficient conditions for DCM, we use a technique involving incremental probability mass shifts in the distribution vector of the environmental process between two background states. This technique is used to establish many results in this paper.  

\vspace*{4px}
\enlargethispage{-0.2\baselineskip}

Briefly, the sufficient conditions allowing derivation of our CPMRE results are as follows.
\begin{longlist}
	\setlength{\itemindent}{28pt}
	\setlength{\labelsep}{0.4em}
	\item The infection and recovery rates are monotone functions of the environmental state. 
	\item The environmental process is monotone and irreducible.
	\item Conditioning on there never being any arrivals (partially) preserves the monotonicity of the environmental process (downward conditional monotonicity).
\end{longlist}

Other contact processes having some kind of random environment have also been studied: a randomly generated but static environment in \cite{bramson1991} (the contact process in a random environment, CPRE), and those with dynamic random environments were studied in \cite{broman2007} (the contact process in a randomly evolving environment, CPREE) and \cite{Rem2008} (the contact process in a dynamic random environment, CPDRE). It is important to note that the random environment in our model is dynamic since it is based on coupled MMPPs.

The work herein is largely inspired by that in \cite{broman2007}.
Therein, the infection rate was constant, but the recovery rate alternated between two values.  The underlying environmental process in \cite{broman2007} was not recognized as an MMPP though. 
Our usage of the well-known quasi-birth--death process (QBD) formulation for the MMPP and the conditional intensity function (CIF) allows numerous improvements for studying the random environment. That the optimal coupling parameter is related to an eigenvalue of a block from the QBD generator matrix is a particular benefit. 

Various terms are used for MMPPs in the literature. They were originally introduced in \cite{neuts1979}. The collection of time points where arrivals occur forms a point process, which is a type of Cox process, one directed by a Markov process (see Example 10.3(e) in \cite{daley_pp2}). This is also referred to as a type of doubly-stochastic Poisson process (see Section 6.2 in \cite{daley_pp}). It is also known as a type of Markovian arrival process (MAP). The counting process associated with it is a pure-birth version of what is called a quasi-birth--death process (see Chapter 3 in \cite{neutsMatGeo} on matrix methods for QBD processes). MMPPs have been used extensively in the queueing literature for over four decades. 

First, we give an overview, provide some background, and state the main results. Then, in Section \ref{sec:mmpp-dom}, we prove our MMPP stochastic domination theorem. Section \ref{sec:CPMRE} covers our CPMRE survival and extinction results. Finally, in Section \ref{sec:dcm}, we cover downward conditional monotonicity and the sufficient conditions for our MMPP domination results to hold.

\subsection{Main results}

There are three main results: downward conditional monotonicity for MMPPs, stochastic domination between MMPPs and standard single-rate Poisson processes, and, finally, survival and extinction regimes for CPMRE models. 

\subsubsection{Downward conditional monotonicity and domination for MMPPs}
\label{sec:intro_cm_dom_mmpp} 

Let $B=(B_t)_{t\geq0}$ be an irreducible, monotone, continuous-time Markov jump process on ordered set of states $S=\{0,1,\ldots,k-1\}$ with $k\geq2$. Note that $k=|S|$ always. The generator matrix is denoted $T$ and the stationary distribution by $\pi$. This is the background or environmental process for the MMPP. Let $X=(X_t)_{t\geq0}$, with $X_0=0$, be a counting process which has arrivals at rate $\alpha_j$ when the background process is in state $j$. We refer to its sample paths as \textit{count paths}. Assume an increasing vector of arrival rates $\alpha=(\alpha_j)_{j\in S}$. An initial distribution $x$ for $B_0$ must be specified for the process to be well defined, denoted $B_0\sim x$ (and also $\mathcal L(B_0)=x$). We denote this process by \MMPPt{$\alpha,T$} or by $(B,X)$ when the particular $\alpha$ and $T$ are understood. 
Let $e_j$ be the row vector with unit mass on state $j$, that is, $(e_j)_s = 0$ for $s \neq j$, and
$(e_j)_j = 1$. To avoid the trivial case of $X$ being a Poisson process, we generally assume that $\alpha$ is not a constant vector.  
The diagonal matrix with diagonal given by vector $\alpha$ is denoted $D_\alpha=\text{diag}(\alpha)$. 

A Markov process with generator matrix $T$ is called monotone if $e^{tT}$ is a monotone transition matrix for all $t\geq0$ (see Definition 2.1 in \cite{keilson1977monMC}). Theorem 2.1 in \cite{keilson1977monMC} then provides a way to determine monotonicity from the generator matrix alone. We use a formulation which is easily seen as equivalent to the latter and provides that the process is monotone if and only if  
\begin{equation}\label{eq:Tmon}
	\begin{aligned}
		\sum_{s\leq \ell} T_{is} \geq \sum_{s\leq \ell} T_{js} \ \text{ for all triples } \ \ell<i<j,\\
		\sum_{s> \ell} T_{is} \leq \sum_{s> \ell} T_{js} \ \text{ for all triples } \  i<j\leq\ell. 
	\end{aligned}
\end{equation}
We say $T$ is \textit{strictly monotone} if the main inequalities comparing the summations in \eqref{eq:Tmon} are all strict. The above definition effectively assumes that $k>2$. All $2\times 2$ generator matrices are monotone since $e^{tT}$ is always monotone in that case. This is easily seen by applying the methods of Theorem 2.1 in \cite{keilson1977monMC}. This is contrasted with the discrete-time case, where nonmonotone $2\times2$ transition matrices are simple to construct.

An intuitive understanding of monotonicity is as follows. Given triple of states $\ell<i<j$, state $i$ has a larger sum total transition rate into decreasing set $\{0,1,\ldots,\ell\}$ than state $j$. Likewise, given triple $i<j\leq\ell$, state $j$ has a larger total transition rate into increasing set of states $\{\ell+1,\ldots,k-1\}$. This implies that, for any $t\geq0$ and any $\ell\in S$, 
\begin{equation*}
	\mathbb P(B_t\leq\ell\mid B_0=i)\geq\mathbb P(B_t\leq\ell\mid B_0=j).
\end{equation*} 
Monotonicity allows two copies of the process with distinct and ordered initial conditions, $B_0=i\leq j=B'_0$, to be coupled so that the ordering is preserved between the coupled processes for all time (almost surely): $\mathbb P(B_t\leq B'_t ~~\forall t>0)=1$.

The distribution of the background process at time $t$, denoted $\mathcal L(B_t)$, is always treated as a row vector. For two such vectors, we have the stochastic domination $x\prec y$ if and only if, for each $\ell\in S$, 
\begin{equation}\label{eq:x<y-def}
	\sum_{s\leq\ell} x_s \geq \sum_{s\leq\ell} y_s.
\end{equation}
We say $y$ \textit{strictly dominates} $x$, denoted $x\precn y$, if the inequality comparing the summations in \eqref{eq:x<y-def} is always strict, except for $\ell=k-1$, for which both sums are equal to one. The concepts of strict monotonicity and strict domination are important for our later arguments (see, e.g., Lemma \ref{lem:v-sdom-pi} and Section \ref{sec:CPMRE} for usage of strict domination; Section \ref{sec:dcm} uses strict monotonicity extensively). We sometimes write $x>0$ for a positive vector and $M>0$ for a positive matrix (all components being positive) and $x\geq0$ and $M\geq0$ when they are nonnegative. They are called nonzero when at least one component is nonzero.

Let $Y^\lambda=(Y^\lambda_t)_{t\geq0}$ be the Poisson (counting) process with arrival rate $\lambda$ and distribution $\Poi(\lambda)$, which we denote by $Y^\lambda\sim\Poi(\lambda)$. Such a process satisfies $Y^\lambda_0=0$ and has jumps of size one after exponential wait times, with cadlag sample paths.  We write $N(Y^\lambda)$ for the underlying Poisson point process with law $\PPP(\lambda)$ (i.e.,  $N(Y^\lambda)\sim\PPP(\lambda)$) and $N(X)$ for the point process corresponding to our MMPP counting process $X$. Each point in the point process corresponds to an arrival time for the counting process. 


We write $Y^\lambda \prec X$ to indicate domination between their underlying point processes. By the celebrated Strassen's Theorem (see Theorem 17.58 in \cite{klenke}), $Y^\lambda \prec X$ if and only if we can couple $X$ and $Y^\lambda$ so that $X$ has an arrival wherever $Y^\lambda$ does (almost surely). This can also be thought of as stochastic domination involving random counting measures on $[0,\infty)$, and this view provides a rich theoretical basis for point processes. See \cite{rolski1991} for background on stochastic domination for point processes, and see, for example, Chapter 17 (particularly Section 17.7) and Chapter 24 in \cite{klenke} for background on stochastic domination and random counting measures, respectively.  Define the \textit{optimal coupling parameter}, by 
\begin{equation}\label{eq:lam-opt-def}
	\lopt\coloneq\sup\{\lambda\geq0 \mid Y^\lambda\prec X\}
\end{equation}
where the $Y^0\sim\Poi(\lambda=0)$ is considered to have unit mass on the event of no arrivals. That the domination obtains precisely at $\lambda=\lopt$ follows via weak convergence: $Y^{\lambda-\epsilon}\prec X$ for all sufficiently small $\epsilon>0$ implies $Y^{\lambda}\prec X$ (an argument similar to that in the proof of Lemma 4.3 in \cite{broman2007} shows this). Theorem \ref{thm:main_mpp_dom} (stated below) also establishes this under our specific model assumptions. 

We write $\mathcal L(X)$ for the law of random process $X$ and $\mathcal L(X\mid A)$ for the law of $X$ conditioned on event $A$. For stochastic dominations, such as between the point processes underlying $Y^\lambda$ and $X$, we write $Y^\lambda\prec X$ with the understanding that stochastic domination is technically a relation between the appropriate probability measures:  $\mathcal L(N(Y^\lambda))\prec\mathcal L(N(X))$. Similarly, when we write stochastic domination between vectors ($x\prec y$), we technically mean a relationship between the probability measures the vectors represent. 

\begin{definition}\label{def:v}
	For an MMPP (with some assumed initial distribution $B_0\sim x$), the \emph{equilibrium no arrival distribution}, when it exists, is defined by
	\begin{equation*}
		v^*\coloneq\lim_{t\to\infty} \mathcal{L}(B_t \mid X_s=0 ~\forall  s\in[0,t]).
	\end{equation*}
\end{definition}
Definition \ref{def:v} gives the limiting conditional distribution of the environmental process when conditioned on there never being any arrivals. We address its existence and that it is indeed (conditionally) stationary later (see Lemma \ref{lem:lamstar-v-pos}). Note that $v^*$ also denotes a particularly defined row vector and that it plays a critical role throughout this paper.

\begin{definition}\label{def:DCM}
	Consider MMPPs $(B,X),(B',X')$ with possibly different initial distributions for $B$ and $B'$ but with the same arrival rate vector $\alpha$ and generator matrix $T$. The MMPP with $\alpha,T$ is called \emph{downward conditionally monotone} if $v^*\prec \mathcal{L}(B_0)\prec\mathcal{L}(B'_0)$ implies, for all $t>0$,
	\begin{equation*}
		v^*\prec\mathcal L(B_t\mid X_s=0 ~\forall s\leq t)\prec\mathcal L(B'_t\mid X'_s=0~\forall s\leq t).
	\end{equation*}
\end{definition}

Definition \ref{def:DCM} describes a stochastic ordering when conditioning on there being no arrivals. 
If $T$ is monotone, then standard monotonicity is that $\mathcal{L}(B_0)\preceq\mathcal{L}(B'_0)$ implies $\mathcal{L}(B_t)\preceq\mathcal{L}(B'_t)$ for all $t>0$, but our conditional domination generally does not hold.

The term ``downward conditional monotonicity'' has appeal since the behavior in Definition \ref{def:DCM} resembles monotonicity, but of course is conditional on there being no arrivals. It is ``downward'' in two senses, one of which is that $\mathcal L(B_t\mid X_s=0 ~\forall s\leq t)$ converges to $v^*$ (Definition \ref{def:v}) but is always stochastically larger than it (Definition \ref{def:DCM}). The other sense is that, for the space of count paths and its natural ordering, $\{\text{no arrivals in }[0,t]\}$ is a decreasing event: if $\omega_1,\omega_2$ are count paths satisfying $\omega_1(t)\leq\omega_2(t)$ for all $t$ with $\omega_2$ having no arrivals in $[0,t]$, then $\omega_1$ also has no arrivals there. We use that particular decreasing event many times. 

\begin{theorem}[DCM]
	\label{thm:main_mmpp_cm}
	Consider an MMPP with monotone, irreducible generator matrix $T$ and increasing vector of arrival rates $\alpha$. If $k\geq3$, then the model is DCM if and only if, for all triples $\ell<i<j$, 
	\begin{equation}\label{eq:dcm}
		(\alpha_j-\alpha_i) \sum_{s\leq\ell}v^*_s  \leq \sum_{s\leq\ell}(T_{is}-T_{js}).
	\end{equation}
	All MMPPs with $k=2$ are DCM.
\end{theorem}

It is not immediately obvious why \eqref{eq:dcm} gives the desired DCM property. Proving this theorem is the most computationally intensive part of our work, and is saved for last (see Section \ref{sec:dcm}). The technique of incrementally shifting mass between two environmental states is the key to its proof, along with a use of weak convergence (see \cite{billingsley} for background theory on weak convergence, especially Section 12 in Chapter 3). 

\addtolength{\abovedisplayskip}{-2pt}
\addtolength{\belowdisplayskip}{-2pt}

Given the MMPP with $\alpha,T$, we define the \textit{permutation-reversed MMPP} with $\widehat{\alpha},\widehat{T}$ to have the arrival rate vector in reverse order, $\widehat{\alpha}_j=\alpha_{k-1-j}$ for all $j\in S$, and the generator similarly reversed, $\widehat{T}_{ij}=T_{k-1-i,k-1-j}$ for all $i,j\in S$. The latter can be understood as flipping $T$ horizontally and then vertically (or applying the transpose and anti-transpose). Note that if $\alpha$ is decreasing, then $\widehat{\alpha}$ is increasing, and that $T$ is monotone if and only if $\widehat{T}$ is monotone. This permutation-reversal can also be thought of as reversing the ordering on the state space $S$. 

\begin{remark}\label{rmk:decr-alpha}
	Theorem \ref{thm:main_mmpp_cm} is easily adapted to when $\alpha$ is decreasing by applying it to the permutation-reversed model with $\widehat{\alpha},\widehat{T}$.  Alternatively, one can simply reverse all stochastic orderings in Definition \ref{def:DCM}. Then, swap $i$ and $j$ in \eqref{eq:dcm} and sum over $s>\ell$ instead, for all triples $i<j\leq\ell$ (but preserve the main inequality), but this results in conditional monotonicity with stochastic domination from above by $v^*$, which we sometimes call the \textit{permutation-reversal of DCM}. 
\end{remark}

\begin{definition}\label{def:alpha*}
	Given rate vector $\alpha$ and irreducible generator matrix $T$, we denote the minimum modulus of all eigenvalues of $D_\alpha-T$ by 
	\begin{equation*}
		\alpha^*\coloneq\min\{|\lambda| ~:~ \lambda \text{ is an eigenvalue of } D_\alpha-T\}.
	\end{equation*}
\end{definition}
Nonnegative matrix theory easily shows that $\alpha^*$ is in fact the unique minimal-modulus eigenvalue of $D_\alpha-T$ since $T$ is irreducible (see Lemmas \ref{lem:Da-T_idd} and \ref{lem:lamstar-v-pos}). 
We take this for granted and refer to $\alpha^*$ as \textit{the minimal-modulus eigenvalue of $D_\alpha-T$}. 

For our MMPP, we generally use $\alpha$ as the arrival rate vector, but Definition \ref{def:alpha*} is taken to specify a notation convention where rate vector $u$ is associated with (eigenvalue) $u^*$ (of course the generator matrix $T$ must be understood as well). Later, we show that $\alpha^*>0$ and that $v^*$ is its associated strictly positive eigenvector (again, see Lemmas \ref{lem:Da-T_idd} and \ref{lem:lamstar-v-pos}). 

Now, we state our MMPP domination result. The following theorem can be thought of as a generalization of Theorem 1.4 in \cite{broman2007}. 

\begin{theorem}[MMPP domination]
	\label{thm:main_mpp_dom}
	Let $Y^\lambda\sim\Poi(\lambda)$, and let the MMPP $(B,X)$ have irreducible, monotone generator matrix $T$ and increasing arrival rate vector $\alpha$. If the MMPP is DCM, and $B_0\sim x$ with $v^*\prec x$, then 
	\begin{longlist}
		\setlength{\itemindent}{28pt}
		\setlength{\labelsep}{0.4em}
		\item $Y^\lambda\prec X$ if and only if 
		$\lambda\leq \lopt=\alpha^*$, the minimal-modulus eigenvalue of $D_\alpha-T$, 
		\item $X\prec Y^\lambda$ if and only if $\lambda\geq\max_j\alpha_j$. 
	\end{longlist}
\end{theorem}


\begin{remark}\label{rmk:alpha-decr-x<v}
	When $\alpha$ is increasing, $\max_j\alpha_j=\alpha_{k-1}$, of course. For decreasing $\alpha$, Theorem \ref{thm:main_mpp_dom}(i) still applies exactly as stated but requires $x\prec v^*$ instead of $v^*\prec x$, and assessing DCM requires use of the permutation-reversed MMPP (or the permutation-reversal of DCM---see Remark \ref{rmk:decr-alpha} above). Theorem \ref{thm:main_mpp_dom}(ii), in fact, holds exactly as stated for arbitrary $\alpha$ and irreducible $T$; a detailed proof is given in Remark \ref{rmk:alpha-max-gen-T}.
\end{remark}


Briefly, Theorem \ref{thm:main_mpp_dom} is established by directly computing the conditional intensity function for the point process associated with $X$ and showing that dominations with Poisson point processes occur only under the stated conditions on $\lambda$. 

The matrix $D_\alpha-T$ is identifiable in the literature on queueing and quasi-birth--death processes (e.g., see \cite{neutsMatGeo}). The MMPP is a pure-birth QBD process with state space $S\times \mathbb N_0$ and generator matrix having block form $Q=(Q_{ij})_{i,j\in\mathbb N_0}$  given by 
\begin{equation}\label{eq:Q}
	Q=\left(\begin{matrix}
		\ T-D_\alpha \ & D_\alpha & & &\\
		& \ T-D_\alpha \ & \ \ \ D_\alpha & &\\[-6px]
		& & \hspace{8px}\rotatebox{19}{$\ddots$} \quad & \hspace{18px}\rotatebox{19}{$\ddots$} \quad & \\
	\end{matrix}\right).
\end{equation}

\addtolength{\abovedisplayskip}{2pt}
\addtolength{\belowdisplayskip}{2pt}

\pagebreak

\noindent This is a double-diagonal, infinite-dimensional block matrix with each block having size $k\times k$. The diagonal blocks satisfy $Q_{ii}=T-D_\alpha$, and the off-diagonal $Q_{i,i+1}=D_\alpha$, for all $i\in\mathbb N_0=\{0,1,\ldots\}$ (index $i$ is the current size of the queue). All other blocks are $k\times k$ zero matrices. The $T-D_\alpha$ blocks indicate transition rates for the environmental process on $S$ without incrementing the queue, whereas each $D_\alpha$ block represents an arrival in the queue with no change in environmental state.

The proof of Theorem \ref{thm:main_mpp_dom} is given in Section \ref{sec:mmpp_dom_proof} (and assumes Theorem \ref{thm:main_mmpp_cm}). The proof of Theorem \ref{thm:main_mmpp_cm} is given in Section \ref{sec:proof-of_main_dcm-thm}.

\subsubsection{CPMRE survival and extinction}
\label{sec:intro_CPMRE} 

We always work with the $d$-dimensional integer lattice $\mathbb Z^d$ with nearest neighbors (sites $x,y\in\mathbb Z^d$ are connected by an edge only if they differ by $\pm1$ in a single coordinate). Each site is either in state zero (referred to as uninfected, healthy, or recovered) or state one (infected) and flips its value at exponential rates. Note that, when working with the CPMRE, $x,y$ refer to sites in the lattice instead of distributions.

For the standard contact process, a healthy site becomes infected (flips from state zero to state one) at rate $\lambda$ multiplied by the number of infected neighboring sites (this, of course, varies over time). An infected site recovers (flips from one to zero) at constant rate one. 

Our primary interest is in whether the infection survives or goes extinct, when starting with finitely many initially infected sites.
The infection is said to survive if, with positive probability, there are infected sites for all time. It goes extinct if, with probability one, after some finite time there are no infected sites (an absorbing state for the process). \textit{Strong survival} occurs if the process, when started from a single infected site, has a positive probability of that site being infected infinitely often (which means recovering and being reinfected infinitely many times). 
\textit{Weak survival} occurs if there is a positive probability of always having infected sites, but the process is eventually pushed out of any finite region (see p. 42 in \cite{LiggettSIS} for more details about strong and weak survival). We only consider the nearest-neighbor integer lattice, and, on this graph, survival for the contact process implies strong survival (when starting from finitely many initially infected sites). Our results partially extend to arbitrary graphs in a natural way (see Remark \ref{rmk:cpmre-arb-graph}). 

For the contact process on the nearest-neighbor $d$-dimensional integer lattice, there is only a single critical value $\lambda_c$ (which depends on the dimension) such that the contact process survives strongly if $\lambda>\lambda_c$ and dies out if $\lambda\leq \lambda_c$ (see Part I, Section 2 in \cite{LiggettSIS}, and also \cite{BezuidenhoutGrimmett1990}). 

We study a variation of the contact process where both the infection and recovery rates depend on the state of an underlying random environment with $k\geq2$ states. The CPMRE is a process denoted by $(B_t^x,\eta_t^x)_{x\in\mathbb Z^d,t\geq0}$. Each site $x\in\mathbb Z^d$ has an environmental process $(B_t^x)_{t\geq0}$ with generator matrix $T$ and states $S=\{0,1,\ldots,k-1\}$ (the same $S$ used for our MMPP above). For given $x\neq y$, $(B_t^x)_{t\geq0}$ and $(B_t^y)_{t\geq0}$ are independent random environments. The health status at time $t$ for site $x$ is denoted $\eta_t^x$. It is infected if $\eta_t^x=1$ and uninfected if $\eta_t^x=0$. A realization of the CPMRE has two cadlag sample paths associated with each site, one for the environment and the other for its health status.  Let $\mathcal N(x)=\{y : \|x-y\|_1=1\}$ denote the set of nearest neighbors for site $x$. 

If site $x$ is healthy and in environmental state $j$, then it becomes infected at rate $\beta_j$ multiplied by the number of infected neighbors, irrespective of the environmental states of these neighboring sites. If infected, it recovers at rate $\mu_j$. In this way, the random environment can be thought of as governing the site's susceptibility to infection and its ability to return to health. 
Let $\beta=(\beta_j)_{j\in S}$ be the vector of infection rates and $\mu=(\mu_j)_{j\in S}$ the vector of recovery rates. The infection and recovery rate processes can be thought of as coupled MMPPs with a common background process. The rate vectors $\beta,\mu$ 
are analogous to $\alpha$, the MMPP arrival rate vector discussed in Section \ref{sec:intro_cm_dom_mmpp}. 

Following the convention in Definition \ref{def:alpha*}, let $\beta^*$ and $\mu^*$ be the minimal-modulus eigenvalues of $D_{\beta}-T$ and $D_{\mu}-T$, respectively. These eigenvalues are shown to be positive later since we assume $\beta,\mu$ are nonzero, as mentioned above already, see Lemmas \ref{lem:Da-T_idd} and \ref{lem:lamstar-v-pos}. Also, they will be shown to be associated with eigenvectors $v^\beta,v^\mu$, respectively, which are in fact the analogous equilibrium no arrival distributions by the just mentioned lemmas. 

Denote by $(j,\ell)$ the state of the process at site $x$ when the (background) environment is in state $j\in S$ and the (foreground) contact process is in state $\ell\in\{0,1\}$. The transition rates for site $x$ at time $t$ are summarized as follows:
\begin{equation*}
	\begin{aligned}
		(i,0)\to(i,1) &\text{ at rate } \beta_i n_1(\eta_t,x),\\
		(i,1)\to(i,0) &\text{ at rate } \mu_i,\\
		(i,\ell)\to(j,\ell) &\text{ at rate } T_{ij} \text{ (for $i\neq j$)},\\
	\end{aligned}
\end{equation*}
where $n_1(\eta_t,x)=\sum_{y\in\mathcal N(x)} \eta_t^y$ is the number of infected neighbors and $\eta_t=(\eta_t^x)_{x\in\mathbb Z^d}$ is the state of the infection process at time $t$. 

We refer to this process as the \textit{child-CPMRE}. There is an analogous \textit{parent-CPMRE} model where the environmental state at ``parent'' site $y$ determines its ability to infect neighboring ``child'' site $x$. In the child-CPMRE, the random environment determines the site's susceptibility to infection, whereas, in the parent-CPMRE, it has no effect on susceptibility but instead determines the site's infectiousness. Our results apply to both models as they use the same graphical representation (except with reversed arrows). They are not identical processes though. Whenever we write ``CPMRE,'' we mean the child-CPMRE unless explicitly stated otherwise. 

The next theorem is a direct application of Theorem \ref{thm:main_mpp_dom} and relies on a suitable graphical construction method for the CPMRE along with two modified CPMRE models (see Section \ref{sec:CPMRE}). A decreasing arrival rate vector and the permutation-reversed MMPP described above Remark \ref{rmk:decr-alpha} are also used. Theorem \ref{thm:main_CPMRE}, about to be stated, can be thought of as a generalization of some results stated in Section 1.1.2 of \cite{broman2007}, but there are significant differences.  
\begin{theorem}[CPMRE survival and extinction]
	\label{thm:main_CPMRE}
	Consider the CPMRE described above with increasing infection rate vector $\beta=(\beta_0,\beta_1,\ldots,\beta_{k-1})$, decreasing recovery rate vector $\mu=(\mu_0,\mu_1,\dots,\mu_{k-1})$, 
	and irreducible, monotone environmental process with generator matrix $T$. Let the initial environmental distribution at each site be arbitrary.	
	\begin{longlist}
		\setlength{\itemindent}{28pt}
		\setlength{\labelsep}{0.4em}
		\item If the MMPP with $\beta,T$ is DCM and 
		$\frac{\beta^*}{\max_j\mu_j}>\lambda_c$, 
		then the CPMRE survives (strongly).
		\item If the permutation-reversed MMPP with $\widehat\mu,\widehat{T}$ is DCM and 
		$\frac{\max_j\beta_j}{\mu^*}\leq\lambda_c$, 
		then the CPMRE goes extinct.
	\end{longlist}
\end{theorem}

Even though we assume increasing $\beta$ and decreasing $\mu$, Theorem \ref{thm:main_CPMRE} is easily adapted as long as $\beta,\mu$ are both monotone. Remarks \ref{rmk:decr-alpha} and \ref{rmk:alpha-decr-x<v} discuss how our DCM and MMPP domination results can be applied when the rate vector is decreasing. When $\beta,\mu$ are not monotone, then permutations may still allow Theorem \ref{thm:main_CPMRE} to be applicable (see Remark \ref{rmk:cmpre-perm}).

Theorem \ref{thm:main_CPMRE} is proved in Section \ref{sec:cpmre-proof} and gives sufficient survival and extinction criteria for a broad class of models with no restrictions on the initial environmental distributions at any site. This result gives no knowledge on model behavior for $\frac{\beta^*}{\max_j\mu_j}\leq\lambda_c<\frac{\max_j\beta_j}{\mu^*}$. Two example CPMRE models are studied in Sections \ref{sec:example-cpmre-2} (with $k=2$) and \ref{sec:example-cpmre-3} (with $k=3$).

\section{Stochastic domination for MMPPs}
\label{sec:mmpp-dom}

We first provide a few results for general MMPPs without any particular parameter restrictions---we are always explicit when parameter restrictions (e.g., increasing $\alpha$, monotone $T$) are required. 
The QBD approach for the MMPP views this process as a continuous-time Markov process $(B_t,X_t)_{t\geq0}$ with state space $S\times \mathbb N_0$ and generator matrix $Q$ given by \eqref{eq:Q}. The background environmental process $B=(B_t)_{t\geq0}$ has state space $S=\{0,\ldots,k-1\}$ (with $k=|S|$ always fixed) and evolves according to generator $T$, and $X=(X_t)_{t\geq0}$ is the counting process of arrivals, taking values in $\mathbb N_0=\{0,1,\ldots\}$ and having arrivals at rate $\alpha_j$ when the background process is in state $j$.


Let random variable $\tau_j$ be the $j$th arrival time with the convention $\tau_0=0$. When the initial distribution is understood as $B_0\sim x$, we frequently use the shorthand notation $x_t$ defined by
\begin{equation}\label{eq:xt-def}
	x_t\coloneq\mathcal L(B_t\mid\tau_1>t)=\mathcal L(B_t\mid X_s=0 ~\forall s\in[0,t]).
\end{equation} 
This defines a notation convention, for example, $B_0\sim y$ gives $y_t$, and $B_0\sim\pi$ gives $\pi_t$. 

The sum of a (row) vector is denoted $x\bo=\sum_{j\in S}x_j$ which can be thought of as the standard matrix product of row vector $x$ and column vector $\bo$, which is an appropriately-sized column of ones. Column vectors are always denoted in bold font. Other than $\bo$, which appears frequently, we only require column vectors on a small number of occasions. Standard matrix notation is generally that $x$ is a column vector and its transpose $x^\top$ is a row vector. We wish to avoid repetitive use of the transpose symbol and the potential for mixing column and row vectors in stochastic dominations such as $x\prec y^\top$. So, for us, $x$ is a row vector and $\mathbf x$ is a column vector. 

We generally imagine the distribution of $B_t$ evolving over time as arrivals are observed for $X$. We know the precise time of each arrival in our observation history but do not know the specific underlying environmental state. 

Let $0<t_1<t_2<\cdots<t_n$ be a list of times for some $n$. For $t\geq t_n$, define 
\begin{equation*}
	\mathcal{H}_t=\{\tau_1=t_1,\ldots,\tau_n=t_n\}
\end{equation*} 
to be the history of arrivals for our MMPP up until time $t$. That is, exactly $n$ arrivals have occurred, and their precise arrival times are specified in the notation. This notation implicitly assumes $\tau_{n+1}>t$, meaning the $(n+1)$st arrival has not yet occurred by time $t$. We write $\mathcal{H}_t=\{\tau_1>t\}=\{X_s=0~\forall s\in[0,t]\}$ when no arrivals have occurred by time $t$. The notation $\mathcal{H}_t$ does not make any assumption on the number of arrivals as $n$ may be zero or any positive integer. Intuitively, $\mathcal H_t$ can be thought of as a random point process which unfolds over time.

We now provide some formulas for the distribution of $B_t$ when conditioned on the observed history of arrivals. This next lemma contains results which can be found in Example 10.3(e) of \cite{daley_pp2}, so no proof is given here.

\addtolength{\abovedisplayskip}{-3pt}
\addtolength{\belowdisplayskip}{-3pt}

\begin{lemma} \label{lem:new-L(Bt|H)} 
	For the MMPP with initial distribution $B_0\sim x$ on $S$, the conditional distribution of $B_t$ before the first arrival is given by
	\begin{equation}\label{eq:LBt|0}
		\mathcal L(B_t\mid\tau_1>t)=x_t=\frac{xe^{t(T-D_\alpha)}}{xe^{t(T-D_\alpha)}\bo}.
	\end{equation}
	Conditioned on the first arrival time being precisely $\tau_1=t_1$, the conditional distribution of $B_{t_1}$ is given by 
	\begin{equation}\label{eq:LBt|tau1}
		\mathcal L(B_{t_1}\mid\tau_1=t_1)=\frac{x_{t_1}D_\alpha}{x_{t_1}D_\alpha\bo}=\frac{xe^{{t_1}(T-D_\alpha)}D_\alpha}{xe^{{t_1}(T-D_\alpha)}D_\alpha\bo}.
	\end{equation}
	Consider arrival history $\mathcal{H}_{t}=\{\tau_1=t_1,\ldots,\tau_n=t_n\}$, which implicitly assumes that $t\geq t_n$ and that $\tau_{n+1}>t$. Then, 
	\begin{equation}\label{eq:Bt|Htn}
		\mathcal L(B_t\mid \mathcal{H}_{t})=
		\frac{x\prod_{j=1}^{n}\left[e^{(t_j-t_{j-1})(T-D_\alpha)}D_\alpha\right] e^{(t-t_{n})(T-D_\alpha)}}{x\prod_{j=1}^n\left[e^{(t_j-t_{j-1})(T-D_\alpha)}D_\alpha\right] e^{(t-t_{n})(T-D_\alpha)}\bo}.
	\end{equation}
\end{lemma}

\addtolength{\abovedisplayskip}{3pt}
\addtolength{\belowdisplayskip}{3pt}

We provide some intuition on these formulas. Consider the QBD generator $Q$ for the MMPP given by \eqref{eq:Q}. Let $\tilde{x}$ be an initial distribution on $S\times \mathbb N_0$. We treat $\tilde{x}$ as an infinite row vector whose first $k$ components correspond to $S\times \{0\}$, and the next group of $k$ components to $S\times \{1\}$, etc. The distribution of the process at time $t$ is then given by $\tilde{x}e^{tQ}$. Assume that $\tilde{x}$ places full mass on $S\times \mathbb \{0\}$, which means that the process initiates with a count of zero arrivals. Let $x$ be the distribution on $S$ corresponding to truncating $\tilde{x}$ to the first $k$ components. We illustrate the structure of $\tilde{x}$ and its relationship to $x$ and the state space $S\times \mathbb N_0$: 
\looseness=-1
%
\begin{equation*}
	\tilde{x}=(\hspace*{1px}\overbrace{\underbrace{\tilde{x}_0,\tilde{x}_1,\ldots,\tilde{x}_{k-1}}_{x}}^{S\times\{0\}},\overbrace{\tilde{x}_k,\tilde{x}_{k+1},\ldots,\tilde{x}_{2k-1}}^{S\times\{1\}},\ldots).
\end{equation*}
\noindent We drop the ``tilde'' for $x$ though, $x=(x_0,\ldots,x_{k-1})$, and treat it as a distribution on $S$. 

\enlargethispage{1.0\baselineskip}

The diagonal blocks of the matrix exponential $e^{tQ}$ are given by $e^{t(T-D_\alpha)}$ since the diagonal blocks of $Q^n$ are all equal to $(T-D_\alpha)^n$. The probability that the first arrival has not yet occurred by time $t$ and $B_t=j$ is thus given by the $j$th component of $xe^{t(T-D_\alpha)}$, 
\begin{equation}\label{eq:x'Q-to-xTD}
	\mathbb P(B_t=j,\tau_1>t)=\big(\tilde{x} e^{tQ}\big)_j=\big(xe^{t(T-D_\alpha)}\big)_j, 
\end{equation}
which uses the upper left block of $e^{tQ}$ only.
The probability that the first arrival has not yet occurred is the sum of \eqref{eq:x'Q-to-xTD} over all $j$ giving 
$\mathbb P(\tau_1>t)=xe^{t(T-D_\alpha)}\bo.$ 
Thus, our formula for $\mathcal L(B_t\mid\tau_1>t)$ is easily understood.  Conditioning precisely at the time of a jump \eqref{eq:LBt|tau1} simply involves a reweighting of the states relative to their arrival rates in $\alpha$. The formula \eqref{eq:Bt|Htn} is simply iterating the operations in \eqref{eq:LBt|0} (with appropriate interarrival times) and \eqref{eq:LBt|tau1} repeatedly for an arbitrary history $\mathcal{H}_t$.

Note that \eqref{eq:LBt|0} gives our notation convention $x_t$ (defined in \eqref{eq:xt-def}) in matrix form. We use this convention extensively, especially in Section \ref{sec:dcm}.

The conditional intensity function (CIF) is the expected arrival rate according to the distribution of $B_t$ conditioned on the history of arrivals. At time $t$ with arrival history $\mathcal{H}_{t}$, the CIF is denoted $\lambda(t\mid \mathcal{H}_{t})$ and is defined as 
\begin{equation}\label{eq:cif}
	\lambda(t\mid \mathcal{H}_{t})=
	\mathcal L(B_t\mid \mathcal{H}_{t}) D_\alpha\bo.
\end{equation}

Note that our CIF is right-continuous since $\mathcal L(B_t\mid \mathcal{H}_t)$ is. Precisely at an arrival, $\tau_1=t_1$, the distribution of $B_{t_1}$ is conditioned on the fact of that arrival at $t_1$, and this generally causes a jump making $\mathcal L(B_t\mid \mathcal{H}_t)$ continuous from the right (thus causing a jump for $	\lambda(t\mid \mathcal{H}_{t})$). Technically, a CIF is generally only defined up to a measure-zero set of points, much like a probability density function, but one can choose to take the left- or right-continuous version. A typical convention is to make the CIF left-continuous for reasons connected to predictability---see the discussion which follows Definition 7.2.II in \cite{daley_pp}. Also, Proposition 14.1.VI in \cite{daley_pp2}, and the surrounding text, is relevant for the technical issues about these matters---this regards the \textit{compensator} of a point process which is the integral of the CIF (changing the CIF on a set of measure zero does not change its integral).

The CIF can be thought of as a stochastic process which unfolds in time as arrivals are observed. For MMPP models with monotone, irreducible $T$, increasing $\alpha$, and which are DCM, the CIF typically (deterministically) decays in between arrivals and jumps up at (random) arrival times. 

Next, we use Lemma \ref{lem:new-L(Bt|H)} to translate DCM, Definition \ref{def:DCM}, into our matrix formulation. Also, we note that arrivals always create a stochastically larger distribution. For now, assume the existence of equilibrium no arrival distribution $v^*$ as described by Definition \ref{def:v} (see Lemma \ref{lem:lamstar-v-pos} below for proof of its existence).


\begin{lemma}\label{lem:dcm-matrix-form}
	An MMPP is DCM if and only if $v^*\preceq x\preceq y$ implies that, for all $t>0$, 
	\begin{equation}\label{eq:dcm-matrix-def}
		v^*\preceq \frac{xe^{t(T-D_\alpha)}}{xe^{t(T-D_\alpha)}\bo}\preceq \frac{ye^{t(T-D_\alpha)}}{ye^{t(T-D_\alpha)}\bo}.
	\end{equation}
\end{lemma}
\begin{proof}
	This is a direct application of \eqref{eq:xt-def} and \eqref{eq:LBt|0} to Definition \ref{def:DCM}. 
\end{proof}

\pagebreak

Using our shorthand notation \eqref{eq:xt-def}, we write \eqref{eq:dcm-matrix-def} as $v^*\prec x_t\prec y_t$, and we use this expression frequently, especially in Section \ref{sec:dcm}. 

\begin{lemma}\label{lem:rescale-Da}
	If $\alpha$ is increasing and $x$ satisfies $xD_\alpha\bo\neq0$, then
	$x\preceq \frac{x D_\alpha}{x D_\alpha\bo}$.
\end{lemma}
\begin{proof}
	Let $x$ be a probability distribution on $S$ with random variable $Z\sim x$, and assume $\alpha=(\alpha_0,\ldots,\alpha_{k-1})$ is increasing (and nonzero to avoid trivialities). 
	We temporarily assume that $x$ is strictly positive, but we relax this later. 
	Then, it is clear that, for any $\ell$, 
	\begin{equation}\label{eq:EaW>}
		\mathbb E[\alpha_Z]\geq\mathbb E[\alpha_Z\mid Z\leq\ell].
	\end{equation}
	Writing \eqref{eq:EaW>} in terms of summations yields 
	\begin{equation*}
		xD_\alpha\bo=\sum_{s=0}^{k-1} x_s\alpha_s\geq \frac{\sum_{s\leq\ell} x_s\alpha_s}{\sum_{s\leq\ell}x_s}.
	\end{equation*}
	So long as $xD_\alpha\bo\neq0$, we can rearrange this to become
	\begin{equation*}
		\sum_{s\leq\ell}x_s\geq \frac{\sum_{s\leq\ell} x_s\alpha_s}{xD_\alpha\bo}=\sum_{s\leq\ell} \left( \frac{x D_\alpha}{x D_\alpha\bo}\right)_s
	\end{equation*}
	which is exactly \eqref{eq:x<y-def} for the stochastic domination $x\prec\frac{x D_\alpha}{x D_\alpha\bo}$. It is not hard to see that this being true for any strictly positive $x$ implies that it is true for any $x$ with $xD_\alpha\bo\neq0$ by a simple limiting argument.
\end{proof}
Note that, if $\alpha$ is decreasing instead of increasing, then it is an easy exercise to show that the stochastic domination in Lemma \ref{lem:rescale-Da} is reversed: $\frac{xD_\alpha}{x D_\alpha\bo}\prec x$.

We required $x$ and $\alpha$ to not be orthogonal in Lemma \ref{lem:rescale-Da}. In practice, this is always the case for us since $\alpha$ is generally assumed to be increasing and not a constant vector and that $x$ will normally be stochastically larger than a strictly positive vector. Note that $x_t$ is a strictly positive vector for any $x$ when $T$ is irreducible making $x_tD_\alpha\bo>0$. This ``multiplication by $D_\alpha$ with normalization'' occurs for us precisely at arrival times (see \eqref{eq:LBt|tau1}).

\begin{lemma}\label{lem:v<L(B_t|H)}
	Assume that $T$ is irreducible, monotone, that $\alpha$ is increasing, and that the MMPP is DCM. Given any $t>0$ and any history of arrivals $\mathcal{H}_{t}$, $v^*\prec\mathcal L(B_0)$ implies that 
	\begin{equation*}
		v^*\prec\mathcal L(B_t\mid\mathcal{H}_{t} ).
	\end{equation*}
\end{lemma}
\begin{proof}
	Let $B_0\sim x$. Lemma \ref{lem:dcm-matrix-form} shows that $v^*\prec x$ implies $v^*\prec x_t=\mathcal L(B_t\mid \tau_1>t)$ since DCM is assumed. Lemma \ref{lem:rescale-Da} and \eqref{eq:LBt|tau1} shows that the domination of $v^*$ is also preserved at the first arrival, $\tau_1=t_1$, giving 
	\begin{equation*}
		v^*\prec x_{t_1}\prec \frac{x_{t_1}D_\alpha}{x_{t_1}D_\alpha\bo}=\mathcal L(B_{t_1}\mid\tau_1=t_1)=y 
	\end{equation*}
	(and defining $y$ here). We then treat $y$ as an initial distribution, and the stochastic domination of $v^*$ continues up until, and including precisely at, the next arrival: $v^*\prec y_t$ and $v^*\prec \frac{y_tD_\alpha}{y_tD_\alpha\bo}$ for any $t>0$. By induction, we conclude that the stochastic domination of $v^*$ is preserved for all time and any history of arrivals. 
\end{proof}

Lemmas \ref{lem:dcm-matrix-form} and \ref{lem:v<L(B_t|H)} provide that $v^*\prec x\prec y$ implies $v^*\prec x_t\prec y_t$, but this is only guaranteed while the first arrival has not yet occurred. At arrivals, the domination of $v^*$ is preserved, but the domination between processes with distinct ordered initial conditions may not be. In other words, assume $(X_t,B_t),(X'_t,B'_t)$ are coupled MMPPs with the same $\alpha,T$ but that $v^*\prec\mathcal L(B_0)\prec\mathcal L(B'_0)$. Then it is guaranteed that $v^*\prec\mathcal L(B_t\mid\tau_1>t)\prec\mathcal L(B'_t\mid\tau'_1>t)$, but it is possible that $\mathcal L(B_{t_1}\mid\tau_1=t_1)\not\prec\mathcal L(B'_{t_1}\mid\tau'_1=t_1)$. It is not hard to come up with examples of this, for example, let $\alpha=(0,1,2)$. Then, for $x,y$ satisfying $x\prec y$ and $y_1>x_1$, it follows that $xD_\alpha/xD_\alpha\bo\not\prec yD_\alpha/yD_\alpha\bo$. Knowing $v^*$ is unimportant here, as we can always construct such $x$ and $y$ when $v^*$ is strictly positive (which is always the case for us). Letting $x,y$ be the initial distributions for $B,B'$, respectively, gives $\mathcal L(B_{t_1}\mid\tau_1=t_1)\not\prec\mathcal L(B'_{t_1}\mid\tau'_1=t_1)$ for small enough $t_1$. 

We now proceed to discuss some properties of the matrix $D_\alpha-T$, and, for this, we require some standard matrix theory results.

A matrix $M$ is called \textit{irreducibly diagonally dominant} if it is irreducible (strongly-connected directed graph) and if $|M_{ii}|\geq \sum_{j\neq i} |M_{ij}|$ for all $i$, with the inequality strict for at least one $i$ (see Definition 6.2.25 in \cite{matan_horn}). For an irreducibly diagonally dominant matrix with strictly positive diagonal, all eigenvalues have strictly positive real part (see Corollary 6.2.27 therein). It is straightforward to see that $D_\alpha-T$ is irreducibly diagonally dominant since $T$ is irreducible and $\alpha$ not strictly all zeros. This is an important fact for us, so we state it as a lemma. 

\begin{lemma}
	\label{lem:Da-T_idd}
	The matrix $D_\alpha-T$ is irreducibly diagonally dominant when $T$ is an irreducible generator matrix and $\alpha$ not identically all zeros. Hence, all eigenvalues have positive real part.
\end{lemma}
\begin{proof}
	That $T$ is an irreducible generator implies $D_\alpha-T$ is also irreducible and that the latter has a strictly positive diagonal. Since we assume that $\alpha$ has at least one nonzero component, it follows that $\alpha_j+|T_{jj}|>\sum_{i\neq j}T_{ij}$ for at least one $j$.  
\end{proof}

See \cite{matan_horn}, again, for these next few statements. An irreducible, nonnegative matrix $M$ is called \textit{primitive} if it has a unique maximum-modulus eigenvalue (Definition 8.5.0 therein). This eigenvalue is its spectral radius, denoted $\rho(M)$. For an irreducible, nonnegative matrix, its spectral radius is a simple eigenvalue and has strictly positive left and right eigenvectors (Theorem 8.4.4 therein). These eigenpair components are often called the Perron root and Perron vectors. A simple eigenvalue has algebraic and geometric multiplicities equal to one and is thus associated with a one-dimensional eigenspace (Definition 1.4.3 therein). A nonnegative matrix which is irreducible and aperiodic is primitive (Theorem 8.5.3 therein). A strictly positive matrix is primitive (Theorem 8.2.8 therein). 

\begin{lemma}\label{lem:lamstar-v-pos}
	Assume $T$ is an irreducible generator matrix and that $\alpha$ is a nonnegative, nonzero vector. 
	The matrix $D_\alpha-T$ has unique minimal-modulus, simple, eigenvalue $\alpha^*>0$ associated with a strictly positive eigenvector $v'$, which is the equilibrium no arrival distribution for the MMPP with $\alpha,T$, that is, $v'=v^*$ from Definition \ref{def:v}.
\end{lemma}
\begin{proof}	
	With $I$ being the identity matrix, let $c>0$ be large enough so that $T-D_\alpha+cI$ is a nonnegative matrix with nonzero diagonal. It is also irreducible since $T$ is, and it is aperiodic since it has at least one nonzero diagonal entry. It is thus primitive---see Theorem 8.5.3 in \cite{matan_horn} regarding the directed graph being aperiodic and its relation to primitivity. It has (left) Perron eigenvector $v'>0$ and spectral radius $\rho(T-D_\alpha+cI)=-\lambda'+c$, which is its unique maximum-modulus eigenvalue due to primitivity (why we write the eigenvalue in this form will become clear shortly). 
	
	This implies that $v',\lambda'$ is an eigenpair for $D_\alpha-T$. Since $D_\alpha-T$ is irreducibly diagonally dominant, we know that $\lambda'>0$ (by Lemma \ref{lem:Da-T_idd} above). By convention, we always assume our Perron vector's components sum to one: $v'\bo=1$ (see Theorem 8.2.8(f) in \cite{matan_horn}, but note that they use the opposite convention where the right Perron vector sums to one).
	
	Every eigenvalue $-\lambda$ of $T-D_\alpha$ gives $-\lambda+c$ as an eigenvalue of $T-D_\alpha+cI$. It follows that $\mathrm{Re}(-\lambda+c)<-\lambda'+c$ since the latter is the spectral radius of $T-D_\alpha+cI$ and its unique maximum-modulus eigenvalue. This gives $\mathrm{Re}(\lambda)-c=\mathrm{Re}(\lambda-c)>\lambda'-c$ and hence that $\lambda'<\mathrm{Re}(\lambda)$ (and also $\lambda'<|\lambda|$) for any other eigenvalue $\lambda$ of $D_\alpha-T$. Thus, $\lambda'$ is the (unique) minimal-modulus eigenvalue of $D_\alpha-T$. We now denote it by $\alpha^*$ via Definition \ref{def:alpha*}.
	
	All eigenvalues of $T-D_\alpha$ have negative real part (since $T-D_\alpha=-(D_\alpha-T)$ and Lemma \ref{lem:Da-T_idd} shows that all eigenvalues of $D_\alpha-T$ have positive real part). Because $\alpha^*$ is the unique minimal-modulus eigenvalue of $D_\alpha-T$, all other eigenvalues of $T-D_\alpha$ have negative real part strictly less than $-\alpha^*$.
	
	It is a standard result that $v',-\alpha^*$ being an eigenpair for $T-D_\alpha$ implies that $v',e^{-t\alpha^*}$ is an eigenpair for $e^{t(T-D_\alpha)}$. It is also a well-known fact that $e^{t(T-D_\alpha)}\geq0$ for $T$ an arbitrary generator matrix and arbitrary nonnegative $\alpha$ and that $e^{t(T-D_\alpha)}>0$ if and only if $T$ is irreducible (see Theorem 3.12 in Chapter 6 of \cite{berman1994}). Hence, $e^{t(T-D_\alpha)}$ is strictly positive. Corollary 8.1.30 in \cite{matan_horn} says that a positive eigenvector for a nonnegative matrix is associated with the Perron root (the maximum-modulus eigenvalue, its spectral radius). This implies that $\rho(e^{t(T-D_\alpha)})=e^{-t\alpha^*}$. Since $e^{t(T-D_\alpha)}>0$, it is primitive and, thus, $e^{-t\alpha^*}$ is a simple eigenvalue associated with a one-dimensional eigenspace. 
	
	Using \eqref{eq:xt-def} and Lemma \ref{lem:new-L(Bt|H)}, we have that 
	\begin{equation}\label{eq:xt-matrix-def}
		\mathcal{L}(B_t\mid X_s=0~\forall s\in[0,t])=x_t=
		\frac{x e^{t(T-D_\alpha)}}{x e^{t(T-D_\alpha)}\bo}.
	\end{equation}
	We wish to show that $x_t\to v^*$ as $t\to\infty$. 
	
	Let the (real) Jordan canonical form (see the Theorem in 1.8 of \cite{perkoDEtext}) of $T-D_\alpha$ be 
	\begin{equation}\label{eq:J-norm-form}
		T-D_\alpha=V^{-1} 
		\left(\begin{matrix}
			-\alpha^* & 0\\
			\mathbf 0 & J
		\end{matrix}\right)
		V
	\end{equation}
	where $J$ is a real matrix and captures all remaining eigenvalues (besides $-\alpha^*$) and any Jordan blocks or blocks associated with complex-conjugate eigenvalue pairs
	---also see Theorem 3.4.1.5 or Lemma 6.3.10 in \cite{matan_horn} for the existence of such a form. Note that $\mathbf 0$ is a column of zeros. 
	In fact, the rows of $V$ (and hence the columns of $V^{-1}$) form a basis for $\mathbb R^k$. 
	The first row of $V$ is our eigenvector $v'$. Using \eqref{eq:J-norm-form}, it follows that 
	\begin{equation*}
		\frac{e^{t(T-D_\alpha)}}{e^{-t\alpha^*}}=
		V^{-1} 
		\left(\begin{matrix}
			1 & \quad 0\\
			\mathbf 0 & \quad e^{t(J+\alpha^* I)} 
		\end{matrix}\right)
		V
	\end{equation*}
	where $I$ is an appropriate identity matrix. Note that $J+\alpha^* I$ is a real matrix, and all of its eigenvalues have negative real part since $J$ contains all eigenvalues of $T-D_\alpha$ besides $-\alpha^*$, and all of these additional eigenvalues have negative real part strictly less than $-\alpha^*$.
	
	Divide the numerator and denominator of \eqref{eq:xt-matrix-def} both by $e^{-t\alpha^*}$. Then, stability theory (Theorem 2 in 1.9 of \cite{perkoDEtext}) implies that $e^{t(J+\alpha^* I)}\to0$ (the zero matrix) as $t\to\infty$ since all eigenvalues of real matrix $J+\alpha^* I$ have negative real part. Writing $x_t$ in coordinates with respect to the rows of $V$ shows that all coefficients vanish except the one corresponding to $v'$ (associated with the first component of $x V^{-1}$). Since $x_t\bo=v'\bo=1$, this yields convergence to $v'$. That the column eigenvector in $V^{-1}$ associated with $\alpha^*$ is also positive (it is the right Perron vector of a primitive matrix) provides that the coordinate coefficient corresponding to $v'$ is always positive for any nonnegative, nonzero $x$. This implies that $v'$ is the equilibrium no arrival distribution for our MMPP. In other words, $v'=v^*$ as defined in Definition \ref{def:v}.
	
	We conclude that $\alpha^*>0$ is the simple, minimal-modulus eigenvalue of $D_\alpha-T$ and that $v^*>0$ is the associated eigenvector and the equilibrium no arrival distribution.
\end{proof}

\subsection{Proof of Theorem \ref{thm:main_mpp_dom}} 
\label{sec:mmpp_dom_proof}

We begin by proving (i). 
Note that $T\bo=\mathbf 0$ (the zero column vector) since $T$ is a generator, and recall that $\alpha^*,v^*$ is an eigenpair for $D_\alpha-T$ with $v^*\bo=1$ (Lemma \ref{lem:lamstar-v-pos}). It is now straightforward to see that 
\begin{equation}\label{eq:a=av1}
	\alpha^*=\alpha^*v^*\bo =v^*(D_\alpha-T)\bo =v^*D_\alpha\bo.
\end{equation} 

Let $\alpha_j=\alpha_0+\gamma_1+\cdots+\gamma_j$ and note that all $\gamma_j\geq0$ for all $j$ since $\alpha$ is increasing. Then, 
\begin{equation*}
	xD_\alpha\bo=\alpha_0+\sum_{j=1}^{k-1}\gamma_j(x_j+\cdots+x_{k-1}).
\end{equation*}
For any $x\prec y$, it follows that  $\gamma_j(x_j+\cdots+x_{k-1})\leq\gamma_j(y_j+\cdots+y_{k-1})$ for all $j$ and, hence, that $xD_\alpha\bo\leq yD_\alpha\bo$. Using the fact that $v^*\prec\mathcal L(B_0)$ implies $v^*\prec \mathcal L(B_t\mid \mathcal{H}_t)$ for arbitrary history $\mathcal H_t$ (Lemma \ref{lem:v<L(B_t|H)}), along with the CIF definition \eqref{eq:cif} and \eqref{eq:a=av1}, then shows that 
\begin{equation*}
	\alpha^*=v^*D_\alpha\bo\leq \mathcal L(B_t\mid \mathcal{H}_t)D_\alpha\bo=\lambda(t\mid\mathcal{H}_t).
\end{equation*}
Note that this is the only way that downward conditional monotonicity is used in this proof, to establish that the conditional distribution of $B_t$ is always bounded below by $v^*$. That $\alpha$ is increasing is then used to ensure that the CIF is bounded below by $\alpha^*$.

By a standard result (see Theorem 2 in \cite{rolski1991}), $\alpha^*\leq \lambda(t\mid\mathcal{H}_t)$ implies that $Y^\lambda\prec X$ for all $\lambda\leq\alpha^*$ (with $Y^\lambda\sim\Poi(\lambda)$ the Poisson counting process). Recall that this is domination of the underlying point processes as described in the paragraph below \eqref{eq:x<y-def} in the introduction. Thus, the optimal coupling parameter (defined by \eqref{eq:lam-opt-def}) satisfies $\lopt\geq\alpha^*$. 

To see necessity of $\lambda\leq\alpha^*$ for $Y^\lambda\prec X$, let $v^*\prec x$, and consider that the probability of the decreasing event $E_t=\{\text{no arrivals in } [0,t]\}$ for $X$ is $xe^{t(T-D_\alpha)}\bo$ (see the text after \eqref{eq:x'Q-to-xTD}). 
It follows that $Y^\lambda\prec X$ requires, for all $t>0$, 
\begin{equation}\label{eq:YX-no-arriv[0,t]}
	\mathbb P^{Y^\lambda}(E_t)=e^{-\lambda t}\geq xe^{t(T-D_\alpha)}\bo=\mathbb P^X(E_t).
\end{equation}
Solving \eqref{eq:YX-no-arriv[0,t]} for $\lambda$ shows that $Y^\lambda\prec X$ requires, for all $t>0$, 
\begin{equation}\label{eq:lam<-1tlogxetbo}
	\lambda\leq -\frac{1}{t}\log\left(xe^{t(T-D_\alpha)}\bo\right).
\end{equation}
Taking the limit of \eqref{eq:lam<-1tlogxetbo} as $t\to\infty$ requires a use of L'H\^opital's:
\begin{equation}\label{eq:lim-alph-star}
	\begin{aligned} 
		\lambda\leq\lim_{t\to\infty}-\frac{xe^{t(T-D_\alpha)}(T-D_\alpha)\bo}{xe^{t(T-D_\alpha)}\bo}&=\lim_{t\to\infty}\frac{xe^{t(T-D_\alpha)}}{xe^{t(T-D_\alpha)}\bo}D_\alpha\bo\\
		&=\lim_{t\to\infty}x_tD_\alpha\bo\\
		&=v^*D_\alpha\bo=\alpha^*.
	\end{aligned}
\end{equation}
This uses the convergence of $x_t$ to $v^*$ as shown in the proof of Lemma \ref{lem:lamstar-v-pos}. 
This establishes that $\lambda\leq\alpha^*$ is required. We conclude that $\lopt=\alpha^*$, proving (i). This argument includes some aspects of arguments in Section 4 of \cite{broman2007} but with a significantly different construction. 

Next, we prove (ii). Sufficiency of $\lambda\geq\max_j\alpha_j=\alpha_{k-1}$ for $X \prec Y^\lambda$ is trivial. Necessity follows from a straightforward adaptation of the arguments in \cite{broman2007} (see the proof of Theorem 1.4 there)---that $\lambda<\alpha_{k-1}$ implies that there is an $n$ large enough so that 
\begin{equation}\label{eq:PXn-arriv>PYn-arriv}
	\mathbb P(X \text{ has at least $n$ arrivals in } [0,t])> \mathbb P(Y^\lambda \text{ has at least $n$ arrivals in } [0,t])
\end{equation} 
which contradicts that $X\prec Y^\lambda$ since having at least $n$ arrivals is an increasing event for the space of count paths. See Remark \ref{rmk:alpha-max-gen-T} below for a related and more detailed computation.

We conclude that, for our MMPP with irreducible, monotone $T$, $\alpha$ increasing and initial background distribution $B_0\sim x$ satisfying $v^*\preceq x$, the point process $N(X)$ associated with our counting process $X$ dominates a Poisson point process with constant intensity $\lambda$ if and only if $\lambda\leq\alpha^*$, the minimal eigenvalue for $D_\alpha-T$, and is dominated by a Poisson point process with intensity $\lambda$ if and only if $\lambda\geq\max_j\alpha_j=\alpha_{k-1}$. \hfill $\square$

\medskip

Note that a sufficiently long interarrival time can make the CIF approach $\alpha^*$ arbitrarily closely. This can be seen by recognizing that the right side of \eqref{eq:lim-alph-star} shows the CIF converging to $\alpha^*$ as the wait time to the next arrival approaches infinity (with arbitrary distribution $x$). 

Also, a sufficiently large number of rapid arrivals with short enough interarrival times can make the CIF approach $\alpha_{k-1}$. To see this latter case, consider $\mathcal L(B_t\mid\mathcal H_t)$ the conditional distribution of $B_t$ given by \eqref{eq:Bt|Htn}. Assume arbitrarily small interarrival times $t_j-t_{j-1}$, and let the number of arrivals in $[0,t]$ grow unboundedly, $n\to\infty$. 
This makes the CIF \eqref{eq:cif}
\begin{equation*}
	\lambda(t\mid\mathcal H_t)=\mathcal L(B_t\mid\mathcal H_t)D_\alpha\bo \approx \frac{x D_\alpha^n}{x D_\alpha^n\bo}D_\alpha\bo\overset{n\to\infty}{\longrightarrow} e_{k-1}D_\alpha\bo=\alpha_{k-1}
\end{equation*}
due to standard Perron--Frobenius convergence.

However, these two facts alone are not sufficient to establish our optimal bounds of $\alpha^*$ and $\alpha_{k-1}$. 
Domination of CIFs is only a sufficient condition and not necessary for domination of point processes---it is not hard to construct counterexamples which prove this. For example, on $[0,1]$, let $N(X)$ have no points with probability $e^{-\delta}$ and have a single, uniformly-distributed point otherwise. That $\mathcal L(N(X))\prec \PPP(\delta)$ easily follows via coupling, however, by Proposition 7.2.I and (7.2.3) in \cite{daley_pp}, the CIF for $N(X)$ can be computed as 
\begin{equation*}
	\lambda(t\mid\mathcal \tau_1>t)=\frac{-\frac{d}{dt}\mathbb P(\tau_1>t)}{\mathbb P(\tau_1>t)}=\frac{1-e^{-\delta}}{1-t+te^{-\delta}}
\end{equation*}
and satisfies $\lambda(t\mid\tau_1>t)>\delta$ for large enough $t\in(0,1)$. 

\begin{remark}\label{rmk:shift-alpha-a}
	One can assume that $\alpha_0=0$ since uniformly shifting the arrival rates $\tilde{\alpha}=\alpha+a$ simply shifts the minimal eigenvalue: $\tilde{\alpha}^*=\alpha^*+a$.
\end{remark}

\begin{remark}\label{rmk:alpha-max-gen-T}
	As long as $T$ is irreducible, Theorem \ref{thm:main_mpp_dom}(ii) still applies when $\alpha$ is neither increasing nor decreasing (and with arbitrary initial distribution $x$). Assume state $m\in S$ gives the maximum arrival rate: $\max_j\alpha_j=\alpha_m$ (there may be several such states). Since $T$ is irreducible, for any $\epsilon>0$, 
	\begin{equation*}
		\mathbb P(B_t=m ~\forall t\in[\epsilon,1+\epsilon])=(xe^{\epsilon T})_me^{-\sum_{j\neq m}T_{mj}}>0.
	\end{equation*}
	Letting $E=\{\text{at least $n$ arrivals in }[\epsilon,1+\epsilon]\}$, we have 
	$\mathbb P^{Y^\lambda}(E)=\sum_{j=n}^\infty \frac{e^{-\lambda} \lambda^j}{j!}$, and
	\begin{equation*}\label{eq:PX>n-arriv}
		\mathbb P^X(E)\geq (xe^{\epsilon T})_me^{-\sum_{j\neq m}T_{mj}} \sum_{j=n}^\infty \frac{e^{-\alpha_m} \alpha_m^j}{j!}.
	\end{equation*}
	If $\lambda<\alpha_m$, then there is an $n$ large enough so that \eqref{eq:PXn-arriv>PYn-arriv} is satisfied, i.e., $\mathbb P^X(E)>\mathbb P^{Y^\lambda}(E)$, 
	contradicting that $X\prec Y^\lambda$. Hence, $\lambda\geq\max_j\alpha_j$ is necessary (and trivially sufficient) for $X\prec Y^\lambda$ irrespective of any monotonicity properties of $\alpha$. Reducible $T$ simply requires considering the maximum arrival rate accessible given a particular initial distribution $x$.
\end{remark}

\begin{remark}\label{rmk:non-dcm}
	Theorem \ref{thm:main_mpp_dom}(i) only establishes $\alpha$ being increasing, $T$ being monotone, and DCM as sufficient conditions. The case when $\alpha$ is decreasing has already been mentioned (see Remarks \ref{rmk:decr-alpha} and \ref{rmk:alpha-decr-x<v}). It is not hard to find example MMPP models where neither $\alpha$ nor $T$ are monotone, and which are non-DCM, but still satisfy $\lopt=\alpha^*$. This can be accomplished via the method of spectral radius-preserving row sum expansions developed in \cite{stover2022} or via permutations---it is not hard to understand that such methods have no effect on the CIF since the ordering on the state space has no bearing on the CIF. We leave exploring this further for future work. 
\end{remark}

\section{Survival and extinction for the CPMRE}
\label{sec:CPMRE}

Our CPMRE result is, in part, a straightforward consequence of our MMPP domination result. The main necessity is a graphical construction which allows appropriate comparisons to contact processes. All of our CPMRE models will be assumed to have $\beta=(\beta_0,\beta_1,\ldots,\beta_{k-1})$ increasing and $\mu=(\mu_0,\mu_1,\ldots,\mu_{k-1})$ decreasing so that higher environmental states are more favorable to the infection with higher infection rates and lower recovery rates. As noted in the Introduction, $\beta,\mu$ will be treated as analogous to $\alpha$, the arrival rate vector of an MMPP.

Our random environment uses a single generator matrix $T$ which is always assumed irreducible and monotone. In this section, we use downward conditional monotonicity extensively---see Definition \ref{def:DCM}, and we assume Theorem \ref{thm:main_mmpp_cm} here as well. Much of Section \ref{sec:mmpp-dom} is also relevant here as we use Theorem \ref{thm:main_mpp_dom} to prove Theorem \ref{thm:main_CPMRE}.

Recall that, since $\mu$ is decreasing, we use the permutation-reversed \MMPPt{$\widehat\mu,\widehat{T}$} (defined just before the statement of Remark \ref{rmk:decr-alpha}) to assess for DCM. When \MMPPt{$\beta,T$} is DCM, the CPMRE is referred to as \textit{infection-DCM}, and when \MMPPt{$\widehat\mu,\widehat{T}$} is DCM, the CPMRE is called \textit{recovery-DCM}. 

We denote the equilibrium no arrival distributions (see Definition \ref{def:v}) for these MMPPs by $v^\beta$ and $v^\mu$, respectively. Recall that $\beta^*>0$ and $\mu^*>0$ are the respective minimal eigenvalues of $D_\beta-T$ and $D_\mu-T$ (via Definition \ref{def:alpha*}) and that $v^\beta$ and $v^\mu$ are their associated (strictly positive) eigenvectors (see Lemmas \ref{lem:Da-T_idd} and \ref{lem:lamstar-v-pos}).

Theorem \ref{thm:main_CPMRE} states that we have no restrictions on the initial distribution at any site. Its proof (in Section \ref{sec:cpmre-proof}) first assumes the initial distribution to be $\pi$ at every site. Using our DCM property requires the initial distribution to satisfy stochastic domination relations with the equilibrium no arrival distributions, and we will shortly see that $\pi$ satisfies this. Remark \ref{rmk:alpha-decr-x<v} is important here since we assume $\mu$ is decreasing. 

The following lemma states that the equilibrium no arrival distribution is strictly dominated by the stationary distribution for the background process (under suitable assumptions on $\alpha,T$). Recall that strict domination means the main inequalities comparing the summations in \eqref{eq:x<y-def} are all strict except for when  $\ell=k-1$. This next lemma is important for relaxing all restrictions on the initial environmental distribution. 
\begin{lemma} \label{lem:v-sdom-pi}
	Consider the MMPP with irreducible, monotone $T$, increasing $\alpha$, and which is DCM. Then,  the equilibrium no arrival distribution is strictly dominated by the stationary distribution: $v^*\precn \pi.$
\end{lemma}
The proof of Lemma \ref{lem:v-sdom-pi} is given in Section \ref{sec:v-sdom-pi}, since its proof uses methods similar to those repeated several times in Section \ref{sec:dcm}. Note that the conditions in Lemma \ref{lem:v-sdom-pi} are only proven as sufficient for the stated strict domination.

Applying Lemma \ref{lem:v-sdom-pi} and its reversal for decreasing $\alpha$ (again, see Remark \ref{rmk:alpha-decr-x<v}), since $\beta$ is increasing and $\mu$ decreasing, we see that   
\begin{equation}\label{eq:vlam<pi<vdel}
	v^\beta\precn \pi\precn v^\mu.
\end{equation}
If our model is infection-DCM, then our MMPP domination result (Theorem \ref{thm:main_mpp_dom}) is only useful if the initial environmental distribution $b$ satisfies $v^\beta\prec b$. Similarly, for recovery-DCM to be useful requires $b\prec v^\mu$. The stationary distribution $\pi$ satisfies both.

Here is a preview on how we use Lemma \ref{lem:v-sdom-pi} to relax our requirements on the initial environmental distribution. For any distribution $b$, it follows that $be^{sT}\to\pi$ as $s\to\infty$. The strict dominations in \eqref{eq:vlam<pi<vdel} show that we can pick an $s$ large enough so that $v^\beta\prec be^{sT}\prec v^\mu$ (note that these are actually strict dominations, but that is not important here). To understand more clearly why this is important, $u\prec \pi\prec w$ (for some distributions $u,w$) does not generally imply that $u\prec be^{sT}\prec w$ for large $s$, and it is not hard to find such examples. For example, if $u_0=\pi_0$ and $b\prec u$, then it is possible that $(be^{sT})_0>u_0$ for all $s>0$, and we would never have $u\prec be^{sT}$ for any large $s$. The strict dominations in \eqref{eq:vlam<pi<vdel} give strict inequalities when applying \eqref{eq:x<y-def} (ignoring $\ell=k-1$). These strict dominations guarantee that 
\begin{equation*}
	\sum_{i\leq\ell}v^\beta_i >\underbrace{\sum_{i\leq\ell}\big(be^{sT}\big)_{i}}_{\approx\sum_{i\leq\ell}\pi_i} > \sum_{i\leq\ell}v^\mu_i
\end{equation*}
when $s$ is sufficiently large (and $\ell\neq k-1$), giving $v^\beta\precn be^{sT}\precn v^\mu$. Hence, regardless of the initial environmental distributions, we can wait long enough so that the distributions at all sites simultaneously match our required conditions for DCM to be useful.  

%

\subsection{Graphical construction}
\label{sec:CPMRE_graphrep}

Let $\Omega$ denote the space of all graphical structures which we are about to describe. For each site $x\in\mathbb Z^d$, we extend a timeline ``upward.'' Along each timeline, there are several Poisson point processes associated with where infections and recoveries may occur. Our next step is to describe in detail how this graphical structure is constructed along the timeline for each site $x$. An example graphical structure is shown in Figure \ref{fig:CPMRE_graph}---the symbols and notation in the figure and its caption will be explained shortly. Recall that, in this section and generally whenever discussing the CPMRE, we use $x,y$ to represent sites in the integer lattice and not distributions. 

\begin{figure}[htb]
	\begin{tikzpicture}[x=1.5cm,y=0.75cm]
		
		\path[use as bounding box] (-0.2,-0.5) rectangle (7.1,8.0);
		
		
		\begin{scope}[transparency group, opacity=0.4]	
			\draw[pathred] (0,0) -- (0,6.55);
			\draw[pathred] (0,1.1) -- (1,1.1) -- (1,2.75);
			\draw[pathred] (0,5.1) -- (1,5.1) -- (1,7.65);
		\end{scope}
		
		\begin{scope}[transparency group, opacity=0.4]	
			\draw[pathgreen] (3,0) -- (3,5.05);
			\draw[pathgreen] (3,3.4) -- (2,3.4) -- (2,8);
			\draw[pathgreen] (2,5.9) -- (3,5.9) -- (3,8);
			\draw[pathgreen] (3,6.4) -- (4,6.4) -- (4,8);
			\draw[pathgreen] (3,1.55) -- (4,1.55) -- (4,4.45);
		\end{scope}
		
		\begin{scope}[transparency group, opacity=0.4]	
			\draw[pathblue] (6,0) -- (6,8);
			\draw[pathblue] (6,2.2) -- (5,2.2) -- (5,5.95);
			\draw[pathblue] (6,4.8) -- (5,4.8);
			\draw[pathblue] (6,6.9) -- (5,6.9) -- (5,8);
			\draw[pathblue] (6,0.9) -- (7,0.9) -- (7,2.45);
			\draw[pathblue] (6,3.6) -- (7,3.6) -- (7,6.65);
			\draw[pathblue] (6,5.6) -- (7,5.6);
			\draw[pathblue] (6,7.5) -- (7,7.5) -- (7,8);
		\end{scope}
		
		
		\draw[worldline] (-0.25,0) -- (7.25,0);
		
		\foreach \i in {0,...,7}{
			\draw[timeline] (\i,0) -- (\i,8);
		}
		
		\draw[timeaxis] (-0.35,2.5) -- (-0.35,5.5);
		\node[left] at (-0.45,4) {\large $t$};
		
		\node[below] at (3.5,-0.2125) {\large$x$};
		
		
		\rarrowcircstar{0}{1}{1.1}
		\rarrowplain{0}{1}{3.98}
		\rarrowcircstar{0}{1}{5.1}
		\markot{0}{6.6}
		
		\rarrowcirc{1}{2}{1.9}
		\markx{1}{2.8}
		\rarrowplain{1}{2}{7}
		\rarrowcirc{1}{0}{7.2}
		\markotstar{1}{7.7}
		
		\markx{2}{4.55}
		\rarrowcirc{2}{3}{5.9}
		\rarrowplain{2}{1}{6.18}
		\rarrowplain{2}{3}{7.5}
		
		\markx{3}{0.75}
		\rarrowcircstar{3}{4}{1.55}
		\larrowcircstar{3}{2}{3.4}
		\markotstar{3}{5.1}
		\rarrowcirc{3}{4}{6.4}
		
		\rarrowplain{4}{5}{3.35}
		\markot{4}{4.5}
		\markx{4}{7.6}	
		
		\markx{5}{4}
		\markotstar{5}{6.0}
		
		\rarrowplain{6}{7}{0.9}
		\rarrowcirc{6}{5}{2.2}
		\markx{6}{3.1}
		\rarrowcircstar{6}{7}{3.6}
		\markot{6}{4.4}
		\rarrowplain{6}{5}{4.8}
		\rarrowcircstar{6}{7}{5.6}
		\rarrowcirc{6}{5}{6.9}
		\rarrowplain{6}{7}{7.5}
		
		\markotstar{7}{2.5}
		\markotstar{7}{6.7}
		
	\end{tikzpicture}
	\caption{
		An example of our graphical construction is shown. Note that this graph includes nonempty $A_*^{xy}$ and $D_*^x$ which assumes the initial distribution allows this (e.g., $B^x_0\sim\pi$ for all $x$ is sufficient). Also, we disregard the numerical labels ($j\in S$) for all symbols as they are only important for determining whether they are marked with a circle or not. Example paths for the \CPt{$\beta^*,\mu_0$} (left/red), CPMRE (center/green), and \CPt{$\beta_{k-1},\mu^*$} (right/blue) are shown. The \CPt{$\beta^*,\mu_0$} uses only arrows marked with an asterisk but uses all recovery symbols. The CPMRE uses only arrows marked with circles and circled recovery symbols. The \CPt{$\beta_{k-1},\mu^*$} uses all arrows but only recovery symbols marked with an asterisk.
	}
	\label{fig:CPMRE_graph}
\end{figure}
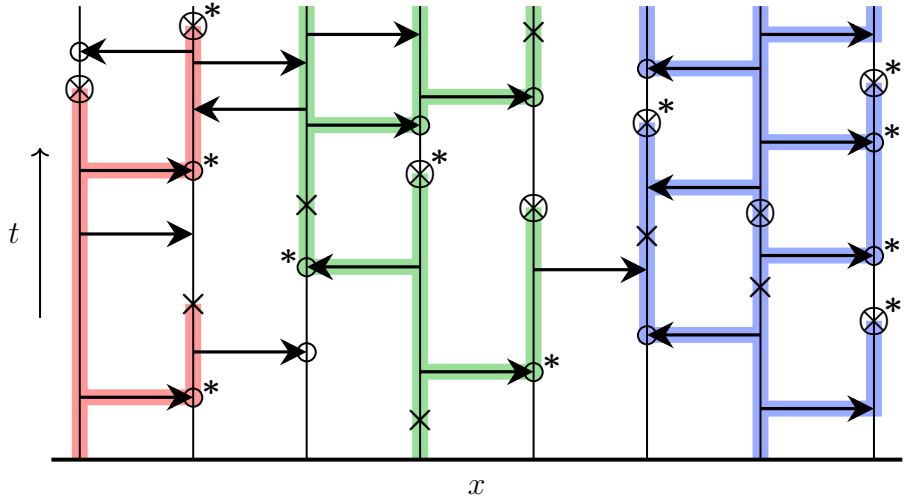

For each $j=0,\ldots, k-2$, denote by $D^x_j$ a Poisson point process with intensity $\mu_{j}-\mu_{j+1}$ and $D_{k-1}^x$ with intensity $\mu_{k-1}$. Distribute  ``$\scalebox{0.9}{$\boldsymbol\times$}_j$'' symbols, labeled with a $j$, along the timeline for site $x$ at each point in $D_j^x$. These are where potential recoveries can occur at site $x$. 

For each $j=1,\ldots,k-1$, and for each $y\in\mathcal N(x)$ (the set of nearest neighbors for site $x$), denote by $A^{xy}_j$ a Poisson point process with intensity $\beta_j-\beta_{j-1}$ and $A^{xy}_0$ with intensity $\beta_0$. Distribute arrow symbols ``$\gto_j$'' pointing from site $y$ to site $x$, and labeled with a $j$, at each point in $A_j^{xy}$. These denote where the infection can transmit from site $y$ to site $x$. 

If the background process is in state $\ell$ for some interval of time at site $x$, then arrows are usable only if they are marked with $j\leq\ell$, and recovery symbols are usable only if they are marked with $j\geq\ell$. We refer to such symbols as compatible with the background state.

Let $D^x=\bigcup_{j=0}^{k-1} D_j^x$ be the collection of all recovery symbols for site $x$ and, similarly, let $A^{xy}=\bigcup_{j=0}^{k-1} A^{xy}_j$ be all arrows from $y$ to $x$. 
We view all of these point processes simply as random sets of points. 

Now, we fix a realization of the random environment for each site $x$. In other words, for each realization of the graphical structure $\omega\in\Omega$, the timeline for each site $x$ has a fixed cadlag sample path $B^x(\omega): [0,\infty)\to S$ associated with it, and these paths are chosen independently. Assume that the each site $x$ has its initial distribution specified. Once the environmental state at each site and for all time is known, then any recovery symbol that is compatible with the background state is circled: $\scalebox{0.9}{$\boldsymbol\otimes$}_j$. Similarly, any infection arrow that is compatible with the background state is marked with a circle at its tip: $\gtoCirc$. Denote the sets of points adorned with circles as $D^x_\circ$ and $A^{xy}_\circ$. After marking with circles, the numerical labels on symbols are no longer important. 

Note that, for a fixed timeline, the collection of arrows marked with circles, $A^{xy}_\circ$, is simply a realization of the point process associated with the \MMPPt{$\beta,T$} process. Likewise, that for circled recovery symbols is a realization of the arrival times of the MMPP with $\mu,T$. 

Envision placing fluid at each site which is to be initially infected, at $t=0$. The fluid percolates up the graphical structure to simulate the CPMRE using only those symbols marked with circles. The fluid is stopped at any circled recovery symbol and branches along arrows with circled tips. The set of infected sites at time $t$ is those where fluid is present.

For fixed parameters, $\beta,\mu,T$, and $B_0^x\sim b$ for all $x$, we have a probability measure, $\mathbb P^b$, on $\Omega$ the set of all graphs. 
This graphical construction is suitable for our CPMRE with any set of initially infected sites. The CPMRE only uses symbols marked with circles, of course. 

Note that this graphical construction is also suitable for a standard contact process with fixed infection rate $\beta_i$ and fixed recovery rate $\mu_j$ for arbitrary $i,j\in S$. Such a contact process can use any arrow symbol with label less than or equal to $i$. It uses any recovery symbol with numerical label greater than or equal to $j$. 

We now describe some point processes which are appended onto the graphical structure only when the initial environmental distribution at each site allows it. Lemma \ref{lem:v-sdom-pi} shows it is sufficient that $B_0^x\sim\pi$ at each site, for example.

When the initial distribution of the environment at each site stochastically dominates $v^\beta$, then a set of graphs is constructed simply including a Poisson point process $A_*^{xy}\sim\PPP(\beta^*)$ along each timeline and for each neighboring $y\in\mathcal N(x)$.  We couple these graphs into ours so that $A_*^{xy}\subset A_\circ^{xy}$ (with probability one)---this is achievable due to Theorem \ref{thm:main_mpp_dom} and explicitly via standard CIF thinning methods (see, e.g., Section 7.5 in \cite{daley_pp}). We simply mark any such arrow in $A_\circ^{xy}$ that coincides with a point in $A_*^{xy}$ with an asterisk ``$*$'' label, in addition to their being labeled with a value of $j$ and a circle at the tip: $\gtoCircStar$.  

Similarly, when the initial distribution of the environment at each site is dominated by $v^\mu$, we construct point processes $D_*^x\sim\PPP(\mu^*)$ along each timeline and couple these into our graphical representation so that $D_*^x\subset D_\circ^x$ (almost surely). The coincident recovery symbols are marked with an asterisk ``$*$'' label: $\scalebox{0.9}{$\boldsymbol\otimes$}_j^{\scalebox{0.9}{$*$}}$. See Figure \ref{fig:CPMRE_graph} for an example of this graphical structure (but note that numerical labels are suppressed in this figure). 

For each site $x$, we now have $2dk+k+1$ independent processes: one background process $B^x$, several $A_j^{xy}$ for arrows ($2dk$ of them), and several $D^x_j$ for recovery symbols ($k$ of these). Hence, for example, when we couple the background and recovery processes $(B^x,D^{x}_0,\ldots,D^{x}_{k-1})$ with $D^x_*$ in order to achieve the domination $D^x_*\subset D^x_\circ$, each $A_j^{xy}$ is still independent of this new $D^x_*$. Similarly, $A_*^{xy}$ is independent of each $D_j^x$. But $A_*^{xy}$ and $D^x_*$ may be dependent since they are both coupled with $B^x$. 

Note that our construction of $\Omega$ effectively assumes that the conditions required for our MMPP dominations are satisfied, allowing inclusion of all asterisk-marked processes. We drop those processes as required when such dominations are not achievable, say, due to an incompatible choice of the initial distribution. We make no adjustments to our construction or notation though. This is easily compensated for by applying unit mass to the trivial counting measure (the empty point process) as appropriate. 

This graphical structure is also suitable for a contact process with infection rate $\beta_j$ and recovery rate $\mu^*$ or with infection rate $\beta^*$ and recovery rate $\mu_j$. Soon, we use such contact processes in our arguments. This construction may not be suitable for a contact process with infection rate $\beta^*$ and recovery rate $\mu^*$ though, since the asterisk-marked arrow and recovery symbol processes may be dependent, through their simultaneous dependence on the same background environmental process. This is not a problem as we never use such a contact process. Denote the contact process with infection rate $\lambda$ and recovery rate $\delta$ by \CPt{$\lambda,\delta$}. 

We say there is a CPMRE-path from site $x_1$ at time $t_1$ to site $x_2$ at time $t_2$, denoted by $(x_1,t_1)\rightsquigarrow(x_2,t_2)$, 
if there is a sequence of adjacent sites $x_1=y_1,y_2,\ldots,y_n=x_2$ and a sequence of times $t_1=t_0'\leq t_1'\leq\ldots\leq t_n'= t_2$ with $y_{j+1}\in\mathcal N(y_{j})$ and $t_j'\in A^{y_{j+1}y_{j}}_\circ$ for all $j=1,\ldots,n-1$ (i.e., there are circle-marked arrows at each time $t_j'$ pointing from $y_{j}$ to $y_{j+1}$) and $D_\circ^{y_j}\cap[t_{j-1}',t_{j}']=\emptyset$ for all $j=1,\ldots,n$. One can view this as placing fluid at $(x_1,t_1)$ in the graph and percolating it upward. It can only cross arrows with circles at their tips and is only stopped by circled recovery symbols; $(x_1,t_1)\rightsquigarrow(x_2,t_2)$ if and only if fluid percolates from $(x_1,t_1)$ to $(x_2,t_2)$. If no such path exists, we write $(x_1,t_1)\not\rightsquigarrow(x_2,t_2)$. 

Similarly, we say there is a \CPt{$\beta_{k-1},\mu^*$}-path, denoted $(x_1,t_1)\rightsquigarrow_{\mu^*}(x_2,t_2)$, for a path that can use any arrow symbol types and never crosses any asterisk-marked recovery symbols (but it can cross all others). A \CPt{$\beta^*,\mu_0$}-path is denoted $(x_1,t_1)\rightsquigarrow_{\beta^*}(x_2,t_2)$, only uses asterisk-marked arrows and never crosses any recovery symbols, irrespective of type. 

The above establishes that (almost surely) every CPMRE-path is also a \CPt{$\beta_{k-1},\mu^*$}-path (i.e., that $(x_1,t_1)\rightsquigarrow(x_2,t_2)$ implies $(x_1,t_1)\rightsquigarrow_{\mu^*}(x_2,t_2)$), and every \CPt{$\beta^*,\mu_0$}-path is a CPMRE-path (i.e., that $(x_1,t_1)\rightsquigarrow_{\beta^*}(x_2,t_2)$ implies $(x_1,t_1)\rightsquigarrow(x_2,t_2)$). Both claims follow since every asterisk-marked symbol is also circle-marked. 

Let $\eta=(\eta_t^x)_{x\in\mathbb Z^d,t\geq0}$ be the infection process for the CPMRE, and let $\xi$ be that for the contact process with parameters $\beta^*,\mu_0$. It follows that $\xi_0^x\leq\eta_0^x$ initially for all $x$ implies that $\xi_t^x\leq\eta_t^x$ for all $x$ and all $t>0$ (almost surely) as this is explicitly realized in our graphical construction (since every \CPt{$\beta^*,\mu_0$}-path is a CPMRE-path). This is stochastic domination between these two infection processes. Let $\mathbb P_{\mathrm{cpmre}}^{A,b}$ be the distribution of our CPMRE with $A$ the set of initially infected sites and $B_0^x\sim b$ at each site, and let $\mathbb P_{\lambda,\delta}^{A}$ be that for the contact process with infection rate $\lambda$, recovery rate $\delta$ and $A$ the set of initially infected sites. We denote this stochastic domination of infection processes by 
\begin{equation}\label{eq:P<P-def}
	\mathbb P_{\beta^*,\mu_0}^{A} \prec \mathbb P_{\mathrm{cpmre}}^{A,b}
\end{equation}
(with $v^\beta\prec b$ required to allow the MMPP domination $A_*^{xy}\subset A_\circ^{xy}$). Dominations between various infection processes are denoted similarly, and all such dominations that we use are realizable on a suitable graphical structure. 


We define survival and extinction events in general for such infection processes, for example, the CPMRE, modified versions of it, or standard contact processes. Let  $\eta=(\eta_t^x)_{x\in\mathbb Z^d,t\geq0}$ be such a process. Let $\mathcal E$ denote the event that the process goes extinct, and let $\mathcal S$ be the event that it survives strongly when the origin is the only initially infected site. Extinction means 
\begin{equation*}
	\mathcal E=\{\exists t<\infty : \eta_s^x=0 \text{ for all } x\in\mathbb Z^d \text{ and all } s>t\},
\end{equation*}
and (strong) survival means (when the origin is the only initially infected site) 
\begin{equation*}
	\mathcal S=\{\eta_t^0=1 \text{ i.o.}\}, 
\end{equation*}
where the notation ``i.o.'' assumes a sequence of times $(t_n)_{n\in\mathbb N}$ with $t_n\to\infty$ as $n\to\infty$. 

\pagebreak

Survival and extinction probabilities can now be understood in terms of this graphical representation. Extinction and survival probabilities for the CPMRE are given as 
\begin{equation*} \begin{aligned}
	\mathbb P_{\mathrm{cpmre}}^{A,b}\left(\mathcal E\right)&=\mathbb P^b\big[\exists t<\infty : \forall s>t,\forall x\in A,\forall y\in\mathbb Z^d, (x,0)\not\rightsquigarrow(y,s)\big],\\
	\mathbb P_{\mathrm{cpmre}}^{\{0\},b}\left(\mathcal S\right)&=\mathbb P^b\big[\exists  (t_n)_{n\in\mathbb N} : t_n\overset{n\to\infty}{\longrightarrow}\infty, t_1<t_2<\cdots, \forall n\in\mathbb N, (0,0)\rightsquigarrow(0,t_n)\big].
\end{aligned}\end{equation*}
Similar statements can be made about survival and extinction for the contact processes which only use \CPt{$\beta_{k-1},\mu^*$}-paths or \CPt{$\beta^*,\mu_0$}-paths. It should now be clear that, for example, 
\begin{equation*}
	\mathbb P_{\beta^*,\mu_0}^{A}(\mathcal S)\leq\mathbb P_{\mathrm{cpmre}}^{A,b}(\mathcal S)
\end{equation*}
when $v^\beta\prec b$ since the domination \eqref{eq:P<P-def} exists.

This graphical construction as stated is for the child-CPMRE (where the random environment at child site $x$ determines whether an arrow pointing from $y$ to $x$ can be used). By reversing all arrows (with all arrows now pointing from $x$ to $y$), we achieve a graphical construction for the parent-CPMRE, where the random environment at (now) parent site $x$ determines whether it can transmit the infection to site $y$. Note that reversing arrows for the graphical construction of a standard contact process still gives a graphical construction for the same process. However, the child- and parent-CPMREs are not identical processes. We address this no further and by default are always considering the child-CPMRE.

\subsection{Proof of Theorem \ref{thm:main_CPMRE}}
\label{sec:cpmre-proof}

We proceed in two primary stages. First, we establish dominations with standard contact processes which give our initial survival and extinction results for the CPMRE. These have restrictions on the initial environmental distributions though. The second stage is to show that these initial distributions can be arbitrary. We remind the reader that $e_j$ is the vector with unit mass on state $j$ and that $v^\beta,v^\mu$ are the eigenvectors associated, respectively, with minimal eigenvalues $\beta^*,\mu^*$. 

First, assume $B^x_0\sim\pi$ at each site $x$, as \eqref{eq:vlam<pi<vdel} shows this is sufficient to allow both desired ``asterisk-circle dominations'' to exist. Given an arbitrary realization of our graphical structure $\omega\in\Omega$, then, almost surely, along the timeline for each site $x$, $A_*^{xy}(\omega) \subset A_\circ^{xy}(\omega)$ for each $y\in\mathcal N(x)$, and $D_*^{x}(\omega) \subset D_\circ^{x}(\omega)$. 

It should be clear that we now have a coupling of our CPMRE with a contact process having infection rate $\beta_{k-1}$ and recovery rate $\mu^*$. This contact process uses every infection arrow, regardless of its markings, but only uses recovery symbols marked with an asterisk (which are all circled as well). 

The CPMRE only uses circled recovery symbols and those arrows with circles at their tips. Hence, the \CPt{$\beta_{k-1},\mu^*$} uses every arrow the CPMRE does (and possibly more), and the CPMRE uses every recovery symbol that this contact process does (and possibly more). By a standard time-rescaling argument, the contact process with infection rate $\beta_{k-1}$ and recovery rate $\mu^*$ is subcritical when $\beta_{k-1}/\mu^*\leq\lambda_c$ and dies out. Since it dominates our CPMRE, then the CPMRE dies out as well. 

We also have a coupling of our CPMRE with the \CPt{$\beta^*,\mu_0$} where the former dominates the latter. The CPMRE uses all arrows with circles at their tips, but this contact process can only use those with asterisks (which, by construction, all have circles at their tips). This contact process uses all recovery symbols, but the CPMRE only uses those which are circled. Hence, if this contact process is supercritical with $\beta^*/\mu_0>\lambda_c$, then so is the CPMRE. On the $d$-dimensional integer lattice, it is a well-known fact that survival of the contact process implies strong survival, so we have strong survival of the CPMRE. 

Now we proceed to remove the requirement that the initial environmental distribution satisfy any stochastic domination relationship with $v^\beta$ or $v^\mu$. The basic idea is intuitive. Our results are already established when each site is initially at environmental equilibrium, $B^x_0\sim\pi$ for all $x$, and convergence to $\pi$ occurs given any initial distribution. So, irrespective of the initial environmental state at each site, we can wait long enough until the environmental distribution at each site in the entire lattice is close enough to $\pi$, and our dominations with standard contact processes apply from that point in time onward. Our strict dominations are important for achieving this. We also use the fact that there is always a positive probability of avoiding infection or recovery symbols for sufficiently long periods of time. \looseness=-2

Two modified CPMRE models are used in this proof. First, fix a time $s$ so large that 
\begin{equation}\label{eq:lem31-mod-dom}
	v^\beta \precn e_0e^{sT}\prec e_{k-1}e^{sT}\precn v^\mu.
\end{equation}
This is clearly possible since $e_0\prec e_{k-1}$, $T$ is monotone, $ue^{sT}\to\pi$ as $s\to\infty$ for any (row vector) distribution $u$, and that we have the strict dominations in \eqref{eq:vlam<pi<vdel} (also see the explanation which follows \eqref{eq:vlam<pi<vdel}).

The first modified CPMRE $(\widetilde B,\widetilde X)$ has constant infection rate $\beta_{k-1}$ and disallows recoveries for $t\in[0,s]$ (again, with $s$ fixed so that \eqref{eq:lem31-mod-dom} is satisfied). The initial distribution at each site is fixed at $e_{k-1}$, the most favorable to the infection (i.e., $B^x_0=k-1$ for all $x$). For $t\in(s,\infty)$, it reverts to following the standard CPMRE parameters, $\beta,\mu, T$. Denote the distribution of this process by $\mathbb P_{s,\beta_{k-1}}^{A,e_{k-1}}$. Our graphical construction is already compatible with this modified model (as long as it is constructed with initial distribution $e_{k-1}$ at every site). This model simply ignores all recovery symbols and uses all arrows on the time interval $[0,s]$ and then operates using only circle-marked symbols thereafter. Clearly, $\mathbb P_{\mathrm{cpmre}}^{A,e_{k-1}} \prec \mathbb P_{s,\beta_{k-1}}^{A,e_{k-1}}$ since this domination can be achieved by coupling these processes on the same graphical structure. Note that $D^x_*$ is suppressed for all $x$ due to an incompatible initial distribution.

It is also trivial that $\mathbb P_{\mathrm{cpmre}}^{A,b}\prec\mathbb P_{\mathrm{cpmre}}^{A,e_{k-1}}$ since the CPMRE is stochastically increasing in the initial distribution due to monotonicity of $T$, but we cannot achieve this on our graphical construction which has a fixed initial distribution at each site. It is not hard to construct a different graph which allows different environmental processes to be coupled while respecting the monotonicity of $T$. Hence, it is clear that 
\begin{equation}\label{eq:P-compare-b-arb}
	\mathbb P_{\mathrm{cpmre}}^{A,b} (\mathcal E) \geq \mathbb P_{s,\beta_{k-1}}^{A,e_{k-1}}(\mathcal E)
\end{equation}
for any arbitrary initial distribution $b$.  We desire to show that $\mathbb P_{s,\beta_{k-1}}^{A,e_{k-1}}(\mathcal E)=1$ for this particular modified CPMRE when $\beta_{k-1}/\mu^*\leq\lambda_c$. By additivity, the set of initially infected sites is unimportant (so long as it is finite)---see (1.2) in Part I of \cite{LiggettSIS}.

Let $\epsilon>0$ be fixed and arbitrarily small (and recall that $s$ is fixed). Then, pick a large enough set $N_A=N_A(s)\subset \mathbb Z^d$ satisfying $A\subset N_A$, such that 
\begin{equation}\label{eq:P-compare-1-eps}
	\mathbb P_{s,\beta_{k-1}}^{A,e_{k-1}}(\text{all infected sites at time $s$ are in $N_A$})>1-\epsilon.
\end{equation}
This is possible since the modified CPMRE is a simple growth model up to time $s$, also referred to as the Richardson model or Eden model (a larger $s$ value requires a larger set of sites $N_A=N_A(s)$). From time $s$ onward, this modified CPMRE behaves as if it were a CPMRE being initialized with environmental distribution $e_{k-1}e^{sT}$ at each site (and some unspecified finite set of initially infected sites). 

The environmental process and the infection arrow processes are independently laid down in the graphical construction, so there is no problem with dependence in claiming $\widetilde B_s^x\sim e_{k-1}e^{sT}$ irrespective of the modified CPMRE's dynamics before time $s$.

Conditioning on the modified CPMRE's set of infected sites being a subset of $N_A$ at time $s$ (which occurs with high probability as discussed above), the probability of extinction is bounded below by that of a CPMRE whose initial set of infected sites is precisely $N_A$:
\begin{equation}\label{eq:P-compare-NA}
	\mathbb P_{s,\beta_{k-1}}^{A,e_{k-1}}(\mathcal E\mid \text{all infected sites at time $s$ are in $N_A$})\geq \mathbb P_{\mathrm{cpmre}}^{N_A,e_{k-1}e^{sT}}(\mathcal E).
\end{equation}




\addtolength{\abovedisplayskip}{-3.75px}
\addtolength{\belowdisplayskip}{-3.75px}

\vspace*{-16pt}

Since $s$ is fixed large enough to allow $e_{k-1}e^{sT}\prec v^\mu$, our initial argument where all sites were initiated with stationary environmental processes now applies, and the CPMRE with $B^x_0\sim e_{k-1}e^{sT}$, for all $x$, is dominated by a contact process 
with infection rate $\beta_{k-1}$ and recovery rate $\mu^*$, that is, $\mathbb P_{\mathrm{cpmre}}^{N_A,e_{k-1}e^{sT}}\prec \mathbb P_{\beta_{k-1},\mu^*}^{N_A}$, giving $\mathbb P_{\mathrm{cpmre}}^{N_A,e_{k-1}e^{sT}}(\mathcal E)\geq \mathbb P_{\beta_{k-1},\mu^*}^{N_A}(\mathcal E)$.

Now, combining \eqref{eq:P-compare-b-arb}, \eqref{eq:P-compare-1-eps}, and \eqref{eq:P-compare-NA}, along with assuming that $\beta_{k-1}/\mu^*\leq\lambda_c$, gives 
\begin{equation*}
	\begin{aligned}
		\mathbb P_{\mathrm{cpmre}}^{A,b} (\mathcal E) &\geq \mathbb P_{s,\beta_{k-1}}^{A,e_{k-1}}(\mathcal E)\\[-2px]
		& > (1-\epsilon)\mathbb P_{s,\beta_{k-1}}^{A,e_{k-1}}(\mathcal E\mid \text{all infected sites at time $s$ are in $N_A$})\\
		& \geq (1-\epsilon)\mathbb P_{\mathrm{cpmre}}^{N_A,e_{k-1}e^{sT}}(\mathcal E)\\
		& \geq (1-\epsilon)\mathbb P_{\beta_{k-1},\mu^*}^{N_A}(\mathcal E)=1-\epsilon.\\
	\end{aligned}
\end{equation*}
The last line follows since the contact process is subcritical with $\mathbb P_{\beta_{k-1},\mu^*}^{N_A}(\mathcal E)=1$.

Since $\epsilon>0$ is arbitrary, we conclude that the CPMRE goes extinct when the initial distribution at each site is arbitrary. The assumption that all sites have the same initial environmental distribution $b$ is unimportant since every possible initial distribution is dominated by $e_{k-1}$.

Now, we create our second modified CPMRE in the opposite way to what was just done. On $t\in[0,s]$  (with $s$ still fixed to allow \eqref{eq:lem31-mod-dom}) we prevent it from using any infection arrows and give it the maximal recovery rate $\mu_0$. Its initial distribution is $e_0$ at each site (least favorable to the infection), and we use this initial distribution in the construction of our graphical representation. Denote the distribution of this modified CPMRE with only site $x$ initially infected by $\mathbb P_{s,\mu_0}^{\{x\},e_0}$. Clearly, for arbitrary initial distribution $b$,  
\begin{equation}\label{eq:P-compare-S-b-arb}
	\mathbb  P_{s,\mu_0}^{\{x\},e_0}\prec \mathbb P_{\mathrm{cpmre}}^{\{x\},b}.
\end{equation}
Hence, all we need to show is that this modified CPMRE survives when $\beta^*/\mu_0>\lambda_c$. 

Assume that $\beta^*/\mu_0>\lambda_c$ and that both processes have only the origin initially infected. The event that this modified CPMRE experiences no recoveries during time interval $[0,s]$ is given by 
$\mathcal D_\emptyset=\{D^0\cap[0,s]=\emptyset\}$, which clearly has a positive probability, independent of the initial environmental distributions. Also, independence of the environmental process and the recovery point processes implies that we can consider the distribution of the environment at each site at time $s$ to be $e_{0}e^{sT}$ when conditioning on $\mathcal D_\emptyset$. Hence, from time $s$ onward, this modified CPMRE behaves exactly as the CPMRE with initial distribution  $e_{0}e^{sT}$ at each site (with the origin still the only infected site). This implies that 
\begin{equation}\label{eq:P-compare-S-cond-D0}
	\mathbb P_{s,\mu_{0}}^{\{0\},e_{0}}(\mathcal S\mid\mathcal D_\emptyset)=\mathbb P_{\mathrm{cpmre}}^{\{0\},e_{0}e^{sT}}(\mathcal S).
\end{equation}
Note that $\mathbb P_{s,\mu_{0}}^{\{0\},e_{0}}(\mathcal S)=\mathbb P_{s,\mu_{0}}^{\{0\},e_{0}}(\mathcal S\mid\mathcal D_\emptyset)\mathbb P(\mathcal D_\emptyset)$ since this modified CPMRE only survives if there are no recovery symbols for the origin before time $s$. Of course, $\mathbb P(\mathcal D_\emptyset)=e^{-s\mu_0}$ since $D^0$ is equivalent to a Poisson point process with intensity $\mu_0$. 

Since $s$ is fixed and chosen large enough that $v^\beta\prec e_{0}e^{sT}$, the CPMRE with $B^x_0\sim e_{0}e^{sT}$ at each site dominates a contact process with infection rate $\beta^*$ and recovery rate $\mu_0$: 
\begin{equation}\label{eq:P-compare-S-cp-dom}
	\mathbb P_{\beta^*,\mu_0}^{\{0\}}\prec \mathbb P_{\mathrm{cpmre}}^{\{0\},e_{0}e^{sT}}.
\end{equation}

\enlargethispage{1.0\baselineskip}

Putting together \eqref{eq:P-compare-S-b-arb}, \eqref{eq:P-compare-S-cond-D0}, and \eqref{eq:P-compare-S-cp-dom}, and assuming that $\beta^*/\mu_0>\lambda_c$, gives 
\begin{equation*}
	\begin{aligned}
		\mathbb P_{\mathrm{cpmre}}^{\{0\},b} (\mathcal S) &\geq \mathbb P_{s,\mu_0}^{\{0\},e_0}(\mathcal S)\\
		& = \mathbb P( \mathcal D_\emptyset )
		\mathbb P_{s,\mu_{0}}^{\{0\},e_{0}}(\mathcal S \mid \mathcal D_\emptyset)\\
		&= \mathbb P( \mathcal D_\emptyset )
		\mathbb P_{\mathrm{cpmre}}^{\{0\},e_{0}e^{sT}}(\mathcal S)\\
		&\geq \mathbb P( \mathcal D_\emptyset )
		\mathbb P_{\beta^*,\mu_0}^{\{0\}}(\mathcal S)
		>0.
	\end{aligned}
\end{equation*}
The last line follows since the contact process is supercritical with $\mathbb P_{\beta^*,\mu_0}^{\{0\}}(\mathcal S)>0$. 


\pagebreak

\addtolength{\abovedisplayskip}{3.75px}
\addtolength{\belowdisplayskip}{3.75px}

Again, as in the discussion about extinction, each site can have an arbitrary initial distribution here. It is unimportant that $B^x_0\sim b$ at all sites in the above argument since any initial distribution dominates $e_0$. 

The proof is now complete and shows that, when $\beta^*/\mu_0>\lambda_c$, the CPMRE with the arbitrary initial distribution at each site exhibits strong survival. Likewise, when $\beta_{k-1}/\mu^*\leq\lambda_c$ the CPMRE goes extinct irrespective of the initial environmental distributions. Of course, each result only applies when the relevant DCM property is present. 
\hfill $\square$ 

\medskip

Survival of the CPMRE may be possible when $\beta^*/\mu_0\leq\lambda_c$ (the opposite of our survival criteria), but stochastic domination of a subcritical \CPt{$\beta^*,\mu_0$} does not allow any conclusion about survival of the CPMRE. A similar statement can be made about the possibility of extinction when $\beta_{k-1}/\mu^*>\lambda_c$. We do not explore these parameter regimes here. 

Our use of stochastic domination is arguably a less technical approach than, say, a comparison with oriented percolation. The latter method is used frequently in the interacting particle system literature. There are still potentially interesting questions for CPMREs which might be addressed using such an approach, for example, to define critical parameter values and to study the behavior precisely at them. Also, one could consider relaxing the requirement that $T$ be monotone. Our MMPP stochastic domination can still apply when $T$ is not monotone (see Remarks \ref{rmk:FCM} and \ref{rmk:non-dcm}), but the argument we used to allow the initial distributions to be arbitrary no longer works as it relies on monotonicity of $T$. However, a comparison to percolation may be useful. Such an approach is used for a version of this CPMRE model with $k=2$ in \cite{steif-warf} (but where only recovery rates vary). We leave this open to future exploration of more general CPMREs. 

\begin{remark}\label{rmk:cmpre-perm}
	Theorem \ref{thm:main_CPMRE} can sometimes still be used when the infection and recovery rate vectors are not monotone. If a permutation applied to $\beta,T$ can create a DCM MMPP, then the survival criteria, Theorem \ref{thm:main_CPMRE}(i), applies. Similarly, if a permutation can turn the MMPP with $\mu,T$ into a DCM model, then our extinction criteria in Theorem \ref{thm:main_CPMRE}(ii), applies. In particular, see Remark \ref{rmk:non-dcm}. Also, the infection and recovery processes could be assumed to use distinct generator matrices (denoted, e.g., by $T^\beta,T^\mu$, respectively), and a trivial modification to Theorem \ref{thm:main_CPMRE} applies.
\end{remark}

\begin{remark}
	It is easy to prove that the CPMRE is monotone in each component of $\beta$ and $\mu$. If the CPMRE with $\beta,\mu,T$ is not infection-DCM, then it might be possible to decrease some components of $\beta$ to create a CPMRE with $\tilde\beta,\mu,T$ (with $\tilde\beta_j\leq\beta_j$ for all $j$) which is infection-DCM and survives according to Theorem \ref{thm:main_CPMRE}. This implies survival of the former CPMRE with $\beta$. Let $\beta^\lambda=\lambda\beta$ with parameter $\lambda$, and assume the model with $\lambda=1$ is infection-DCM and survives via Theorem \ref{thm:main_CPMRE}. It follows that there there is a critical value $\lambda_{\mathrm{cpmre}}\geq0$ such that $\lambda>\lambda_{\mathrm{cpmre}}$ gives survival and $\lambda<\lambda_{\mathrm{cpmre}}$ gives extinction (trivially, $\lambda=0$ gives extinction). However, if the conditions of Theorem \ref{thm:main_CPMRE}(i) are not met for any $\lambda$, then we can make no claims about survival nor about the existence of $\lambda_{\mathrm{cpmre}}$. We also leave open questions on the behavior at criticality and whether there is ever a weak survival regime. 
\end{remark}

\begin{remark}\label{rmk:cpmre-arb-graph}
	Theorem \ref{thm:main_CPMRE} can be restated for the CPMRE on a general graph by comparing $\beta^*/\max_j\mu_j$ and $\max_j\beta_j/\mu^*$ to the appropriate critical values for the contact process. However, if the CPMRE dominates a contact process which survives weakly, then, although this proves that the CPMRE survives, it is unclear if it is only weak survival or if it is strong survival. Similarly, if the CPMRE is dominated from above by a weakly surviving contact process, then the CPMRE at best survives weakly, but might instead go extinct. We leave this for future exploration. 
\end{remark}

\begin{remark}
	Reversing arrows in its graphical construction has no effect on the standard contact process. This shows that Theorem \ref{thm:main_CPMRE} applies equally to both the child-CPMRE and parent-CPMRE, since reversing all arrows in our graphical construction preserves the dominations with contact processes.
\end{remark}

Now, we look at a few example CPMRE models to exhibit determining when a model is infection- and recovery-DCM and deriving survival and extinction regimes for them. 

\subsection{A two-state CPMRE model}
\label{sec:example-cpmre-2}

For the following model, environmental state $0$ allows only recoveries, and state $1$ allows only infections. 
Let $\beta=(0,\lambda)$, $\mu=(\delta,0)$, and
\begin{equation*}
	T=
	\left(\begin{matrix}
		-\gamma & \gamma\\
		c\gamma & -c\gamma
	\end{matrix}\right)
\end{equation*}
with parameters $\lambda,\delta,\gamma,c>0$. The parameter $\gamma$ controls the overall speed of the random environment flipping between states, and $c$ governs any asymmetry between the environmental states. By Theorem \ref{thm:main_mmpp_cm}, this model is both infection- and recovery-DCM since $k=2$. 

Clearly, increasing $c$ makes survival more difficult since that makes the environment transition into the recovery-only state faster. Hence, decreasing $c$ makes survival easier. This reasoning leads to the intuition that, for any $\lambda,\delta$, there is a $c$ small enough which allows survival and a $c$ large enough which leads to extinction. This follows from the fact that this CPMRE behaves more and more like a contact process with births only in the former and one with deaths only in the latter. We offer no proof for these intuitive claims though. 

One could use the method in \cite{broman2007} to derive the survival and extinction regimes, but this requires matching the parametrization to the model given therein. Knowing that $\beta^*,\mu^*$ are eigenvalues enables the parametrization to be arbitrary. 

Computing the minimal eigenvalues of $D_{\beta}-T$ and $D_{\mu}-T$ (defined by the convention in Definition \ref{def:alpha*}) gives
\begin{equation*}
	\begin{aligned}
		\beta^*&= \frac12\left(\lambda + (1+c)\gamma - \sqrt{(\lambda+(1+c)\gamma)^2-4\lambda\gamma } \right),\\
		\mu^*&= \frac12\left(\delta + (1+c)\gamma - \sqrt{(\delta+(1+c)\gamma)^2-4c \delta\gamma } \right).  
	\end{aligned}
\end{equation*}
Note that both $\beta^*,\mu^*>0$ for any $c,\gamma,\lambda,\delta>0$. 

Theorem \ref{thm:main_CPMRE} gives extinction if $\max_j\beta_j/\mu^*=\lambda/\mu^*\leq\lambda_c$ and survival if $\beta^*/\max_j\mu_j=\beta^*/\delta>\lambda_c$. This survival criterion requires $\delta\lambda_c<\gamma$ since it is equivalent to (via some algebra)
\begin{equation*}
	\frac{\lambda}{\delta}>\lambda_c\left( 1+\frac{c\gamma}{\gamma-\delta\lambda_c}\right).
\end{equation*}

Note that our survival result requires the environment to evolve sufficiently fast. As $\gamma$ decreases to $\delta\lambda_c$, then $\lambda/\delta$ must be increasingly large for Theorem \ref{thm:main_CPMRE} to give survival. It is left as an open question as to whether survival is possible for $\gamma\leq\delta\lambda_c$. One could pursue this model further, but our goal was to simply show application of some of our results. 

\subsection{A three-state CPMRE model} 
\label{sec:example-cpmre-3}

Now, we consider a CPMRE model with three environmental states. Assume $\lambda,\delta>0$. State $0$ allows only recoveries (with maximal recovery rate $2\delta$), state $1$ allows infections and recoveries (with rate parameters $\lambda$ and $\delta$, respectively) and state $2$ only allows infections (with maximal infection rate $2\lambda$). 

The infection and recovery rate vectors are thus $\beta=(0,\lambda,2\lambda)$, $\mu=(2\delta,\delta,0)$, and we assume the generator matrix, for $\gamma>0$, is given by 
\begin{equation*}
	T=\left(\begin{matrix}
		-\gamma & \gamma & 0\\
		\frac\gamma2 & -\gamma & \frac\gamma2\\
		0 & \gamma & -\gamma\\
	\end{matrix}\right).
\end{equation*}
The environment can thus only increment up or down by one, and $\gamma$ can be thought of as governing the speed of environmental transitions. 

The minimal eigenvalues of $D_\beta-T$ and $D_\mu-T$ are tractable: 
\begin{equation*}
	\begin{aligned}
		\beta^* &= \lambda+\gamma-\sqrt{\lambda^2+\gamma^2},\\
		\mu^* &= \delta+\gamma-\sqrt{\delta^2+\gamma^2}. 
	\end{aligned}
\end{equation*}
Applying Theorem \ref{thm:main_CPMRE} and assuming the model is infection-DCM, we have survival if
\begin{equation}\label{eq:k3model-surv}
	\frac{\beta^*}{\max_j\mu_j}=\frac{\lambda+\gamma-\sqrt{\lambda^2+\gamma^2}}{2\delta} >\lambda_c.
\end{equation}
Assuming the model is recovery-DCM, we have extinction if
\begin{equation}\label{eq:k3model-ext}
	\frac{\max_j\beta_j}{\mu^*}=\frac{2\lambda}{\delta+\gamma-\sqrt{\delta^2+\gamma^2}} \leq \lambda_c.
\end{equation}

We now  determine when this CPMRE is infection-DCM and further analyze the survival criteria \eqref{eq:k3model-surv}. Since $k=3$, we only need to check \eqref{eq:dcm} for $\ell=0,i=1,j=2$. Hence, confirming infection-DCM only requires showing that $(\beta_2-\beta_1)v_0^\beta\leq T_{10}-T_{20}$. Applying the parameterization of this model transforms that inequality to 
\begin{equation}\label{eq:lam-vlam-gam2}
	\lambda v^\beta_0\leq \gamma/2.
\end{equation}
Computing the eigenvector $v^\beta$ is somewhat tedious but possible, and its first component is
\begin{equation}\label{eq:vlam0}
	v_0^\beta=\frac{(\lambda + \sqrt{\lambda^2 + \gamma^2})^2}{(\lambda + \gamma + \sqrt{\lambda^2 + \gamma^2})^2}.
\end{equation}
Using \eqref{eq:vlam0} and solving \eqref{eq:lam-vlam-gam2} for $\lambda$ leads to the model being infection-DCM if and only if
$0<\lambda\leq\gamma$. 
By symmetry, it is recovery-DCM if and only if $0<\delta\leq \gamma$. 

We proceed to investigate the bounds on $\lambda$ which give thresholds for survival and extinction: $\lambda\leq\lambda_\mathcal E$ gives extinction and $\lambda>\lambda_\mathcal S$ gives survival---the bounds $\lambda_\mathcal E,\lambda_\mathcal S$ are defined below. From here onward, we set the timescale by letting $\delta=1$. 

Solving \eqref{eq:k3model-surv} for $\lambda$, we define $\lambda_\mathcal S$ and see that 
survival occurs if
\begin{equation}\label{eq:lamS}
	\lambda>\lambda_\mathcal S \coloneq 2\lambda_c\left(1+\frac{\lambda_c}{\gamma-2\lambda_c}\right) 
\end{equation}
with the additional assumption that $\gamma>2\lambda_c$ (still requiring $0<\lambda\leq \gamma$ for infection-DCM as our survival result relies on this DCM property). Note that $\gamma>2\lambda_c$ gives $\lambda_\mathcal S>2\lambda_c$. 

It is interesting that we now have upper and lower bounds on $\lambda$, giving an interval of values for $\lambda$ for which both survival occurs and the model is infection-DCM: 
\begin{equation*}
	\lambda_\mathcal S=2\lambda_c\left(1+\frac{\lambda_c}{\gamma-2\lambda_c}\right) < \lambda \leq \gamma.
\end{equation*}
It is intuitively easy to understand that the probability of survival is increasing in $\lambda$ and not hard to construct a graphical representation which illustrates that fact. So, all that is necessary here is to understand when $\lambda_\mathcal S<\gamma$, in which case, we know survival occurs for all $\lambda\in(\lambda_\mathcal S,\gamma]$. By monotonicity in $\lambda$, the infection then survives for all $\lambda>\lambda_\mathcal S$. This example CPMRE dominates a supercritical contact process for all $\lambda>\lambda_\mathcal S$, but this domination is achieved directly through our graphical construction only when $\lambda\leq \gamma$. For $\lambda> \gamma$, this domination of a supercritical contact process still holds by way of dominating another CPMRE with a smaller $\lambda\in(\lambda_\mathcal S,\gamma]$. 

Solving $\lambda_\mathcal S<\gamma$ for $\gamma$ leads to a quadratic and the requirement that
\begin{equation*}\label{eq:gam-surv-dcm-bound}
	\gamma>\lambda_c\big(2+\sqrt{2}\big)
\end{equation*}
in order for $(\lambda_\mathcal S,\gamma]$ to be nonempty. Note that this supersedes our initial requirement that $\gamma>2\lambda_c$. The latter is from the survival criteria \eqref{eq:lamS}. 

Using the well-known bounds of $1/(2d-1)\leq\lambda_c\leq 2/d$ for the $d$-dimensional integer lattice $\mathbb Z^d$ (see \cite{LiggettIPS}), we derive the bound on $\gamma$: 
\begin{equation*}\label{eq:gam-surv-suff-bnd}
	\gamma>\frac2d\big(2+\sqrt{2}\big).
\end{equation*}
This bound is sufficient for any dimension to have $(\lambda_\mathcal S,\gamma]\neq\emptyset$,  providing a range of $\lambda$ values giving both survival of the infection and the model being infection-DCM. 

Solving \eqref{eq:k3model-ext} for $\lambda$ (and defining $\lambda_\mathcal E$) shows that extinction occurs if
\begin{equation}\label{eq:lamE-def}
	\lambda\leq \lambda_\mathcal E \coloneq \frac12 \lambda_c \left(1+\gamma-\sqrt{1+\gamma^2}\right)
\end{equation}
with the requirement that $\gamma\geq1$ (since $\delta=1$) for the model to be recovery-DCM. Absent recovery-DCM, we cannot make this claim about extinction. We have that $\lambda_\mathcal E<\lambda_c/2$ since $1+\gamma-\sqrt{1+\gamma^2}<1$.

Checking consistency for our survival and extinction bounds, we see that $\lambda_\mathcal E<\lambda_\mathcal S$. Note that, as $\gamma\to\infty$, $\lambda_\mathcal E\uparrow \frac12\lambda_c$ and $\lambda_\mathcal S\downarrow2\lambda_c$. For $\gamma$ extremely large, the CPMRE jumps rapidly between its infection and recovery parameters and behaves more and more like a contact process with (average) infection rate $\lambda$ and (average) recovery rate one (with $\delta=1$ assumed) since the stationary distribution of $T$ is $\pi=(1/3,1/3,1/3)$. Hence, intuition might lead one to believe that, for large $\gamma$, survival should be possible for some $\lambda$ above the critical $\lambda_c$ but below our threshold for survival $\lambda_\mathcal S$, however, we cannot say for sure if this is the case. Our survival and extinction results are only sufficient bounds on $\lambda$ for the respective behaviors. This intuition makes these bounds seem quite rough, but we leave an open question as to what happens when $\lambda_\mathcal E<\lambda\leq\lambda_\mathcal S$.

Using the bounds $1/(2d-1)\leq\lambda_c\leq 2/d$, and noting that $\lambda_\mathcal S$ \eqref{eq:lamS} is an increasing function of $\lambda_c$, it follows that this CPMRE model survives in dimension $d$ when
\begin{equation*}
	\lambda>\frac4d \left(1+\frac{2/d}{\gamma-4/d}\right), 
\end{equation*}
with $\gamma>4/d$ required to ensure the model is infection-DCM. If the model is not infection-DCM, this survival result is not guaranteed. This CPMRE goes extinct if
\begin{equation*}
	\lambda\leq\frac{1}{4d-2}\raisebox{1px}{$\scalebox{0.8}{\Big(}$}1+\gamma-\sqrt{1+\gamma^2}\hspace*{1px}\raisebox{1px}{$\scalebox{0.8}{\Big)}$},
\end{equation*}
provided that $\gamma\geq1$ to ensure the model is recovery-DCM. If recovery-DCM fails to hold, this extinction result is not guaranteed. 

We now summarize our analysis of this CPMRE model as a proposition.
\begin{proposition}
	Consider the CPMRE having parameters $\beta,\mu,T$ as above, and with $\delta=1$. We have established the following results. 
	\begin{longlist}
		\setlength{\itemindent}{28pt}
		\setlength{\labelsep}{0.4em}
		\item The model is infection-DCM when $0<\lambda\leq \gamma$. 
		\item  If $\gamma>\lambda_c(2+\sqrt2)$, then the interval $(\lambda_\mathcal S, \gamma]$ is nonempty, and, for every $\lambda$ in this interval, the CPMRE is infection-DCM and dominates a supercritical contact process. Then, by monotonicity in $\lambda$, the CPMRE survives for all $\lambda>\lambda_\mathcal S$ (with $\lambda_\mathcal S$ defined in \eqref{eq:lamS}).
		\item The model is recovery-DCM when $\gamma\geq1$, in which case it goes extinct when $\lambda\leq\lambda_\mathcal E$ (with $\lambda_\mathcal E$ defined in \eqref{eq:lamE-def}). 
	\end{longlist}
\end{proposition}

\begin{remark}
	Monotonicity in $\lambda$ implies that there is a critical value $\lambda_{\mathrm{cpmre}}$ such that this particular CPMRE survives when $\lambda>\lambda_{\mathrm{cpmre}}$ and goes extinct when $\lambda<\lambda_{\mathrm{cpmre}}$. We clearly have $\lambda_\mathcal E \leq \lambda_{\mathrm{cpmre}}\leq\lambda_\mathcal S$ (with appropriate assumptions about $\gamma$), although, we conjecture that $\lambda_\mathcal E < \lambda_{\mathrm{cpmre}}<\lambda_\mathcal S$ due to the seeming roughness of our bounds $\lambda_\mathcal E,\lambda_\mathcal S$. We do not explore the possibility of a weak survival regime. Furthermore, we do not address the critical CPMRE, but we conjecture that it dies out---this could possibly be established by adapting the work in \cite{steif-warf}, which covers a similar claim for the CPREE studied in \cite{broman2007}. 
\end{remark}

\addtolength{\abovedisplayskip}{2px}
\addtolength{\belowdisplayskip}{2px}

\section{Downward conditional monotonicity}
\label{sec:dcm}

The main goal of this section is to prove Theorem \ref{thm:main_mmpp_cm}, establishing when an MMPP is DCM. We generally assume that $T$ is irreducible, monotone, that $\alpha$ is increasing, and that $k\geq3$ unless explicitly stated otherwise.  

We remind the reader of several important facts before proceeding. Downward conditional monotonicity (DCM) is defined in Definition \ref{def:DCM} with its matrix formulation given in Lemma \ref{lem:dcm-matrix-form}. Recall that $v^*$ is both the equilibrium no arrival distribution from Definition \ref{def:v} and also the (strictly positive) eigenvector associated with the minimal eigenvalue $\alpha^*$ of $D_\alpha-T$ (see Definition \ref{def:alpha*} and Lemma \ref{lem:lamstar-v-pos}). These facts are used repeatedly throughout this section. Additionally, we frequently use the shorthand $x_t$ notation first defined by \eqref{eq:xt-def} with matrix form given by Lemma \ref{lem:new-L(Bt|H)}. Also, recall that the unit mass on state $j\in S$ is denoted $e_j$. 

First, we will show that, for any arbitrary vector $x$ satisfying $v^*\preceq x$, it holds that 
\begin{equation}\label{eq:v<xt-S4first}
	v^*\preceq x_t=\frac{xe^{t(T-D_\alpha)}}{xe^{t(T-D_\alpha)}\bo}
\end{equation}
for all $t>0$, when the conditions in Theorem \ref{thm:main_mmpp_cm} are satisfied. Applying \eqref{eq:x<y-def}, the definition of stochastic ordering between vectors, we see that \eqref{eq:v<xt-S4first} is equivalent to  
\begin{equation}\label{eq:DCM-w-et(T-Da)}
	xe^{t(T-D_\alpha)}\bo \sum_{s\leq\ell}v^*_s\geq\sum_{s\leq\ell} \left(xe^{t(T-D_\alpha)}\right)_s
\end{equation}
holding for any $\ell\in\{0,1,\ldots,k-2\}=S\setminus\{k-1\}$. 
Considering $\ell=k-1$ is not necessary since it clearly gives equality (this is a theme throughout, that $\ell=k-1$ needs no consideration at many stages of our analysis). In the results that follow, we incrementally explore when \eqref{eq:DCM-w-et(T-Da)} holds for all $\ell$ which then implies that \eqref{eq:v<xt-S4first} holds. Then, we must establish downward conditional monotonicity (DCM), that $v^*\prec x\prec y$ implies $v^*\prec x_t\prec y_t$ for all $t>0$ (see Lemma \ref{lem:dcm-matrix-form}). We achieve all of this with the following sequence of lemmas.
\begin{longlist}
	\setlength{\itemindent}{28pt}
	\setlength{\labelsep}{0.4em}
	\item Lemma \ref{lem:one-mass-shift-nec-suff}: Prove that only $x$ of the form $x=v^*+\eta(e_j-e_i)$ need consideration, where $x$ is identical to $v^*$ but with amount of mass $\eta\in[0,v^*_i]$ shifted from state $i$ to state $j$. 
	\item Lemma \ref{lem:x=v+eta(ej-ei)}: Require strict inequality in \eqref{eq:dcm}, and show that $v^*\prec x$ implies $v^*\prec x_t$ for $t$ small enough when $x$ has form $x=v^*+\eta(e_j-e_i)$, again with $\eta\geq0$.  
	\item Lemma \ref{lem:x=v+SUMeta(ej-ei)}: Still requiring strict inequality in \eqref{eq:dcm}, argue that $v^*\prec x_t$ for arbitrary $x$ stochastically larger than $v^*$ and for any $t>0$. 
	\item Lemma \ref{lem:dcm-v-x-y-all-t}: Show that $v^*\prec x\prec y$ implies $v^*\prec x_t\prec y_t$ for all $t>0$, again, still requiring strict inequality in \eqref{eq:dcm}. 
	\item Lemma \ref{lem:Td-ad-model-exists}: Allow equality in \eqref{eq:dcm} by finding a perturbed model which satisfies it with strict inequality and then using weak convergence. 
\end{longlist}
After this series of lemmas, the proof of Theorem \ref{thm:main_mmpp_cm} is presented in Section \ref{sec:proof-of_main_dcm-thm}. 

\addtolength{\abovedisplayskip}{-2px}
\addtolength{\belowdisplayskip}{-2px}

We now informally describe our incremental upward mass shift technique where we consider $x$ of the form $x=v^*+\eta(e_j-e_i)$, with $\eta\in[0, v^*_i]$. Intuitively, we envision the vector $v^*$ as a collection of $k$ bins, with each bin labeled according to the vector indices in the appropriate order. Bin $i$ holds $v^*_i$ amount of probability mass for each $i\in S$. 

We begin with assuming $v^*\prec x$, that our $k$ bins match the components of $v^*$, and we proceed to shift an amount of mass $\eta_{ij}\geq0$ from bin $i$ to bin $j$ as follows. Shift the maximal amount of mass from bin $k-2$ to bin $k-1$ with the goal of reaching (and never exceeding) $x_{k-1}$ amount of mass in bin $k-1$. Set $\eta_{k-2,k-1}$ equal to the amount of mass shifted. If the amount of mass in bin $k-1$ is not yet equal to $x_{k-1}$, then proceed to shift mass from bin $k-3$ to bin $k-1$, setting the value of $\eta_{k-3,k-1}$ accordingly. After the final mass shift, from bin $0$ to bin $k-1$, the amount in bin $k-1$ is exactly $x_{k-1}$. At any step, the amount of mass shifted might be zero but is always nonnegative since $v^*\prec x$.

Next, shift mass from bins $0,\ldots,k-3$ to bin $k-2$ until the latter holds amount $x_{k-2}$. Continue shifting mass upward among the bins until the mass in all bins precisely matches the vector $x$, always shifting the maximal amount of mass possible at each step. 

The quantity shifted at any stage could be zero since we always stop once bin $j$ holds amount $x_j$ (setting $\eta_{ij}=0$ when appropriate). In this way, we have defined $\eta_{ij}\geq0$ for all pairs $i<j$ (of course, $\eta_{ij}\leq1$ always), and we can write that
\begin{equation}\label{eq:x=v-up-mass-shift}
	x=v^*+\sum_{i<j}\eta_{ij}(e_j-e_i).
\end{equation}
It is possible to formally define the $\eta_{ij}$ which result from this procedure, but that is not important. We simply require $x$ to be represented in the form \eqref{eq:x=v-up-mass-shift}, and such a form is generally not unique. For $x=v^*+\eta(e_j-e_i)$ with $\eta>0$ and pair of states $i<j$, we say that $x$ is \textit{stochastically larger than $v^*$ by a single upward mass shift}.  

For example, consider $v^*=(0.8,0.1,0.1)$ and $x=(0,0.5,0.5)$. Then two possibilities are $x=v^*+0.8(e_1-e_0)+0.4(e_2-e_1)$ and $x=v^*+0.1(e_2-e_1)+0.3(e_2-e_0)+0.5(e_1-e_0)$ (the latter is our ``maximal shift scheme'').  
 
We were careful to avoid any negative mass in our description of the sequence of mass shifts, but this is not important either. The representation \eqref{eq:x=v-up-mass-shift} is not truly a temporal sequence of mass shifts but simply a representation of $x$ in relation to $v^*$. For example, considering the shift of amount $0.4$ from $i=1$ to $j=2$ in the example $x=v^*+0.8(e_1-e_0)+0.4(e_2-e_1)$ as occurring first generates negative mass in the middle component. What is important for our purposes is to obtain any representation of the form \eqref{eq:x=v-up-mass-shift} with all $\eta_{ij}\in[0,1]$, and it is easy to understand that such a representation always exists when $v^*\prec x$. 

The next lemma shows that we only need to consider $x$ stochastically larger than $v^*$ by a single upward mass shift. Recall (by applying \eqref{eq:LBt|0} to \eqref{eq:cif}) that  $x_tD_\alpha\bo=\lambda(t\mid \tau_1>t)$ is the CIF conditioned on the first arrival having not yet occurred by time $t$ with $B_0\sim x$ initially. 

\begin{lemma}\label{lem:one-mass-shift-nec-suff}
Let $t\geq0$ be fixed. Then, $v^*\prec x_t$ for all $x$ satisfying $v^*\prec x$ if and only if it holds for all such $x$ stochastically larger than $v^*$ by a single upward mass shift. 
Similarly, $\alpha^*\leq x_tD_\alpha\bo$ for all $x$ stochastically larger than $v^*$ if and only if it holds for all $x$ stochastically larger than $v^*$ by a single upward mass shift. 
\end{lemma}
\begin{proof}
Clearly, $v^*\prec x_t$ holding for all $x$ stochastically larger than $v^*$ (for some fixed $t$) implies that the same holds for those $x$ stochastically larger than $v^*$ by only a single upward mass shift. So we only need to argue the reverse implication. 

Assume $v^*\prec x_t$ holds (for our fixed $t$ value) for all $x$ stochastically larger than $v^*$ by a single upward mass shift. Now assume a more general $x$ satisfying $v^*\prec x$ and with form $x=v^*+\sum_{i<j}\eta_{ij}(e_j-e_i)$. Recall that all $\eta_{ij}\in[0,1]$ since we work with probability vectors. 

\addtolength{\abovedisplayskip}{2px}
\addtolength{\belowdisplayskip}{2px}

Using \eqref{eq:x<y-def}, the domination $v^*\prec x_t$ is equivalent to 
\begin{equation}\label{eq:v<xt}
	\sum_{s\leq\ell} v^*_s\geq\sum_{s\leq\ell} (x_t)_s
\end{equation}
holding for all $\ell\in S$ (also recall that $\ell=k-1$ generally needs no consideration). 
We substitute $v^*+\sum_{i<j}\eta_{ij}(e_j-e_i)$ in place of $x$ into \eqref{eq:v<xt}, and use the form of $x_t$ \eqref{eq:LBt|0}, to get 
\begin{equation*}
	\sum_{s\leq\ell} v^*_s\geq\sum_{s\leq\ell} (x_t)_s=\frac{\sum_{s\leq\ell}(xe^{t(T-D_\alpha)})_s}{xe^{t(T-D_\alpha)}\bo}=\frac{\sum_{s\leq\ell} \left((v^*+\sum_{i<j}\eta_{ij}(e_j-e_i))e^{t(T-D_\alpha)}\right)_s}{(v^*+\sum_{i<j}\eta_{ij}(e_j-e_i))e^{t(T-D_\alpha)}\bo}.
\end{equation*}
Then, use the fact that $e^{-t\alpha^*},v^*$ is an eigenpair of $e^{t(T-D_\alpha)}$ with $v^*\bo=1$ (see Lemma \ref{lem:lamstar-v-pos}) so that \eqref{eq:v<xt} is now seen as equivalent to
\begin{equation}\label{eq:v<xt-expand}
	\sum_{s\leq\ell} v^*_s\geq\frac{e^{-t\alpha^*}\sum_{s\leq\ell} v^*_s+\sum_{s\leq\ell}\left(\sum_{i<j}\eta_{ij}(e_j-e_i)e^{t(T-D_\alpha)}\right)_s}{e^{-t\alpha^*}+\sum_{i<j}\eta_{ij}(e_j-e_i)e^{t(T-D_\alpha)}\bo}.
\end{equation}
Note that the denominator on the right side of \eqref{eq:v<xt-expand} is positive so that it can be multiplied across to the left side to get
\begin{equation*}
	\begin{aligned}
		\Big(e^{-t\alpha^*}+\sum_{i<j}\eta_{ij}(e_j-e_i)&e^{t(T-D_\alpha)}\bo\Big)\sum_{s\leq\ell} v^*_s\\
		& \geq e^{-t\alpha^*}\sum_{s\leq\ell} v^*_s+\sum_{s\leq\ell}\Bigg(\sum_{i<j}\eta_{ij}(e_j-e_i)e^{t(T-D_\alpha)}\Bigg)_s.
	\end{aligned}
\end{equation*}
Distribute the summation over $s\leq\ell$ on the left side, cancel the terms $e^{-t\alpha^*}\sum_{s\leq\ell} v^*_s$ from both sides, and rearrange the order of summations over $s\leq\ell$ and $i<j$ on the right to then see that \eqref{eq:v<xt} is also equivalent to 
\begin{equation}\label{eq:sum-eta(ej-ej)sum-v}
	\sum_{i<j}\eta_{ij}(e_j-e_i) e^{t(T-D_\alpha)}\bo \sum_{s\leq\ell}v^*_s\geq\sum_{i<j}\eta_{ij} \sum_{s\leq\ell} \left((e_j-e_i)e^{t(T-D_\alpha)}\right)_s.
\end{equation}

Assume that \eqref{eq:v<xt} holds for any $x$ stochastically larger than $v^*$ by a single upward mass shift. Then, for any fixed pair $i<j$ with $x=v^*+\eta_{ij}(e_j-e_i)$, \eqref{eq:v<xt} is equivalent to 
\begin{equation}\label{eq:sum-eta(ej-ej)sum-v-single}
	\eta_{ij}(e_j-e_i) e^{t(T-D_\alpha)}\bo \sum_{s\leq\ell}v^*_s\geq\eta_{ij}\sum_{s\leq\ell} \left((e_j-e_i)e^{t(T-D_\alpha)}\right)_s
\end{equation}
which is simply \eqref{eq:sum-eta(ej-ej)sum-v} but lacking the summation over all pairs $i<j$. That \eqref{eq:sum-eta(ej-ej)sum-v-single} holds for each individual pair $i<j$ and any  $\eta_{ij}\geq0$, implies that \eqref{eq:sum-eta(ej-ej)sum-v} holds as well. 

The computation for the remaining statement of the lemma is nearly identical. 
\end{proof}

By Lemma \ref{lem:one-mass-shift-nec-suff}, we only need to consider those $x$ stochastically larger than $v^*$ by a single upward mass shift, $x=v^*+\eta(e_j-e_i)$ for each  pair of states $i<j$. 

For now, we assume two things: (i) that $T$ is strictly monotone (recall that this means the equations of \eqref{eq:Tmon} are satisfied with all being strict inequalities) and (ii) that \eqref{eq:dcm} is satisfied with strict inequalities for all $\ell<i<j$. We relax both of these assumptions and allow equality in both criteria towards the end of this section via constructing a weakly converging sequence of models (see Lemma \ref{lem:Td-ad-model-exists} and Section \ref{sec:proof-of_main_dcm-thm}).

We write $O(g(t))$ to mean an appropriately sized vector or matrix with each component being $O(g(t))$ as $t\to0$, and we use this notation extensively throughout this section. The next lemma establishes a condition which gives $v^*\prec x_t$ for small $t>0$ and $x$ stochastically larger than $v^*$ by a single upward mass shift.

\addtolength{\abovedisplayskip}{-2px}
\addtolength{\belowdisplayskip}{-2px}

\begin{lemma}\label{lem:x=v+eta(ej-ei)}
	Let $x=v^*+\eta(e_j-e_i)$ for some fixed $\eta>0$ and fixed pair $i<j$. Assume that $T$ is strictly monotone, that $\alpha$ is increasing, and that, for every $\ell<i$,
	\begin{equation}\label{eq:DCM-strict}
		(\alpha_j-\alpha_i) \sum_{s\leq\ell}v^*_s  < \sum_{s\leq\ell}(T_{is}-T_{js}).
	\end{equation}
	Then, there is an $\epsilon(i,j)>0$ such that, for any $t\in(0,\epsilon(i,j))$, \eqref{eq:DCM-w-et(T-Da)} holds for every $\ell\in S$.
\end{lemma}
\begin{proof}
	For small $t$, write $e^{t(T-D_\alpha)}=I+tT-tD_\alpha+O(t^2)$. Note that $T\bo=\mathbf 0$ (the zero column vector), since $T$ is a generator, and that this fact is used repeatedly throughout this section. Now apply this matrix exponential expansion to the components of \eqref{eq:DCM-w-et(T-Da)}, and use the fact that $v^*$ is an eigenvector with $v^*\bo=1$, to get 
	\begin{equation}\label{eq:xetTDbo}
		\begin{aligned}
			xe^{t(T-D_\alpha)}\bo &=(v^*+\eta(e_j-e_i))e^{t(T-D_\alpha)}\bo\\
			&= e^{-t\alpha^*}+\eta(e_j-e_i)(I+tT-tD_\alpha)\bo+\eta O(t^2)\\
			&=  e^{-t\alpha^*}-t\eta(\alpha_j-\alpha_i)+\eta O(t^2),
		\end{aligned}
	\end{equation}
	\begin{equation}\label{eq:S_ell_x}
		\begin{aligned}
			\sum_{s\leq\ell} \left(xe^{t(T-D_\alpha)}\right)_s &=\sum_{s\leq\ell} \left((v^*+\eta(e_j-e_i))e^{t(T-D_\alpha)}\right)_s\\
			&= e^{-t\alpha^*} \sum_{s\leq\ell} v^*_s + \eta\sum_{s\leq\ell}\big((e_j-e_i)(I+tT-tD_\alpha)\big)_s+\eta O(t^2).
		\end{aligned}
	\end{equation}
	
	We apply \eqref{eq:xetTDbo} and \eqref{eq:S_ell_x} to \eqref{eq:DCM-w-et(T-Da)}. Then, by canceling the terms $e^{-t\alpha^*}\sum_{s\leq\ell} v^*_s$ and dividing by $-t\eta$, \eqref{eq:DCM-w-et(T-Da)} is seen as equivalent to
	\begin{equation}\label{eq:DCM-w-Ot}
		(\alpha_j-\alpha_i)\sum_{s\leq\ell}v^*_s\leq\sum_{s\leq\ell}\big((e_i-e_j)\left(I/t+T-D_\alpha\right)\big)_s+O(t).
	\end{equation}
	Note that $e_j-e_i$ is reversed to be $e_i-e_j$ on the right. We proceed to establish when \eqref{eq:DCM-w-Ot} holds. 
	Also, that the $O(t)$ terms were combined and do not depend on $\eta$ is important to note. 
	
	The right side of \eqref{eq:DCM-w-Ot} has three components, which are considered under three cases: Case I: $\ell<i<j$, Case II: $i\leq\ell<j$, and Case III: $i<j\leq\ell$. Table \ref{tab:cases} examines these cases. 
	By strict monotonicity of $T$, the middle column of Table \ref{tab:cases} is positive for Cases I and III. Also, in Cases II and III, we use the fact that $i\leq\ell$ implies $\sum_{s\leq\ell}T_{is}=-\sum_{s>\ell}T_{is}$ since $T$ is a generator matrix.
	\begin{table}[H]
		\caption{
			The cases and components to consider for the right side of \eqref{eq:DCM-w-Ot}.}
		\label{tab:cases}
		\begin{tabular}{c c c c}
			\toprule
			Component:& \ $1/t\sum_{s\leq\ell}( (e_i-e_j)I )_s$ \ & \ $\sum_{s\leq\ell}( (e_i-e_j)T )_s$ \ & \ $-\sum_{s\leq\ell}( (e_i-e_j)D_\alpha )_s$ \ \\
			\midrule
			Case I: $\ell<i<j$ & $0$ &$\sum_{s\leq\ell}(T_{is}-T_{js})$ & $0$\\[18px]
			Case II: $i\leq\ell<j$ & $1/t$ & $-\sum_{s>\ell}T_{is}-\sum_{s\leq\ell}T_{js}$ & $-\alpha_i$ \\[18px]
			Case III: $i<j\leq\ell$ & $0$ & $\sum_{s>\ell}(T_{js}-T_{is})$ & $\alpha_j-\alpha_i$\\
			\bottomrule
		\end{tabular}
	\end{table}
	
	Now, under Case I with $\ell<i<j$, \eqref{eq:DCM-w-Ot} simplifies to		
	\begin{equation}\label{eq:dcm-o(t)}
		(\alpha_j-\alpha_i)\sum_{s\leq\ell}v^*_s \leq
		\sum_{s\leq\ell}(T_{is}-T_{js})+O(t)
	\end{equation}
	which holds for $t$ small enough when \eqref{eq:DCM-strict} is assumed. 
	
	\pagebreak
	
	Case II with $i\leq\ell<j$ gives \eqref{eq:DCM-w-Ot} simplifying to
	\begin{equation}\label{eq:caseII}
		(\alpha_j-\alpha_i)\sum_{s\leq\ell}v^*_s \leq
		\frac1t
		-\displaystyle\sum_{s>\ell}T_{is}-\sum_{s\leq\ell}T_{js}
		-\alpha_i
		+O(t).
	\end{equation}
	This case clearly holds when $t$ is small enough (and requires no restrictions on $\alpha$ or $T$). 
	
	Finally, under Case III with $i<j\leq\ell$, \eqref{eq:DCM-w-Ot} simplifies to 	
	\begin{equation}\label{eq:caseIII}
		(\alpha_j-\alpha_i)\sum_{s\leq\ell}v^*_s \leq
		\sum_{s>\ell}(T_{js}-T_{is})
		+\alpha_j-\alpha_i
		+O(t).
	\end{equation}
	Due to our strict monotonicity assumption on $T$, $\sum_{s>\ell}(T_{js}-T_{is})>0$ when $i<j\leq\ell$. Subtracting $\alpha_j-\alpha_i$ from both sides of \eqref{eq:caseIII} leads to 
	\begin{equation*}\label{eq:caseIII-2}
		-(\alpha_j-\alpha_i)\sum_{s>\ell}v^*_s \leq
		\sum_{s>\ell}(T_{js}-T_{is})
		+O(t)
	\end{equation*}
	which is clearly nonpositive on the left since $\alpha$ is increasing and positive on the right for $t$ small enough. Hence, this case clearly holds for $t$ small enough. 
	If $T$ were allowed to simply be (nonstrictly) monotone, then \eqref{eq:caseIII} could become $0\leq O(t)$ if $\alpha_i=\alpha_j$ (which is allowed), and this is problematic as we do not analyze the $O(t)$ term carefully. 
	
	It is now established that, for any given $\ell$ (and our fixed $i<j$), there exists an $\epsilon(i,j,\ell)>0$ such that \eqref{eq:DCM-strict} implies \eqref{eq:DCM-w-et(T-Da)} for $0<t<\epsilon(i,j,\ell)$. Let $\epsilon(i,j)=\min_{\ell}\epsilon(i,j,\ell)>0$. Now, with our fixed $i<j$, it is proved that \eqref{eq:DCM-strict} holding for all $\ell<i$ implies \eqref{eq:DCM-w-et(T-Da)} for all $\ell\in S$ (and thus also implies \eqref{eq:v<xt-S4first} for all $\ell$) when $t<\epsilon(i,j)$. Note that $\epsilon(i,j)$ does not depend on $\eta$, but it generally depends on the model parameters $\alpha,T$.  
\end{proof}

Lemma \ref{lem:x=v+eta(ej-ei)} establishes that $v^*\prec x$ implies $v^*\prec x_t$ for $t$ small enough when $x$ is stochastically larger than $v^*$ by a single upward mass shift. Note that the assumptions of increasing $\alpha$ and (strictly) monotone $T$ were only used in Case III, otherwise another requirement in addition to \eqref{eq:DCM-strict} arises, though we pursue this no further here. 

Now, we have to show that we can allow any $t>0$ and any $x$ stochastically larger than $v^*$.

\begin{lemma}\label{lem:x=v+SUMeta(ej-ei)}
	For an MMPP which satisfies \eqref{eq:DCM-strict} for all triples $\ell<i<j$ and has strictly monotone $T$ and increasing $\alpha$, 
	$v^*\preceq x$ implies that $v^*\prec x_t$ for all $t>0$. 
\end{lemma}
\begin{proof}
	Lemma \ref{lem:x=v+eta(ej-ei)} shows that, for each pair $i<j$ and $x=v^*+\eta(e_j-e_i)$, we can find an $\epsilon(i,j)>0$ such that $v^*\prec x_t$
	when $t<\epsilon(i,j)$. Let $\epsilon=\min_{i<j}\epsilon(i,j)$.
	
	Now, $x$ being stochastically larger than $v^*$ by a single upward mass shift for any fixed pair $i<j$ gives $v^*\prec x_t$ whenever $t\in(0,\epsilon)$. Lemma \ref{lem:one-mass-shift-nec-suff} then shows that $v^*\prec x_t$ for arbitrary $x$ stochastically larger than $v^*$ and any $t\in(0,\epsilon)$. 
	
	Let $t>0$ be arbitrary, and choose $t_1,\ldots,t_n$ with $t=t_1+\cdots+t_n$ such that $0<t_j<\epsilon$ for all $j$. Since $t_1<\epsilon$, we have that $v^*\prec x$ implies $v^*\prec x_{t_1}$. Then, since $v^*\prec x_{t_1}$ and $t_2<\epsilon$, we have that 
	\begin{equation*}
		v^*\preceq (x_{t_1})_{t_2}=\frac{x_{t_1}e^{t_2(T-D_\alpha)}}{x_{t_1}e^{t_2(T-D_\alpha)}\bo}
		=\frac{xe^{(t_1+t_2)(T-D_\alpha)}}{xe^{(t_1+t_2)(T-D_\alpha)}\bo}=x_{t_1+t_2{}}.
	\end{equation*}
	Note that the above computation involves canceling $xe^{t_1(T-D_\alpha)}\bo$ from the $x_{t_1}$ in the numerator and denominator and uses the fact that the matrices  $e^{t_1(T-D_\alpha)}$ and $e^{t_2(T-D_\alpha)}$ commute. Iterating this further shows that $v^*\prec x_{t_1+\cdots+t_n}=x_t$. 
\end{proof}

The assumptions of increasing $\alpha$ and (strictly) monotone $T$ were only used to avoid the need for further analysis of Case III in the proof of Lemma \ref{lem:x=v+eta(ej-ei)}. Lemma \ref{lem:x=v+SUMeta(ej-ei)} relies on these assumptions only through its dependence on Lemma \ref{lem:x=v+eta(ej-ei)}. 

\pagebreak

The next lemma shows that conditioning on no arrivals preserves monotonicity as required by the definition of DCM given in Definition \ref{def:DCM}. We aim to show that $v^*\preceq x\preceq y$ implies  $v^*\preceq x_t\preceq y_t$ for all $t>0$, which is equivalent to our DCM property by Lemma \ref{lem:dcm-matrix-form}.  

\begin{lemma}\label{lem:dcm-v-x-y-all-t}
	An MMPP which satisfies \eqref{eq:DCM-strict} for all triples $\ell<i<j$ and has strictly monotone $T$ and increasing $\alpha$ is downward conditionally monotone as per Definition \ref{def:DCM}. 
\end{lemma}
\begin{proof}
	First, we establish that $v^*\preceq x\preceq y$ implies that, for $t$ small enough, 
	\begin{equation}\label{eq:vxy-ordered}
		v^*\preceq \frac{xe^{t(T-D_\alpha)}}{xe^{t(T-D_\alpha)}\bo}\preceq \frac{ye^{t(T-D_\alpha)}}{ye^{t(T-D_\alpha)}\bo}.
	\end{equation}
	Then, we show that \eqref{eq:vxy-ordered} holds for all $t>0$ which, by Lemma \ref{lem:dcm-matrix-form}, establishes DCM. 
	
	The proof of Lemma \ref{lem:x=v+eta(ej-ei)} expresses $x$ as upward mass shifts from $v^*$. We now express $y$ in terms of upward mass shifts from $x$. First, assume $v^*\prec x$ and that $y=x+\eta(e_j-e_i)$ for some fixed pair $i<j$ and $\eta>0$. We explore when $x_t\prec y_t$ holds. This is accomplished by following through the same computations as in the proof of Lemma \ref{lem:x=v+eta(ej-ei)} but with $x$ playing the role that $v^*$ played in that proof and $y$ playing the role that $x$ did. We differ from the computation there in one respect: we do not combine the $O(t^2)$ terms which appear in equations analogous to \eqref{eq:xetTDbo} and \eqref{eq:S_ell_x}. We label the resulting $O(t)$ terms with subscripts to specifically identify them. We do not show the detailed computations here, but, with $\ell<i<j$, this results in 
	\begin{equation}\label{eq:sDCM-x}
		\big( (\alpha_j-\alpha_i) +O_1(t) \big) \sum_{s\leq\ell}\left(x_t\right)_s  \leq \sum_{s\leq\ell}(T_{is}-T_{js}) +O_2(t)
	\end{equation}
	as our analogue to \eqref{eq:dcm-o(t)} for Case I. In other words, $x_t\prec y_t$ holds if and only if \eqref{eq:sDCM-x} holds. Note that $O_1(t)$ and $O_2(t)$ in \eqref{eq:sDCM-x} do not depend on $x$ in any way (nor on $\eta$).
	
	Since $v^*\prec x_t$ (for arbitrary $t>0$ by Lemma \ref{lem:x=v+SUMeta(ej-ei)}), it follows that, for any $t>0$, 
	\begin{equation*}\label{eq:xt-vs-v}
		\sum_{s\leq\ell}\left(x_t\right)_s \leq \sum_{s\leq\ell}v^*_s.
	\end{equation*}
	Since $\alpha$ is increasing and \eqref{eq:DCM-strict} is assumed for any $\ell<i<j$, it then follows that 
	\begin{equation}\label{eq:xt-vs-v2}
		(\alpha_j-\alpha_i)\sum_{s\leq\ell}\left(x_t\right)_s \leq (\alpha_j-\alpha_i)\sum_{s\leq\ell}v^*_s< \sum_{s\leq\ell}(T_{is}-T_{js}).
	\end{equation}
	Choose $\epsilon(i,j)>0$ small enough so that each $O_m(t)$ term, for $m=1,2$, in \eqref{eq:sDCM-x} satisfies
	\begin{equation}\label{eq:O(t)bound}
		|O_m(t)|<\frac12\Bigg(\sum_{s\leq\ell}(T_{is}-T_{js}) - (\alpha_j-\alpha_i)\sum_{s\leq\ell}v^*_s\Bigg)
	\end{equation}
	for all $t<\epsilon(i,j)$ and all $\ell<i$. The right side in \eqref{eq:O(t)bound} being positive follows since \eqref{eq:DCM-strict} is assumed. Fix $t<\epsilon(i,j)$, then, apply \eqref{eq:xt-vs-v2}, \eqref{eq:O(t)bound}, and $\sum_{s\leq\ell}\left(x_t\right)_s\leq1$ to see that 
	\begin{equation*}
		\begin{aligned}
			\hspace*{-0.5px}\big(\hspace*{-0.25px}(\alpha_j-\alpha_i) +O_1(t) \hspace*{-0.25px}\big) \hspace*{-0.75px}\sum_{s\leq\ell}\left(x_t\right)_s
			&< (\alpha_j-\alpha_i)\sum_{s\leq\ell}v^*_s + \frac12\Bigg(\sum_{s\leq\ell}(T_{is}-T_{js}) - (\alpha_j-\alpha_i)\sum_{s\leq\ell}v^*_s\Bigg)\\
			&=\sum_{s\leq\ell}(T_{is}-T_{js}) - \frac12\Bigg( \sum_{s\leq\ell}(T_{is}-T_{js}) - (\alpha_j-\alpha_i)\sum_{s\leq\ell}v^*_s \Bigg)\\
			&<\sum_{s\leq\ell}(T_{is}-T_{js}) +O_2(t). 
		\end{aligned}
	\end{equation*}
	This implies \eqref{eq:sDCM-x}. Thus, we conclude that $v^*\prec x_t\prec y_t$ for any $t\in(0,\epsilon(i,j))$, any $x$ satisfying $v^*\prec x$, and any $y=x+\eta(e_j-e_i)$ (for our fixed $i<j$). Similar reasoning applies to Cases II and III, and we simply let our $\epsilon(i,j)>0$ be a value which works for all three cases. 
	
	\pagebreak
	
	As done in the proof of Lemma \ref{lem:x=v+SUMeta(ej-ei)}, we let $\epsilon=\min_{i<j} \epsilon(i,j)$. Then, given any $x,y$ satisfying $v^*\prec x\prec y$, we can write $y=x+\sum_{i<j}\eta_{ij}(e_j-e_i)$. An argument similar to what was done in the proof of Lemma \ref{lem:one-mass-shift-nec-suff} then shows that $v^*\prec x_t\prec y_t$ for any $t\in(0,\epsilon)$ since it holds for any such $y$ stochastically larger than $x$ by any single upward mass shift. We reiterate here the fact that $\epsilon$ does not in any way depend on $x$ (nor on the particular $\eta_{ij}$). 
	
	We then follow through the rest of the argument in the proof of Lemma \ref{lem:x=v+SUMeta(ej-ei)}. Arbitrary $t>0$ is expressed as $t=t_1+\cdots + t_n$ where each $t_j\in(0,\epsilon)$. Now, given $v^*\prec x\prec y$, it follows that $v^*\prec x_{t_1}\prec y_{t_1}$. Then, likewise, we have $v^*\prec x_{t_1+t_2}\prec y_{t_1+t_2}$. Iteratively, we then have $v^*\prec x_{t_1+\cdots+t_n}\prec y_{t_1+\cdots+t_n}$. Hence, \eqref{eq:vxy-ordered} holds for all $t>0$. 
	
	We conclude that $T$ being strictly monotone, $\alpha$ increasing, and \eqref{eq:DCM-strict} holding for all triples $\ell<i<j$ provides that $v^*\prec x\prec y$  implies $v^*\prec x_t\prec y_t$ for all $t>0$. This shows that the MMPP is DCM. 
\end{proof}

We now construct a model $(B^\delta,X^\delta)$ with the same arrival rate vector $\alpha$ but with generator $T^\delta$ that converges to $T$ component-wise as $\delta\to0$, and we refer to this as the \textit{perturbed model}.
\begin{lemma}\label{lem:Td-ad-model-exists}
	Consider an MMPP $(B,X)$ with irreducible, monotone $T$, increasing $\alpha$, eigenvector $v^*$, and which satisfies, for all triples $\ell<i<j$,
	\begin{equation}\label{eq:DCM-lem45}
		(\alpha_j-\alpha_i)\sum_{\ell\leq s} v^*_s\leq \sum_{\ell\leq s} (T_{is}-T_{js}).
	\end{equation}
	For any $\delta>0$, there exists an MMPP $(B^\delta,X^\delta)$ with strictly monotone $T^\delta=T+\delta E$ (with $E$ an appropriate matrix), the same $\alpha$, eigenvector $v^\delta$, and which satisfies, for all triples $\ell<i<j$,
	\begin{equation}\label{eq:strict-dcm-delta}
		(\alpha_j-\alpha_i)\sum_{\ell\leq s} v^\delta_s < \sum_{\ell\leq s} (T^\delta_{is}-T^\delta_{js}).
	\end{equation}
	Furthermore, $(B^\delta,X^\delta)$ weakly converges to $(B,X)$ as $\delta\to0$, and the eigenvectors satisfy $v^*\prec v^\delta$ for all $\delta>0$. 
\end{lemma}
\begin{proof}
	Given the model $(B,X)$ satisfying the lemma statement, we proceed to construct the perturbed model $(B^\delta,X^\delta)$. 
	
	Recall that $v^*>0$ is a strictly positive eigenvector of $D_\alpha-T$ with $v^*\bo=1$ (Lemma \ref{lem:lamstar-v-pos}). Also recall that $k\geq3$ is assumed (as noted at the start of Section \ref{sec:dcm}). Letting $E_{\cdot i}$ denote column $i$ of $E$, the first column of matrix $E$ is defined by
	\begin{equation*}
		E_{\cdot 0}=(-c,1,1/2,1/3,\ldots,1/(k-1)),
	\end{equation*}
	with constant $c>0$ chosen large enough so that $v^*E_{\cdot0}<0$. It is not hard to understand that $c>k/v^*_0$ is sufficient, since it gives
	\begin{equation}\label{eq:vEdot}
		\begin{aligned}
			v^*E_{\cdot 0}&=-cv^*_0+v^*_1+\frac12v^*_2+\cdots+\frac1{k-1}v^*_{k-1}\\
			&< - cv^*_0+v^*_1+v^*_2+\cdots+v^*_{k-1} \\
			&<-cv^*_0+k-1\\
			&< -k+k-1=-1.
		\end{aligned}
	\end{equation}
	There is some freedom in how we choose this column, but this particular choice is sufficient. 
	
	We construct the rest of $E$ so that it is a generator matrix---with strictly negative diagonal and nonnegative off the diagonal. Like the first column of $E$, the off-diagonal components of the last column are to be strictly positive. Other than the first and last columns and the diagonal, all other components of $E$ are zeros. The first column is already specified, and we choose the 
	remaining components so that 
	\begin{equation}\label{eq:Efull}
		E=\begin{pmatrix*}[c]
			-c & 0 & 0 & ~~~\cdots & ~~0 &~~c\\[3px]
			1 &~ -1-2c & 0 & ~~~\cdots & ~~0 & ~~2c\\[3px]
			\frac12 & 0 ~&-\frac12-3c & ~~~\cdots & ~~0 &~~3c\\[3px]		
			\vdots & \vdots & \vdots & ~~~\hspace*{-1px}\raisebox{-2px}{\rotatebox{12}{$\ddots$}} & ~~\vdots & ~~\vdots\\[3px]
			\frac1{k-2} & 0 & 0 & ~~~\cdots & ~~~~-\frac1{k-2}-(k-1)c &~~~(k-1)c\\[3px]
			\frac1{k-1} & 0 & 0 & ~~~\cdots & ~~0 & ~~-\frac1{k-1}\\[3px]
		\end{pmatrix*}.
	\end{equation}
	
	The first column of $E$ is strictly decreasing (ignoring the first component). The last column of $E$ is strictly increasing (ignoring the last component). These facts imply that $E$ is a strictly monotone generator matrix and that $T^\delta=T+\delta E$ is strictly monotone whenever $T$ is monotone. To see this explicitly, apply the first equation of the definition of monotonicity \eqref{eq:Tmon} to $T^\delta$ to get, for $\ell<i<j$, 
	\begin{equation}\label{eq:Td-to-T-sum}
		\sum_{s\leq\ell}(T^\delta_{is}-T^\delta_{js})=\sum_{s\leq\ell}(T_{is}-T_{js})+\delta\left(\frac1i-\frac1j\right)>0.
	\end{equation}
	Similar reasoning applies for the second equation in \eqref{eq:Tmon}. 
	
	That $(B^\delta,X^\delta)$ weakly converges to $(B,X)$ as $\delta\to0$ follows from standard Markov process theory (see \cite{billingsley} for  theoretical background material on weak convergence). 
	By this, we mean that $(B^{\delta_n},X^{\delta_n})$ weakly converges to $(B,X)$ as $n\to\infty$ for any sequence $\delta_n\to0$.
	
	We now have our perturbed model $(B^\delta,X^\delta)$. If we can establish that
	\begin{equation}\label{eq:v<xt-delta}
		v^* \preceq \frac{xe^{t(T^\delta-D_{\alpha})}}{xe^{t(T^\delta-D_{\alpha})}\bo}
	\end{equation}
	for all $t>0$, then, letting $t\to\infty$ leads to $v^*\prec v^\delta$ (by a similar convergence argument as done in the proof of Lemma \ref{lem:lamstar-v-pos}). We will first show that \eqref{eq:v<xt-delta} holds for small $t>0$, and then extend to all $t>0$, thus establishing $v^*\prec v^\delta$. Then, that the perturbed model satisfies \eqref{eq:strict-dcm-delta} will be shown to follow in a straightforward way. 
	
	For an arbitrary $x$ satisfying $v^*\prec x$, we wish to show that \eqref{eq:v<xt-delta} holds for $t$ sufficiently small. As before, we will establish it for all $x$ stochastically larger than $v^*$ by a single upward mass shift. Then, an argument largely identical to that in the proof of Lemma \ref{lem:one-mass-shift-nec-suff} shows that \eqref{eq:v<xt-delta} holds for any $x$ stochastically larger than $v^*$ (we do not reproduce the details here). That \eqref{eq:v<xt-delta} holds for small $t$ will be seen to follow from essentially the same reasoning used to prove Lemma \ref{lem:x=v+eta(ej-ei)}, and extending to all $t>0$ follows by the argument proving Lemma \ref{lem:x=v+SUMeta(ej-ei)}. 
	
	To understand how we use the arguments in the proof of Lemma \ref{lem:x=v+eta(ej-ei)}, note that \eqref{eq:v<xt-delta} is analogous to \eqref{eq:v<xt-S4first}. Applying \eqref{eq:x<y-def} to \eqref{eq:v<xt-delta} with slight rearranging gives
	\begin{equation}\label{eq:v<xtTd}
		xe^{t(T^\delta-D_{\alpha})}\bo\sum_{s\leq\ell} v^*_s
		\geq\sum_{s\leq\ell}\left(xe^{t(T^\delta-D_{\alpha})}\right)_s,
	\end{equation}
	which is analogous to \eqref{eq:DCM-w-et(T-Da)}. 
	We will expand the matrix exponential in \eqref{eq:v<xtTd}, and, as in the proof of Lemma \ref{lem:x=v+eta(ej-ei)}, consider the three cases involving $\ell,i,j$. It is not hard to see that the analogous Case I, with $\ell<i<j$, requires the most scrutiny here. Case III is similarly satisfied since $\alpha$ is increasing and $T^\delta$ is strictly monotone. Just as before, Case II is never a problem for $t$ small enough.
	
	Let $x=v^*+\eta(e_j-e_i)$ for some fixed pair $i<j$ with $\eta\in[0,v^*_i]$ (to ensure that $x$ is nonnegative). We expand the matrix exponential $e^{t(T^\delta-D_{\alpha})}$ in two different ways for use in \eqref{eq:v<xtTd}. One expansion is used when multiplying $v^*$, and the other is used when multiplying $\eta(e_j-e_i)$. The computations we are about to show are straightforward but somewhat tedious.
	
	We multiply the eigenvector $v^*$ by 
	\begin{equation*}
		e^{t(T^\delta-D_\alpha)}=e^{t(T-D_\alpha+\delta E)}=e^{t(T-D_\alpha)}+t\delta E+O(t^2),
	\end{equation*}
	which gives, after noting that $E\bo=\mathbf0$ since it is a generator matrix, 
	\begin{subequations}
		\begin{align}
			v^*e^{t(T^\delta-D_{\alpha})}&=v^*e^{-t\alpha^*}+t\delta v^*E+ O(t^2),\label{eq:v-part-xa}\\
			v^*e^{t(T^\delta-D_{\alpha})}\bo&=e^{-t\alpha^*}+ O(t^2).\label{eq:v-part-xb}
		\end{align}
	\end{subequations}
	We multiply $\eta(e_j-e_i)$ by
	\begin{equation*}
		e^{t(T^\delta-D_\alpha)}=e^{t(T-D_\alpha+\delta E)}=I+t(T-D_\alpha)+t\delta E+O(t^2),
	\end{equation*}
	which, along with the facts that $(e_j-e_i)I\bo=0$ and $E\bo=T\bo=\mathbf 0$, gives
	\begin{subequations}
		\begin{align}
			&\begin{aligned}
				\eta(e_j-e_i)e^{t(T^\delta-D_{\alpha})}&=
				\eta(e_j-e_i)(I+t(T-D_\alpha)+t\delta E)
				+ \eta O(t^2),
			\end{aligned}\label{eq:eta-part-xa}\\
			&\hspace*{-4px}\eta(e_j-e_i)e^{t(T^\delta-D_{\alpha})}\bo=-t\eta(\alpha_j-\alpha_i)+ \eta O(t^2).\label{eq:eta-part-xb}
		\end{align}
	\end{subequations}
	We apply \eqref{eq:v-part-xa}--\eqref{eq:eta-part-xb} to \eqref{eq:v<xtTd} and follow through the three cases for $\ell,i,j$ as done in the proof of Lemma \ref{lem:x=v+eta(ej-ei)}. The detailed computations are not reproduced here for all three cases, but, as mentioned above, Case I needs the most careful consideration. 
	
	To check Case I, we want to show that \eqref{eq:v<xtTd} holds for all $\ell<i$ (with our fixed $i<j$) and with $x=v^*+\eta(e_j-e_i)$. Substitute $x=v^*+\eta(e_j-e_i)$ into \eqref{eq:v<xtTd}, and distribute $e^{t(T^\delta-D_{\alpha})}$, to get 
	\begin{equation}\label{eq:v<xtTd-expand}
		\begin{aligned}
			v^*e^{t(T^\delta-D_{\alpha})}\bo\sum_{s\leq\ell} v^*_s&+\eta(e_j-e_i)e^{t(T^\delta-D_{\alpha})}\bo \sum_{s\leq\ell} v^*_s\\
			&\geq\sum_{s\leq\ell}\left(v^*e^{t(T^\delta-D_{\alpha})}\right)_s+\sum_{s\leq\ell}\left(\eta(e_j-e_i)e^{t(T^\delta-D_{\alpha})}\right)_s.
		\end{aligned}
	\end{equation}
	Now, we apply the four formulas \eqref{eq:v-part-xa}--\eqref{eq:eta-part-xb} to \eqref{eq:v<xtTd-expand}, giving
	\begin{equation}\label{eq:v<xtTd-expand-subst}
		\begin{aligned}
			&\big( e^{-t\alpha^*}+ O(t^2) \big) \sum_{s\leq\ell} v^*_s+ \big( -t\eta(\alpha_j-\alpha_i)+ \eta O(t^2) \big) \sum_{s\leq\ell} v^*_s\\
			&\hspace*{1.1250in}\geq\sum_{s\leq\ell}\left(v^*e^{-t\alpha^*}+t\delta v^*E+ O(t^2)\right)_s\\
			&\hspace*{1.675in}+\sum_{s\leq\ell}\left(\eta(e_j-e_i)\big(I+t(T-D_\alpha)+t\delta E\big)+\eta O(t^2)\right)_s.
		\end{aligned}
	\end{equation}
	Then, we distribute summations in \eqref{eq:v<xtTd-expand-subst} to arrive at 
	\begin{equation}\label{eq:v<xtTd-expand-subst-distr}
		\begin{aligned}
			&e^{-t\alpha^*}\sum_{s\leq\ell} v^*_s+ O(t^2)  -t\eta(\alpha_j-\alpha_i)\sum_{s\leq\ell} v^*_s+ \eta O(t^2) \\
			&\hspace*{1in}\geq e^{-t\alpha^*}\sum_{s\leq\ell}v^*_s+t\delta \sum_{s\leq\ell}(v^*E)_s+ O(t^2)\\
			&\hspace*{1.45in}+\eta\sum_{s\leq\ell}\big((e_j-e_i)I\big)_s
			+t\eta\sum_{s\leq\ell}\big((e_j-e_i)(T-D_{\alpha})\big)_s\\
			&\hspace*{1.45in}+t\delta \eta\sum_{s\leq\ell}\big((e_j-e_i)E\big)_s+ \eta O(t^2).
		\end{aligned}
	\end{equation}
	
	Note that, in the third line of \eqref{eq:v<xtTd-expand-subst-distr}, $\sum_{s\leq\ell}\left((e_j-e_i)I\right)_s=\sum_{s\leq\ell}\left((e_j-e_i)D_{\alpha}\right)_s=0$ since $I,D_\alpha$ are diagonal matrices and $\ell<i<j$. 
	It is also helpful to realize that all $O(t^2)$ terms in \eqref{eq:v<xtTd-expand-subst-distr} depend only on $t$, $\alpha$, and $T$. 
	Subtracting from both sides of \eqref{eq:v<xtTd-expand-subst-distr} the terms $e^{-t\alpha^*}\sum_{s\leq\ell} v^*_s$, and dividing by $t$, results in \eqref{eq:v<xtTd-expand-subst-distr} (and hence \eqref{eq:v<xtTd}) being equivalent to 	
	\begin{equation} \label{eq:final-vd}
		\begin{aligned}
			&- \eta(\alpha_j-\alpha_i)\sum_{s\leq\ell} v^*_s
			\geq \eta\sum_{s\leq\ell}(T_{js}-T_{is}) \\
			&\hspace*{1.5in}+ \delta \sum_{s\leq\ell}\big(v^*E\big)_s +\delta \eta \sum_{s\leq\ell}(E_{js}-E_{is}) +  O(t).
		\end{aligned}
	\end{equation}
	Note that all $O(t)$ and $\eta O(t)$ terms have been combined, but this is seen to not be a problem shortly as we allow any $\eta\in[0,v^*_i]$ (recall that $\eta\leq1$ is always guaranteed). 
	
	It is assumed in the lemma statement that 
	\begin{equation*}
		-(\alpha_j-\alpha_i)\sum_{s\leq\ell} v^*_s\geq \sum_{s\leq\ell}(T_{js}-T_{is}),
	\end{equation*} 
	but note the reversal of $i$ and $j$ on the right so that this is our DCM requirement \eqref{eq:dcm} but multiplied across by $-1$ to reverse the inequality.
	In order to establish \eqref{eq:final-vd} for small $t$, it is therefore sufficient to show that 
	\begin{equation}\label{eq:final-delta-ineq}
		0>\delta \sum_{s\leq\ell}(v^*E)_s +\delta \eta \sum_{s\leq\ell}(E_{js}-E_{is}) 
	\end{equation}
	for any $\delta>0$ and any $\eta\geq0$ (thus combining the $O(t)$ and $\eta O(t)$ terms is not problematic). 
	
	The specific construction of $E$ provides that all components of $v^*E$ are negative except the last. To see this, note that $(v^*E)_0<-1$ as shown in \eqref{eq:vEdot}. Using the definition of $E$ given in \eqref{eq:Efull}, that $c>0$, and that $v^*_j>0$ for all $j$, shows that, for $j=1,\ldots,k-2$, 
	\begin{equation*}
		\big(v^*E\big)_j=v^*_jE_{jj}=-v^*_j\left(\frac1j+(j+1)c\right)<0.
	\end{equation*}
	Of course, $(v^*E)_{k-1}>0$ since $E$ is a generator with $E\bo=\mathbf 0$, and, thus, $v^*E\bo=0$. 
	
	Hence, our assumption that $c>k/v^*_0$ along with \eqref{eq:vEdot} gives (for $\ell\neq k-1$)  
	\begin{equation}\label{eq:sumvE}
		\sum_{s\leq\ell}\big(v^*E\big)_s= \big(v^*E\big)_0+\sum_{j=1}^\ell\big(v^*E\big)_j\leq \big(v^*E\big)_0<-1.
	\end{equation}
	Also, since $E$ is strictly monotone (and $\ell<i<j$), it follows that $\sum_{s\leq\ell}(E_{js}-E_{is})<0$. Hence, the right side of \eqref{eq:final-delta-ineq} is negative for any fixed $\delta>0$ and any $\eta\geq0$. The analogous Case III computation uses \eqref{eq:sumvE} as well, but we do not explicitly show it here. We conclude that there exists $\epsilon(i,j,\delta)>0$ (satisfactory for all three $\ell,i,j$ cases) such that $0<t<\epsilon(i,j,\delta)$ implies \eqref{eq:final-vd} (note that $\epsilon(i,j,\delta)$ is independent of $\eta$). In turn, this implies \eqref{eq:v<xt-delta} holds when $0<t<\epsilon(i,j,\delta)$ with $x=v^*+\eta(e_j-e_i)$ for our fixed pair $i<j$ (and any $\eta\in[0,v^*_i]$). 
	
	Now, just as in the proof of Lemma \ref{lem:x=v+SUMeta(ej-ei)}, we let $\epsilon(\delta)=\min_{i<j}\epsilon(i,j,\delta)$. This implies that \eqref{eq:v<xt-delta} holds for any $t\in(0,\epsilon(\delta))$ and $x=v^*+\eta(e_j-e_i)$ for any pair $i<j$ and any $\eta$. For more general $x=v^*+\sum_{i<j}\eta_{ij}(e_j-e_i)$, we employ an argument similar to that in the proof of Lemma \ref{lem:one-mass-shift-nec-suff}. An expression similar to \eqref{eq:final-vd} results but involving summation over all $i<j$ with $\eta_{ij}$ in place of $\eta$. This shows that \eqref{eq:v<xt-delta} holds for any $t\in(0,\epsilon(\delta))$ and any $v^*\prec x$. 
	
	Again, mimicking the proof of Lemma \ref{lem:x=v+SUMeta(ej-ei)}, for $t>0$, let $t=t_1+\cdots+t_n$ such that each $t_j\in(0,\epsilon(\delta))$. It follows that \eqref{eq:v<xt-delta} holds for each $t_j$, and, therefore, it holds for any $t>0$. 
	
	In summary, we have proved that, given $v^*\preceq x$ and any fixed $\delta>0$, 
	\begin{equation*}
		v^* \preceq \frac{xe^{t(T^\delta-D_{\alpha})}}{xe^{t(T^\delta-D_{\alpha})}\bo}
	\end{equation*}
	for all $t>0$. Taking the limit as $t\to\infty$ proves that $v^*\preceq v^\delta$. The fact that $\alpha$ is increasing was only used implicitly here, for dealing with the analogous Case III for $\ell,i,j$ (see Case III in the proof of Lemma \ref{lem:x=v+eta(ej-ei)}). The next step explicitly uses the fact that $\alpha$ is increasing. 
	
	When $\ell<i<j$, it is not hard to see from the definition of $E$ \eqref{eq:Efull} and also \eqref{eq:Td-to-T-sum} that $\sum_{s\leq\ell}(T^\delta_{is}-T^\delta_{js})>\sum_{s\leq\ell}(T_{is}-T_{js})$.
	This, along with the facts that $v^*\preceq v^\delta$, that all models use the same increasing $\alpha$ (with $\alpha_i\leq\alpha_j$), and that \eqref{eq:DCM-lem45} holds, then shows that 
	\begin{equation*}\label{eq:Td-str<ell}
		(\alpha_j-\alpha_i)\sum_{s\leq\ell} v^\delta_s\leq(\alpha_j-\alpha_i)\sum_{s\leq\ell} v^*_s\leq \sum_{s\leq\ell}(T_{is}-T_{js}) < \sum_{s\leq\ell}(T^\delta_{is}-T^\delta_{js}).
	\end{equation*}
	The perturbed model thus satisfies \eqref{eq:strict-dcm-delta} for any $\delta>0$. 
\end{proof}


We are now ready to prove Theorem \ref{thm:main_mmpp_cm}. We must show that the assumptions of strict monotonicity (strict inequalities in \eqref{eq:Tmon}) and strict inequalities in \eqref{eq:dcm} can be relaxed. Then, we must prove that \eqref{eq:dcm} for all triples $\ell<i<j$ is necessary for DCM (under our assumptions of $T$ monotone, $\alpha$ increasing).


\enlargethispage{0.9\baselineskip}

\subsection{Proof of Theorem \ref{thm:main_mmpp_cm}}
\label{sec:proof-of_main_dcm-thm}

Assume the conditions of Theorem \ref{thm:main_mmpp_cm} hold: that $T$ is monotone, $\alpha$ is increasing, and that for all triples $\ell<i<j$, we have
\begin{equation*}
	(\alpha_j-\alpha_i)\sum_{\ell\leq s} v^*_s\leq \sum_{\ell\leq s} (T_{is}-T_{js}).
\end{equation*}
We have already established that $T$ being strictly monotone and strict inequalities in the above equation imply downward conditional monotonicity as defined in Definition \ref{def:DCM} (see Lemmas \ref{lem:one-mass-shift-nec-suff} through \ref{lem:dcm-v-x-y-all-t}). Now, we need to show that $T$ can simply be monotone and that equality is allowed in the above equation.

Consider the family of matrices $T^\delta=T+\delta E$ with $E$ the strictly monotone generator matrix constructed in the proof of Lemma \ref{lem:Td-ad-model-exists} (recall that $T^\delta$ is strictly monotone and irreducible). Denote by $v^\delta,\alpha^*_\delta$ the eigenpair for the \MMPPt{$\alpha,T^\delta$} process: for each $\delta>0$, $\alpha_\delta^*$ is the minimal-modulus eigenvalue of $D_\alpha-T^\delta$ and $v^\delta$ is the (strictly positive) eigenvector (via Lemma \ref{lem:lamstar-v-pos}), with $v^\delta\bo=1$. 

Here, it is important that the relevant eigenspaces under consideration are one-dimensional. By Lemma \ref{lem:lamstar-v-pos}, the eigenspaces associated with $\alpha^*$ and $\alpha_\delta^*$ are always one-dimensional (under the assumption that $T$ is irreducible). 

It is a well-known fact that eigenvalues are continuous functions of the matrix components. Since $T^\delta\to T$ component-wise as $\delta\to0$ and $\alpha^*,\alpha^*_\delta$ are all simple eigenvalues (and always minimal-modulus for their respective matrices), it follows that $\alpha^{*}_\delta\to\alpha^*$  (see Theorem 6.3.12 in \cite{matan_horn}). 

We aim to show that $v^\delta\to v^*$ as $\delta\to0$ (our normalizing constraint $v^\delta\bo=v^*\bo=1$ is important), but convergence for eigenvectors is generally not guaranteed. However, the relevant eigenspaces all being one-dimensional is sufficient to obtain the desired eigenvector convergence (with our assumed normalization). We prove this next. 

Since $v^\delta$ is in the nullspace of $\alpha_\delta^*I-D_\alpha+T^\delta$ for all $\delta>0$, we have
\begin{equation*}
\begin{aligned}
	v^\delta (\alpha^*I-D_\alpha+T) &=v^\delta (\alpha^*I-\alpha_\delta^*I+T-T^\delta) + v^\delta (\alpha_\delta^*I-D_\alpha+T^\delta)\\
	&=v^\delta (\alpha^*I-\alpha_\delta^*I+T-T^\delta).
\end{aligned}
\end{equation*}
That the eigenvalues converge $\alpha^*_\delta\to\alpha^*$ and the matrices converge $T^\delta\to T$ then implies convergence $v^\delta (\alpha^*I-D_\alpha+T)\to0$, the zero row vector.

Now, we use the real Jordan canonical form \eqref{eq:J-norm-form} from the proof of Lemma \ref{lem:lamstar-v-pos}. Recall our convention to denote column vectors in bold font. Let 
\begin{equation*}
	V^{-1}=\begin{bmatrix}
	\mathbf y & Z
	\end{bmatrix}, \quad 
	V=\begin{bmatrix}
	v^*\\
	U
	\end{bmatrix}, 
\end{equation*}
where $\mathbf y$ is the right eigenvector of $T-D_\alpha$ corresponding to $-\alpha^*$ (i.e., $(D_\alpha-T)\mathbf y=\alpha^*\mathbf y$). 

\pagebreak

\noindent The columns of $Z$ and rows of $U$ capture all other (generalized) eigenvectors and vectors spanning any subspaces corresponding to complex eigenvalues. This gives 
\begin{equation*}
	\alpha^*I-D_\alpha+T =
	\begin{bmatrix}
		\mathbf y & Z
	\end{bmatrix} 
	\left(\begin{matrix}
		0 & \quad 0\\
		\mathbf 0 & \quad \alpha^* I+J
	\end{matrix}\right)
	\begin{bmatrix}
		v^*\\
		U
	\end{bmatrix}.
\end{equation*}
Using this Jordan canonical form shows that
\begin{equation}\label{eq:z-def}
	v^\delta (\alpha^*I-D_\alpha+T)=v^\delta Z(\alpha^* I+J) U =w^\delta
\end{equation}
which defines the vector $w^\delta$. We have that $w^\delta\to0$ as $\delta\to0$, as already shown above.

Let $v^\delta$ be represented as a linear combination of $v^*$ and the rows of $U$ ($v^*$ and the rows of $U$ together form a basis for $\mathbb R^k$) by
\begin{equation}\label{eq:v-delta-form}
	v^\delta=b^\delta v^*+\sum_{j=1}^{k-1} c^\delta_j u_j,
\end{equation} 
where the $u_1,\ldots, u_{k-1}$ are the rows of $U$, and the $b^\delta,c^\delta_j$ are appropriate constants (which vary in $\delta$ of course). 
It follows that $v^\delta Z$ is the coordinate projection of $v^\delta$ onto the rows of $U$, showing that $v^\delta Z=(c^\delta_1,\ldots,c^\delta_{k-1})$. 

Note that $UZ=I$, which is easily seen by the forms of $V^{-1}$ and $V$ shown above. Multiply \eqref{eq:z-def} on the right by $Z$ to get
\begin{equation*}
	v^\delta Z (\alpha^* I+J) =w^\delta Z.
\end{equation*}
Since $-\alpha^*$ is not an eigenvalue of $J$, 
it follows that $(\alpha^* I+J)^{-1}$ exists, giving
\begin{equation*}
	v^\delta Z = w^\delta Z (\alpha^* I+J)^{-1}.
\end{equation*}
That $w^\delta\to0$ as $\delta\to0$ (shown above) implies that $v^\delta Z =(c^\delta_1,\ldots,c^\delta_{k-1})\to0$ as well. In other words, all $c^\delta_j$ coefficients in \eqref{eq:v-delta-form} decay to zero. Multiply \eqref{eq:v-delta-form} on the right by $\bo$ to see
\begin{equation*}
	1=v^\delta\bo=b^\delta v^*\bo+\sum_{j=1}^{k-1} c^\delta_j u_j\bo=b^\delta +\sum_{j=1}^{k-1} c^\delta_j u_j\bo.
\end{equation*} 
Thus, as $\delta\to0$, we have that
\begin{equation*}
	b^\delta=1-\sum_{j=1}^{k-1} c^\delta_j u_j\bo\to1.
\end{equation*}
We conclude that $v^\delta\to v^*$ as $\delta\to0$. 

Now, we have all convergences needed to proceed with proving the model $(B,X)$ is DCM. To see that the model $(B,X)$ satisfies Definition \ref{def:DCM} for DCM, consider some arbitrary $x,y$ with $v^*\preceq x\preceq y$ and construct sequences of vectors $x^\delta,y^\delta$ so that $v^\delta\preceq x^\delta\preceq y^\delta$ for all $\delta>0$ (with limits as $\delta\to0$ indicated by the notation) as follows. Given $v^*\prec x\prec y$, we can write $x=v^*+\sum_{i<j}\eta^x_{ij}(e_j-e_i)$ and $y=x+\sum_{i<j}\eta^y_{ij}(e_j-e_i)$. Now, simply let $x^\delta=v^\delta+\sum_{i<j}\eta^x_{ij}(e_j-e_i)$ and $y^\delta=x^\delta+\sum_{i<j}\eta^y_{ij}(e_j-e_i)$, which clearly gives the desired limits.

Since the model $(B^\delta,X^\delta)$ satisfies \eqref{eq:strict-dcm-delta} 
for all $\delta>0$, Lemma \ref{lem:dcm-v-x-y-all-t} then gives that 
\begin{equation*}\label{eq:v<x<y-delta}
	v^\delta\preceq \frac{x^\delta e^{t(T^\delta-D_{\alpha})}}{x^\delta e^{t(T^\delta-D_{\alpha})}\bo}\preceq \frac{y^\delta e^{t(T^\delta-D_{\alpha})}}{y^\delta e^{t(T^\delta-D_{\alpha})}\bo},
\end{equation*}
for all $\delta>0$ and any $t>0$. Taking the limit as $\delta\to0$ shows that $v^*\prec x_t\prec y_t$ (see Corollary 6.2.32 in \cite{topmatan_horn} for matrix exponential convergence). Since this is true for all $t>0$ and arbitrary $x,y$ satisfying $v^*\prec x\prec y$, it follows that $(B,X)$ is also DCM. 

The arguments up to this point have established that any MMPP with monotone $T$, increasing $\alpha$, and which satisfies \eqref{eq:dcm} for all $\ell<i<j$ is indeed DCM. We now show that these criteria are in fact necessary for DCM. If the inequality \eqref{eq:dcm} is violated for some triple $\ell<i<j$, then there exists a sufficiently small $t>0$ and $x=v^*+\eta(e_j-e_i)$ (with $\eta>0$) such that \eqref{eq:DCM-w-Ot} is violated (see Case I studied in the proof of Lemma \ref{lem:x=v+eta(ej-ei)}). It follows that \eqref{eq:DCM-w-et(T-Da)}, and hence \eqref{eq:v<xt-S4first}, is violated giving the non-DCM behavior that $v^*\prec x$ but  
\begin{equation*}
	v^*\not\preceq \frac{xe^{t(T-D_\alpha)}}{xe^{t(T-D_\alpha)}\bo}.
\end{equation*}
Hence, such an MMPP (with monotone $T$ and increasing $\alpha$) is DCM if and only if it satisfies \eqref{eq:dcm} for all triples $\ell<i<j$. Theorem \ref{thm:main_mmpp_cm} is thus proved for $k\geq3$.

The $k=2$ case is arguably covered partially in \cite{broman2007}, but this is not obvious since no concept of DCM appears therein. In particular, they require $\alpha_0\leq\alpha_1$, but we do not. 

Let $k=2$, and take arbitrary $x\prec y$ (there is no need to consider $v^*$ here, as will become clear shortly). We wish to show that $x_t\prec y_t$ for all $t>0$. Of the cases in Table \ref{tab:cases} in the proof of Lemma \ref{lem:x=v+eta(ej-ei)}, only Case II is relevant with $0=i=\ell<j=1$ (since checking $\ell=k-1=1$ is unnecessary). Follow through the computation as in the proof of Lemma \ref{lem:x=v+eta(ej-ei)} for Case II, but, similar to what was done in the proof of Lemma \ref{lem:dcm-v-x-y-all-t}, do not combine the $O(t^2)$ terms. This shows that $x_t\prec y_t$ holds when
\begin{equation}\label{eq:caseII-k=2}
	\big((\alpha_1-\alpha_0)+O(t)\big)(x_t)_0\leq \frac1t-T_{01}-T_{10}-\alpha_0+O(t).
\end{equation}
This is analogous to \eqref{eq:caseII} but with $x_t$ in place of $v^*$, $\ell=i=0,j=1$, and with the $O(t)$ terms not combined in the way discussed below \eqref{eq:DCM-w-Ot}. It is important here that the $O(t)$ terms in \eqref{eq:caseII-k=2} in no way depend on the particular $x$ and only depend on the components of $\alpha,T$ (and on $t$). Since $0\leq(x_t)_0\leq1$ (for arbitrary distribution $x$ and any $t\geq0$), it is not hard to see that we can find an $\epsilon>0$ such that \eqref{eq:caseII-k=2} holds for all $t\in(0,\epsilon)$ and any $x$. Similar to what was done in the proof of Lemma \ref{lem:dcm-v-x-y-all-t}, it is sufficient to choose $\epsilon>0$ such that, for all $t\in(0,\epsilon)$, both $O(t)$ terms in \eqref{eq:caseII-k=2} satisfy 
\begin{equation*}
	|O(t)|<\frac12\left(\frac1t-T_{01}-T_{10}-\alpha_0-|\alpha_1-\alpha_0|\right).
\end{equation*}
We emphasize that this $\epsilon$ in no way depends on our particular choice of $x$ here. That $\epsilon$ does not depend on $\eta$ follows by the same reasoning mentioned below \eqref{eq:DCM-w-Ot} and at the end of the proof of Lemma \ref{lem:x=v+eta(ej-ei)}. 

Hence, we can find an $\epsilon>0$ such that $x_t\prec y_t$ when $t<\epsilon$ for all pairs of distributions $x\prec y$. As done in the proof of Lemma \ref{lem:dcm-v-x-y-all-t}, we similarly establish that $x_t\prec y_t$ for all $t>0$ by writing $t=t_1+\cdots+t_n$ with each $t_j<\epsilon$. Of course, $(v^*)_t=v^*$ always. Thus $v^*\prec x\prec y$ implies $v^*\prec x_t\prec y_t$ for all $t>0$. By Lemma \ref{lem:dcm-matrix-form}, this establishes DCM when $k=2$. 
\hfill $\square$

\begin{remark}\label{rmk:FCM}
	The $k=2$ case in the proof of Theorem \ref{thm:main_mmpp_cm} actually establishes something stronger than downward conditional monotonicity. DCM only requires conditional monotonicity in relation to $v^*$, that $v^*\prec x\prec y$ implies $v^*\prec x_t\prec y_t$ for all $t>0$. The $k=2$ case shows that $x\prec y$ implies $x_t\prec y_t$ for all $t>0$ without any need to consider $v^*$. We call this \textit{full conditional monotonicity} (FCM). Extending the concept of FCM to $k\geq3$ flows somewhat naturally from the methods in this paper (as does extending the concept of DCM to MMPPs with nonmonotone $T$). The heavy lifting is already done here. However, for $k\geq3$, DCM does not generally imply FCM. It is also not too much additional work to extend the concepts of FCM and DCM to more general continuous-time Markov processes on finite, totally-ordered state spaces. We pursue these extensions no further here. 
\end{remark}	

\begin{remark}\label{rmk:dcm-suff-only}
	The main purpose of our development of downward conditional monotonicity was to allow derivation of the stochastic domination established in Theorem \ref{thm:main_mpp_dom}(i). It is not hard to see that DCM is merely sufficient but not necessary for this (see Remark \ref{rmk:non-dcm}). 
\end{remark}

\subsection{Proof of Lemma \ref{lem:v-sdom-pi}}
\label{sec:v-sdom-pi}

For this proof, it is important that $T$ is monotone and $\alpha$ increasing. Our goal is to show that $v^*\precn \pi$ which is strict domination as defined below \eqref{eq:x<y-def}. 

We begin by showing that, for small enough $t$,
\begin{equation}\label{eq:pi_t<pi}
	\frac{\pi e^{t(T-D_\alpha)}}{\pi e^{t(T-D_\alpha)}\bo} \precn \pi.
\end{equation}
For this purpose, it is convenient to use the matrix exponential representation 
\begin{equation}\label{eq:etDaT-rep-pi-sdom}
	e^{t(T-D_\alpha)}=e^{tT}(I-tD_\alpha)+O(t^2).
\end{equation}
Note that $e^{tT}$ is a monotone transition matrix since $T$ is a monotone generator. 
Recall that $O(t^2)$ stands in place for a scalar, vector, or matrix (with each component satisfying the stated asymptotic behavior) depending on the context (e.g., it is a matrix in \eqref{eq:etDaT-rep-pi-sdom}).  

Applying the matrix exponential representation \eqref{eq:etDaT-rep-pi-sdom} to the left side of \eqref{eq:pi_t<pi} and using the fact that $\pi$ is the stationary distribution (with $\pi\bo=1$) for both the generator $T$ ($\pi T=0$, the zero row vector) and the transition matrix $e^{tT}$ ($\pi e^{tT}=\pi$) gives
\begin{equation*}
	\frac{\pi e^{t(T-D_\alpha)}}{\pi e^{t(T-D_\alpha)}\bo}=\frac{\pi (e^{tT}(I-tD_\alpha)+O(t^2))}{\pi (e^{tT}(I-tD_\alpha)+O(t^2)) \bo}=\frac{\pi -t\pi D_\alpha+O(t^2)}{1-t\pi D_\alpha\bo+O(t^2)}.
\end{equation*}
Hence, the strict stochastic domination \eqref{eq:pi_t<pi} can be written as
\begin{equation*}
	\frac{\pi -t\pi D_\alpha+O(t^2)}{1-t\pi D_\alpha\bo+O(t^2)}\precn \pi.
\end{equation*}
Translating this strict domination into the summation comparison \eqref{eq:x<y-def} with strict inequality (for $\ell\neq k-1$) gives
\begin{equation}\label{eq:pi_t<pi-simpl-distr}
	\frac{\sum_{s\leq\ell}\pi_s -t\sum_{s\leq\ell}(\pi D_\alpha)_s+O(t^2)}{1-t\pi D_\alpha\bo+O(t^2)}> \sum_{s\leq\ell}\pi_s.
\end{equation}
Finally, multiply across by the denominator on the left of \eqref{eq:pi_t<pi-simpl-distr}, subtract the term $\sum_{s\leq\ell}\pi_s$ from both sides, divide both sides by $-t$, and combine the remaining $O(t)$ terms to arrive at
\begin{equation}\label{eq:piDa<}
	\sum_{s\leq\ell}(\pi D_\alpha)_s< 
	\pi D_\alpha\bo\sum_{s\leq\ell} \pi_s+O(t)
\end{equation}
for all $\ell\neq k-1$. We now pursue establishing \eqref{eq:piDa<} for $t$ small enough as this will show that \eqref{eq:pi_t<pi} holds for small $t$. 

Let $B_0\sim\pi$. Recall that $\alpha$ is increasing and not a constant vector and that $\pi$ is strictly positive by irreducibility of $T$. We aim to prove the intuitive fact that, when $\ell\neq k-1$,  
\begin{equation}\label{eq:Epiell<Epi}
	\mathbb E[\alpha_{B_0}\mid B_0\leq\ell] < \mathbb E[\alpha_{B_0}].
\end{equation}
To see why \eqref{eq:Epiell<Epi} holds, we use a simple fact about the \textit{mediant} of fractions $A/B$ and $C/D$ with $B,D>0$: $A/B<(A+C)/(B+D)<C/D$ if and only if $A/B<C/D$. Let  
\begin{equation*}
	A=\sum_{j=0}^\ell\alpha_j\pi_j, \qquad
	B=\sum_{j=0}^\ell\pi_j, \qquad
	C=\!\!\!\sum_{j=\ell+1}^{k-1}\alpha_j\pi_j, \qquad
	D=\!\!\!\sum_{j=\ell+1}^{k-1}\pi_j.
\end{equation*}
Note that $B,C,D>0$ since $\pi>0,\alpha_{k-1}>0$ but that $A=0$ is possible.


We study $CB-DA$ and arrive at
\begin{equation}\label{eq:CB-DA}
\begin{aligned}
	CB-DA&=\sum_{j=\ell+1}^{k-1}\alpha_j\pi_j \sum_{i=0}^\ell\pi_i - \sum_{j=\ell+1}^{k-1}\pi_j \sum_{i=0}^\ell\alpha_i\pi_i \\
	&=(\alpha_{\ell+1}\pi_{\ell+1}+\cdots+\alpha_{k-1}\pi_{k-1})(\pi_0+\cdots+\pi_\ell)\\
	&\hspace*{1in}-(\pi_{\ell+1}+\cdots+\pi_{k-1})(\alpha_0\pi_0+\cdots+\alpha_\ell\pi_\ell)\\
	&=\sum_{j=\ell+1}^{k-1}\sum_{i=0}^\ell(\alpha_j-\alpha_i)\pi_j\pi_i,
\end{aligned}
\end{equation}
which can be proven via induction on $k$ and $\ell$ or other elementary methods, so we leave this as an exercise for the reader. 
That $CB-DA>0$ follows from the fact that $\alpha_j-\alpha_i\geq0$ for all $j\in\{\ell+1,\ldots,k-1\}$ and $i\in\{0,\ldots,\ell\}$, that $\alpha_{k-1}-\alpha_0>0$, and that $\pi_j\pi_i>0$ for all $i,j$. This proves that $A/B<C/D$. Thus, by the above-stated fact about mediants, we have that $A/B<(A+C)/(B+D)$, 
which is simply \eqref{eq:Epiell<Epi} written in terms of $A,B,C,D$: 
\begin{equation}\label{eq:E-alpha-ell-expand}
\begin{aligned}
	\mathbb E[\alpha_{B_0}\mid B_0\leq\ell]&=\frac{\sum_{j=0}^\ell\alpha_j\pi_j}{\sum_{j=0}^\ell\pi_j} \\
	&	=\frac AB < \frac{A+C}{B+D}=\frac{\sum_{j=0}^\ell\alpha_j\pi_j+\sum_{j=\ell+1}^{k-1}\alpha_j\pi_j}{\sum_{j=0}^\ell\pi_j+\sum_{j=\ell+1}^{k-1}\pi_j}=\mathbb E[\alpha_{B_0}].
\end{aligned}
\end{equation}

Recalling that $B+D=\pi\bo=1$, rearranging the inequality in \eqref{eq:E-alpha-ell-expand} gives $A<(A+C)B$. Here, it is important that $(A+C)B>0$. It follows that $A<(A+C)B+O(t)$ for $t$ small enough, and this is exactly
\eqref{eq:piDa<} translated into our notation with $A,B,C$ as defined. This establishes \eqref{eq:pi_t<pi} for $t$ small enough, since it is equivalent to \eqref{eq:piDa<}. Of course, $\pi_t\precn\pi$ implies $\pi_t\prec\pi$. The first step is now complete. 

Now, we proceed to show that $\pi_t\prec\pi$ is in fact true for all $t>0$. To achieve this, we essentially use a discretization of our model without explicitly defining it. 
Since $T$ is a monotone generator, then $e^{T/m}=I+T/m+O(m^{-2})$ is a monotone transition matrix for all $m>0$ (see the paragraph containing \eqref{eq:Tmon}). For transition matrices, each row is a probability distribution, and monotonicity of a transition matrix means that row $i$ is dominated by row $j$ for any pair $i<j$ (see \cite{keilson1977monMC} for more details). 
In fact, we can pick $m$ sufficiently large that $I+T/m$ is a monotone transition matrix. To see this, let $M=I+T/m$, and let $\ell\in \{0,\ldots,k-1\}$. We need to show that, for any $i<j$, $\sum_{s\leq\ell}M_{is}\geq\sum_{s\leq\ell}M_{js}$. This is precisely \eqref{eq:x<y-def} applied to rows $i$ and $j$ of $M$: letting $M_{i\cdot}$ denote the $i$th row of $M$, we need to show that $M_{i\cdot}\prec M_{j\cdot}$ for $i<j$. Let $\gamma_i=\sum_{s\neq i} T_{is}$ be the total exit rate from state $i$. Note that, off the diagonal, $M_{ij}=\frac1mT_{ij}\geq0$ for $i\neq j$, and, on the diagonal, $M_{ii}=1-\frac1m \gamma_i$ (which is clearly nonnegative for $m$ large enough). When $\ell<i<j$, monotonicity of the generator $T$ \eqref{eq:Tmon} gives
\begin{equation*}
	\sum_{s\leq\ell}M_{is}=\frac1m\sum_{s\leq\ell}T_{is}\geq\frac1m\sum_{s\leq\ell}T_{js}=\sum_{s\leq\ell}M_{js}.
\end{equation*} 
For $i<j\leq\ell$, the computation is similar, and $\sum_{s\leq\ell}M_{is}\geq\sum_{s\leq\ell}M_{js}$ follows due to the sum over $s>\ell$ in \eqref{eq:Tmon}. With $i\leq\ell<j$, we have $\sum_{s\leq\ell} M_{is}\geq1-\frac1m\gamma_i$ since it contains the diagonal, and $\sum_{s\leq\ell}M_{js}\leq\frac1m\gamma_j$, which does not capture the diagonal. Hence, we have 
\begin{equation*}
	\sum_{s\leq\ell}M_{is}-\sum_{s\leq\ell}M_{js}
	\geq1-\frac1m(\gamma_i+\gamma_j),
\end{equation*}
which can clearly be made nonnegative for large enough $m$. It is apparent that $m\geq2\max_i\gamma_i$ is sufficient. Also, we require $m$ to be large enough so that $1-\alpha/m$ is nonnegative, for reasons which will become clear shortly.

A monotone transition matrix $M$ has the following property: for any distributions $x\prec y$, it follows that $xM\prec yM$ (see, e.g., \cite{keilson1977monMC}), and this is a property of $I+T/m$ that we use below. Note that, since $\pi T=0$, $\pi$ is also the stationary distribution of transition matrix $I+T/m$.

Let $B_0\sim x$ with strictly positive $x$. For any $\ell\in S$, since $\alpha$ is increasing, it follows that
\begin{equation}\label{eq:Ea<ell}
	\frac{\sum_{s\leq\ell}x_s\alpha_s}{\sum_{s\leq\ell}x_s}=\mathbb E[\alpha_{B_0}\mid B_0\leq\ell] \leq \mathbb E[\alpha_{B_0}]=\sum_{s}x_s\alpha_s
\end{equation} 
%
%
%
by similar reasoning establishing \eqref{eq:Epiell<Epi}. To \eqref{eq:Ea<ell}, multiply the
denominator on the left side 

\pagebreak

\addtolength{\abovedisplayskip}{-3px}
\addtolength{\belowdisplayskip}{-1px}

\noindent across to the right side. Next, to both sides of the result, multiply by $-\frac1m$, and then add $\sum_{s\leq\ell}x_s$ to both. This gives
\vspace*{-3px} 
\begin{equation}\label{eq:Ea<ell-exp}
	\sum_{s\leq\ell}x_s-\frac1m\sum_{s\leq\ell}x_s\alpha_s \geq \sum_{s\leq\ell}x_s-\frac1m\left(\sum_{s}x_s\alpha_s\right) \Bigg(\sum_{s\leq\ell}x_s\Bigg).
\end{equation}

\noindent Noting that $\sum_{s}x_s=x\bo=1$ and that $\sum_{s}x_s\alpha_s=xD_\alpha\bo=1$, \eqref{eq:Ea<ell-exp} can then be rewritten as 
\vspace*{4px} 
\begin{equation}\label{eq:Ea<ell-exp2}
	\sum_{s\leq\ell}\left(xI-\frac1m xD_\alpha\right)_s \geq \left(x\bo-\frac1mxD_\alpha\bo\right) \sum_{s\leq\ell}x_s.
\end{equation}
Finally, we divide \eqref{eq:Ea<ell-exp2} by $x\bo-\frac1mxD_\alpha\bo=x(I-\frac1mD_\alpha)\bo$ 
(which is positive since $x>0$ and $m$ is chosen large enough to ensure $1-\alpha/m\geq0$) 
to arrive at
\vspace*{2px} 
\begin{equation*}
	\frac{\sum_{s\leq\ell}\left(xI-\frac1m xD_\alpha\right)_s}{x(I-\frac1mD_\alpha)\bo} \geq \sum_{s\leq\ell}x_s, 
\end{equation*}
which can be seen as \eqref{eq:x<y-def} applied to 
\begin{equation}\label{eq:xI-Da<x}
	\frac{x(I-D_\alpha/m)}{x(I-D_\alpha/m)\bo} \prec x.
\end{equation}
Note that, since $I-D_\alpha/m$ is a nonnegative diagonal matrix with a decreasing diagonal, \eqref{eq:xI-Da<x} is an example of the modification to Lemma \ref{lem:rescale-Da} mentioned after its proof. Thus, \eqref{eq:xI-Da<x} is established. 

We now have the ingredients we need for the next step of the proof: that $I+T/m$ is a monotone transition matrix, that multiplication by $I-D_\alpha/m$ (with normalization) creates a stochastically smaller distribution \eqref{eq:xI-Da<x}, and that \eqref{eq:pi_t<pi} holds for small $t$. 

Let $E_m=(I+T/m)(I-D_\alpha/m)$. Now, we use $\pi$ in \eqref{eq:xI-Da<x} in place of $x$ and use the fact that $\pi$ is the stationary distribution of $I+T/m$ to get 
\vspace*{3px} 
\begin{equation*}\label{eq:Empi}
	\frac{\pi E_m}{\pi E_m\bo} = \frac{\pi (I+T/m)(I-D_\alpha/m)}{\pi (I+T/m)(I-D_\alpha/m)\bo}= \frac{\pi (I-D_\alpha/m)}{\pi (I-D_\alpha/m)\bo} \prec \pi.
\end{equation*}
\vspace*{-2px}

\noindent Multiplying both sides of ${\pi E_m} / {\pi E_m\bo} \prec \pi $ by monotone transition matrix $I+T/m$ preserves the stochastic domination. That, along with $\pi(I+T/m)=\pi$, gives 
\vspace*{2px} 
\begin{equation}\label{eq:Empi2}
	\frac{\pi E_m}{\pi E_m\bo} (I+T/m) \prec \pi(I+T/m)=\pi.
\end{equation}
\vspace*{-3px} 

\noindent Next, multiplying the left side of \eqref{eq:Empi2} by $(I-D_\alpha/m)$ and normalizing creates a stochastically smaller distribution (an application of \eqref{eq:xI-Da<x}):
\vspace*{3px} 
\begin{equation*}
	\frac{\frac{\pi E_m}{\pi E_m\bo} (I+T/m)(I-D_\alpha/m)}{\frac{\pi E_m}{\pi E_m\bo} (I+T/m)(I-D_\alpha/m)\bo} \prec \pi,
\end{equation*}
which can be seen as equivalent to $\frac{\pi E_m^2}{\pi E_m^2\bo} \prec \pi$. Induction then shows that, for any $n$ (and any sufficiently large $m$),
\vspace*{-2px} 
\begin{equation}\label{eq:piEm}
	\frac{\pi E_m^n}{\pi E_m^n\bo} \prec \pi.
\end{equation}
\vspace*{0.5px} 

Fix some arbitrary $t>0$, and let $n=\lfloor mt \rfloor$. 
As $m\to\infty$, it is not hard to understand that
\begin{equation*}
	E_m^{\lfloor mt \rfloor} = \left(I+\frac1m(T-D_\alpha)+O\left(1/m^2\right)\right)^{\lfloor mt \rfloor} \longrightarrow e^{t(T-D_\alpha)}.
\end{equation*}
Now, let $n=\lfloor mt \rfloor$ in \eqref{eq:piEm}. It follows that 
\begin{equation*}
	\frac{\pi E_m^{\lfloor mt \rfloor}}{\pi E_m^{\lfloor mt \rfloor}\bo}  \overset{m\to\infty}\longrightarrow \frac{\pi e^{t(T-D_\alpha)}}{\pi e^{t(T-D_\alpha)}\bo} \prec \pi.
\end{equation*}
Since $t>0$ is arbitrary here, it is thus established that $\pi_t\prec \pi$ for all $t>0$. Taking the limit as $t\to\infty$ leads to $v^*\prec\pi$.

That the model is DCM provides that $v^*\prec\pi$ implies $v^*\prec\pi_t$ by Lemma \ref{lem:dcm-matrix-form}. Now, pick a time $t$ small enough so that \eqref{eq:pi_t<pi} holds. It follows that 
\begin{equation*}
	v^* \prec \pi_t=\frac{\pi e^{t(T-D_\alpha)}}{\pi e^{t(T-D_\alpha)}\bo} \precn \pi.
\end{equation*}
Hence, it is established that $v^*\precn\pi$. \hfill $\square$

\vspace*{-2px}

\section*{Acknowledgements}

The author thanks the referee for detailed comments and suggestions which greatly improved the clarity and exposition of the manuscript. 

\vspace*{-2px}

\bibliographystyle{abbrv}

\begin{thebibliography}{10}
	
	\bibitem{berman1994}
	Berman, A. and Plemmons, R. (1994).
	{\em Nonnegative Matrices in the Mathematical Sciences}.
	Classics in Applied Mathematics \textbf{9}.
	SIAM, Philadelphia.
	\href{https://mathscinet.ams.org/mathscinet-getitem?mr=MR1298430}{MR1298430} 
	\href{https://doi.org/10.1137/1.9781611971262}{https://doi.org/10.1137/1.9781611971262} 
	
	
	\bibitem{BezuidenhoutGrimmett1990}
	Bezuidenhout, C. and Grimmett, G. (1990).
	The critical contact process dies out.
	{\em Ann. Probab.} 
	\textbf{18}(4):1462--1482. 
	\href{https://mathscinet.ams.org/mathscinet-getitem?mr=MR1071804}{MR1071804} 
	\href{https://doi.org/10.1214/aop/1176990627}{https://doi.org/10.1214/aop/1176990627} 
	
	
	\bibitem{billingsley}
	Billingsley, P. (1999).
	{\em Convergence of Probability Measures}, 
	2nd ed. 
	Wiley Series in Probability and Statistics: Probability and Statistics. Wiley, New York.
	\href{https://mathscinet.ams.org/mathscinet-getitem?mr=MR1700749}{MR1700749} 
	\href{https://doi.org/10.1002/9780470316962}{https://doi.org/10.1002/9780470316962} 
	
	
	\bibitem{bramson1991}
	Bramson, M., Durrett, R., and Schonmann, R.~H. (1991).
	{The contact process in a random environment}.
	{\em Ann. Probab.} 
	\textbf{19}(3):960--983.
	\href{https://mathscinet.ams.org/mathscinet-getitem?mr=MR1112403}{MR1112403} 
	\href{https://doi.org/10.1214/aop/1176990331}{https://doi.org/10.1214/aop/1176990331} 
	
	
	\bibitem{broman2007}
	Broman, E.~I. (2007). 
	{Stochastic domination for a hidden Markov chain with applications to the contact process in a randomly evolving environment}.
	{\em Ann. Probab.} 
	\textbf{35}(6):2263--2293.
	\href{https://mathscinet.ams.org/mathscinet-getitem?mr=MR2353388}{MR2353388} 
	\href{https://doi.org/10.1214/0091179606000001187}{https://doi.} 
	\href{https://doi.org/10.1214/0091179606000001187}{org/10.1214/0091179606000001187} 
	
	
	\bibitem{daley_pp}
	Daley, D.~J. and Vere-Jones, D. (2003).
	{\em An Introduction to the Theory of Point Processes. {V}ol. {I}. Elementary Theory and Methods},	2nd ed. 
	Probability and Its Applications (New York). Springer, New York. 
	\href{https://mathscinet.ams.org/mathscinet-getitem?mr=MR1950431}{MR1950431} 
	\href{https://doi.org/10.1007/b97277}{https://doi.org/10.1007/b97277} 
	
	
	\bibitem{daley_pp2}
	Daley, D.~J. and Vere-Jones, D. (2008).
	{\em An Introduction to the Theory of Point Processes. {V}ol. {II}. General Theory and Structure}, 2nd ed. 
	Probability and Its Applications (New York). Springer, New York. 
	\href{https://mathscinet.ams.org/mathscinet-getitem?mr=MR2371524}{MR2371524} 	
	\href{https://doi.org/10.1007/978-0-387-49835-5}{https://doi.org/10.1007/978-0-387-49835-5} 
	
	
	\bibitem{Harris1974}
	Harris, T.~E. (1974). 
	Contact Interactions on a Lattice. 
	{\em Ann. Probab.}, 2(6):969--988. 
	\href{https://mathscinet.ams.org/mathscinet-getitem?mr=MR0356292}{MR0356292} 
	\href{https://doi.org/10.1214/aop/1176996493}{https://doi.org/10.1214/aop/1176996493} 
	
	
	\bibitem{topmatan_horn}
	Horn, R.~A. and Johnson, C.~R. (1994). 
	{\em Topics in Matrix Analysis}. 
	Cambridge Univ. Press, Cambridge. 
	\href{https://mathscinet.ams.org/mathscinet-getitem?mr=MR1288752}{MR1288752} 	
	\href{https://doi.org/10.1017/CBO9780511840371}{https://doi.org/10.1017/CBO9780511840371} 
	
	
	\bibitem{matan_horn}
	Horn, R.~A. and Johnson, C.~R. (2012). 
	{\em Matrix Analysis}, 2nd ed. 
	Cambridge Univ. Press, Cambridge. 
	\href{https://mathscinet.ams.org/mathscinet-getitem?mr=MR2978290}{MR2978290} 
	\href{https://doi.org/10.1017/CBO9781139020411}{https://doi.org/10.1017/CBO9781139020411} 
	
	
	\bibitem{keilson1977monMC}
	Keilson, J. and Kester, A. (1977)
	Monotone matrices and monotone Markov processes.
	{\em Stochastic Process. Appl.} \textbf{5}(3):231--241.
	\href{https://mathscinet.ams.org/mathscinet-getitem?mr=MR0458596}{MR0458596} 
	\href{https://doi.org/10.1016/0304-4149(77)90033-3}{https://doi.org/10.1016/0304-4149(77)90033-3} 
	
	
	\bibitem{klenke}
	Klenke, A. (2014).
	{\em Probability Theory—a Comprehensive Course}, 2nd ed.  
	\textit{Universitext.} 
	Springer, London. 
	\href{https://mathscinet.ams.org/mathscinet-getitem?mr=MR3112259}{MR3112259} 
	\href{https://doi.org/10.1007/978-1-4471-5361-0}{https://doi.org/10.1007/978-1-4471-5361-0} 
	
	
	\bibitem{LiggettIPS}
	Liggett, T.~M. (1985). 
	{\em Interacting Particle Systems}. Grundlehren der Mathematischen Wissenschaften [Fundamental Principles of Mathematical Sciences] \textbf{276}. 
	Springer-Verlag, New York. 
	\href{https://mathscinet.ams.org/mathscinet-getitem?mr=MR0776231}{MR0776231} 
	\href{https://doi.org/10.1007/978-1-4613-8542-4}{https://doi.org/10.1007/978-1-4613-8542-4} 
	
	
	\bibitem{LiggettSIS}
	Liggett, T.~M. (2010). 
	{\em Stochastic Interacting Systems: Contact, Voter and Exclusion Processes}.
	Grundlehren der Mathematischen Wissenschaften [Fundamental Principles of Mathematical Sciences] \textbf{324}. 
	Springer-Verlag, Berlin.
	\href{https://mathscinet.ams.org/mathscinet-getitem?mr=MR1717346}{MR1717346} 
	\href{https://doi.org/10.1007/978-3-662-03990-8}{https://doi.org/10.1007/978-3-662-03990-8} 
	
	
	\bibitem{neuts1979}
	Neuts, M. (1979).
	A versatile Markovian point process. 
	{\em J. Appl. Probab.} \textbf{16}:764--779.
	\href{https://mathscinet.ams.org/mathscinet-getitem?mr=MR0549556}{MR0549556} 
	\href{https://doi.org/10.2307/3213143}{https://doi.org/10.2307/3213143} 
	
	
	\bibitem{neutsMatGeo}
	Neuts, M. (1981). 
	{\em Matrix-Geometric Solutions in Stochastic Models: an Algorithmic Approach}. 
	Johns Hopkins University Press. 
	\href{https://mathscinet.ams.org/mathscinet-getitem?mr=MR0618123}{MR0618123} 
	
	
	\bibitem{perkoDEtext}
	Perko, L. (2001). 
	{\em Differential Equations and Dynamical Systems}, 
	3rd ed. 
	Texts in Applied Mathematics \textbf{7}. 
	Springer-Verlag, New York. 
	\href{https://mathscinet.ams.org/mathscinet-getitem?mr=MR1801796}{MR1801796}  
	\href{https://doi.org/10.1007/978-1-4613-0003-8}{https://doi.org/10.1007/978-1-4613-0003-8} 
	
	
	\bibitem{Rem2008}
	Remenik, D. (2008). 
	{The contact process in a dynamic random environment}. 
	{\em Ann. Appl. Probab.} \textbf{18}(6):2392--2420. 
	\href{https://mathscinet.ams.org/mathscinet-getitem?mr=MR2474541}{MR2474541}
	
	
	\bibitem{rolski1991}
	Rolski, T. and Szekli, R. (1991). 
	{Stochastic ordering and thinning of point processes}. 
	{\em Stoch. Proc. Appl.} \textbf{37}(2):299--312. 
	\href{https://mathscinet.ams.org/mathscinet-getitem?mr=MR1102876}{MR1102876} 
	\href{https://doi.org/10.1016/0304-4149(91)90049-I}{https://doi.org/10.1016/0304-4149(91)90049-I} 
	
	
	\bibitem{steif-warf}
	Steif, J.~E. and Warfheimer, M. (2008). 
	The critical contact process in a randomly evolving environment dies out. 
	{\em ALEA Lat. Am. J. Probab. Math. Stat.} 
	\textbf{4}:337--357. 
	\href{https://mathscinet.ams.org/mathscinet-getitem?mr=MR2461788}{MR2461788} 
	
	
	\bibitem{stover2022}
	Stover, J.~P. (2022). 
	Bounds via spectral radius-preserving row sum expansions. 
	{\em Electron. J. Linear Algebra}, \textbf{38}:367--376. 
	\href{https://mathscinet.ams.org/mathscinet-getitem?mr=MR4494136}{MR4494136} 
	\href{https://doi.org/10.13001/ela.2022.6981}{https://doi.org/10.13001/ela.2022.6981}
	 	
	
\end{thebibliography}


\end{document}